\documentclass{article}

\usepackage{amsmath,amsfonts,amssymb,bbm,color,fullpage,amsthm}
\usepackage[small]{titlesec}
\numberwithin{equation}{section}

\newtheorem{Thm}{Theorem}[section]
\newtheorem{Def}[Thm]{Definition}
\newtheorem{Prop}[Thm]{Proposition}
\newtheorem{Cor}[Thm]{Corollary}
\newtheorem{Lem}[Thm]{Lemma}

\newtheorem{Hyp}[Thm]{Hypothesis}
\newtheorem{Rem}[Thm]{Remark}

\newcommand{\E}{\mathbb{E}}
\newcommand{\p}{\mathbb{P}}
\newcommand{\R}{\mathbb{R}}
\newcommand{\F}{\mathcal{F}}

\newcommand{\tensor}{\otimes}
\newcommand{\Div}{\operatorname{div}}
\newcommand{\leqs}{\lesssim}

\newcommand{\LL}[2]{L_{t}^{#1}(L_{x}^{#2})}
\newcommand{\LLs}[1]{L_{t,x}^{#1}}

\newcommand{\LLw}[2]{[L_t^{#1}(L_x^{#2})]_{w}}

\newcommand{\LW}[3]{L_t^{#1}(W_x^{#2,#3})}
\newcommand{\LHZ}[2]{L_t^{#1}(H_{0,x}^{#2})}
\newcommand{\LWw}[3]{[L_t^{#1}(W_x^{#2,#3})]_{w}}
\newcommand{\WL}[3]{W_t^{#1,#2}(L_x^{#3})}

\newcommand{\CL}[1]{C_{t}(L_x^{#1})}
\newcommand{\CLw}[1]{C_{t}([L_x^{#1}]_{w})}

\newcommand{\htd}{h_{\text{det}}^{\tau}}
\newcommand{\hts}{h_{\text{st}}^{\tau}}

\newcommand{\Fh}{\hat{\mathcal{F}}}
\newcommand{\ph}{\hat{\p}}

\newcommand{\Oth}{\hat{\Omega}_{\tau}}

\newcommand{\pth}{\hat{\p}_{\tau}}

\newcommand{\Fth}{\hat{\mathcal{F}}_{\tau}}
\newcommand{\btk}{\beta_{\tau}^{k}}
\newcommand{\bthk}{\hat{\beta}_{\tau}^{k}}

\newcommand{\rt}{\rho_{\tau}}
\newcommand{\rth}{\hat{\rho}_{\tau}}
\newcommand{\ut}{u_{\tau}}
\newcommand{\uth}{\hat{u}_{\tau}}

\newcommand{\Onh}{\hat{\Omega}_{n}}
\newcommand{\Fnh}{\hat{\mathcal{F}}_{n}}

\newcommand{\bnhk}{\hat{\beta}_{n}^{k}}
\newcommand{\pn}{\p_{n}}
\newcommand{\pnh}{\hat{\p}_{n}}

\newcommand{\rn}{\rho_{n}}
\newcommand{\rnh}{\hat{\rho}_{n}}
\newcommand{\un}{u_{n}}
\newcommand{\unh}{\hat{u}_{n}}

\newcommand{\Oeh}{\hat{\Omega}_{\epsilon}}

\newcommand{\Feh}{\hat{\mathcal{F}}_{\epsilon}}
\newcommand{\pe}{\p_{\epsilon}}
\newcommand{\peh}{\hat{\p}_{\epsilon}}

\newcommand{\behk}{\hat{\beta}_{\epsilon}^{k}}

\newcommand{\re}{\rho_{\epsilon}}
\newcommand{\reh}{\hat{\rho}_{\epsilon}}
\newcommand{\ue}{u_{\epsilon}}
\newcommand{\ueh}{\hat{u}_{\epsilon}}

\newcommand{\Odh}{\hat{\Omega}_{\delta}}

\newcommand{\Fdh}{\hat{\mathcal{F}}_{\delta}}
\newcommand{\pd}{\p_{\delta}}
\newcommand{\pdh}{\hat{\p}_{\delta}}

\newcommand{\bdhk}{\hat{\beta}_{\delta}^{k}}

\newcommand{\rd}{\rho_{\delta}}
\newcommand{\rdh}{\hat{\rho}_{\delta}}
\newcommand{\ud}{u_{\delta}}
\newcommand{\udh}{\hat{u}_{\delta}}

\newcommand{\rgb}{\overline{\rho^{\gamma}}}

\newcommand{\rdhgb}{\overline{\rdh^{\gamma}}}
\newcommand{\rdhbb}{\overline{\rdh^{\beta}}}

\title{Random Pertubations of Viscous, Compressible Fluids: \\ Global Existence of Weak Solutions} 
\author{Scott A. Smith %\footnotemark[2]
}
\date{\today}

\begin{document}
\maketitle
\begin{abstract}
This article is devoted to the well-posedness of the stochastic compressible Navier Stokes equations.   We establish the global existence of an appropriate class of weak solutions emanating from large inital data, set within a bounded domain.  The stochastic forcing is of multiplicative type, white in time and colored in space.  Energy methods are used to merge techniques of P.L. Lions for the deterministic, compressible system with the theory of martingale solutions to the incompressible, stochastic system.  Namely, we develop stochastic analogues of the weak compactness program of Lions, and use them to implement a martingale method.  The existence proof involves four layers of approximating schemes.  We combine the three layer scheme of Feiresil/Novotny/Petzeltova for the deterministic, compressible system with a time splitting method used by Berthelin/Vovelle for the one dimensional stochastic compressible Euler equations. 
\end{abstract}

\section{Introduction}
%!TEX root = main.tex
This paper is devoted to the analysis of the initial boundary value problem for the compressible isentropic stochastic Navier Stokes equations.  This system of stochastic partial differential equations governs the evolution of a viscous, compressible fluid(or gas) subject to random perturbations by noise. The macroscopic state of the fluid is described by a pair $(\rho,u)$ consisting of the scalar, nonnegative density $\rho$ and an $\R^{d}$ valued velocity field $u$.  In the isentropic case, the system is written as
\begin{equation}  \label{Eq:I:SCNS_System}
\begin{cases}
\partial_{t}\rho + \Div \left (\rho u \right ) = 0  \\
\partial_{t}(\rho u) +\Div \left (\rho u \tensor u\right )+ \nabla P(\rho)-2\mu \Delta u-\lambda \nabla \Div u = \rho\sigma_{k}(\rho,\rho u,x)\dot{\beta}_{k} \\
(\rho(0),(\rho u)(0)) = (\rho_{0},m_{0})
\end{cases}
\end{equation}
where $\mu , \lambda >0$ are positive viscosity coefficients and $P(\rho)$ is the macroscopic pressure. The collection $\{ \beta_{k}\}_{k=1}^{\infty}$ consists of independent, $\R$ valued Brownian motions and $\{\sigma_{k}\}_{k=1}^{\infty}$ is a collection of $\R^{d}$ valued noise coefficients.  We fix a large time $T$ and pose the system \eqref{Eq:I:SCNS_System} on $(0,T] \times D$, supplemented with a Dirichlet boundary condition for the velocity.

The main contribution of this article is as follows: we introduce a notion of global weak solutions to \eqref{Eq:I:SCNS_System} and provide an existence proof under suitable hypotheses on the pressure, the data, and the noise.   
\subsection{Hypotheses and Relevant Literature}
Let us begin with our assumptions about the pressure law, the initial data, and the noise coefficients.
As is common in the literature for the deterministic equation, we impose an equation of state for the pressure law.  
\begin{Hyp} \label{Hyp:gamma} 
The pressure law takes the form $P(\rho)=\rho^{\gamma}$.  In dimensions $d=2,3$ we impose $\gamma > \frac{3}{2},\frac{9}{5}$ respectively, and for $d>3$ we assume $\gamma >\frac{d}{2}$.
\end{Hyp}
The initial data are a finite energy density/momentum pair satisfying a compatibiility criterion. 
\begin{Hyp} \label{Hyp:data} The initial data $(\rho_{0},m_{0})$ are deterministic.  The initial density $\rho_{0}$ is nonnegative and the following compatibility condition holds
\begin{equation} \label{Eq:Pre:compability_of_initial_data} 
m_{0}1_{ \{ \rho_{0}=0 \}} = 0.
\end{equation}
Moreover, the initial energy is finite:
\begin{equation} \label{Eq:Pre:finite_energy_of_initial_data} 
\int_{D} \left ( \frac{1}{\gamma-1} \rho_{0}^{\gamma} + \frac{1}{2}\frac{|m_{0}|^{2}}{\rho_{0}} \right ) dx < \infty.
\end{equation}
\end{Hyp}

We make two types of assumptions about the noise coefficients.  The first is a continuity hypothesis for each $k$ fixed and the second is a trace class type summability condition for the collection $\{ \sigma_{k}\}_{k=1}^{\infty}$.

\begin{Hyp}\label{Hyp:lip}  For each $k$, the coefficient $\sigma_{k}: \R^{+} \times \R^{d} \times D \to \R^{d}$ is bounded and continuous.  Moreover, the following uniform lipschitz condition holds: there exists a constant $C_{k}$ such that for each $m_{1},m_{2} \in \R^{d}$
\begin{equation} \label{Eq:Pre:Lipschitz_Noise_in_the_Momentum}
\sup_{(\rho,x) \in \R^{+} \times D} |\sigma_{k}(\rho,m_{1},x) - \sigma_{k}(\rho,m_{2},x)| \leq C_{k} |m_{1}-m_{2}|.
\end{equation}
In addition, for each $\rho_{1},\rho_{2} \in \R^{d}$
\begin{equation} \label{Eq:Pre:Lipschitz_Noise_in_the_Density}
\sup_{(m,x) \in \R^{d} \times D} |\sigma_{k}(\rho_{1},m,x) - \sigma_{k}(\rho_{2},m,x)| \leq C_{k} |\rho_{1}-\rho_{2}|.
\end{equation}
\end{Hyp}

\begin{Hyp} \label{Hyp:color} The sequence of coefficients $\{ \sigma_{k}\}_{k=1}^{\infty}$ satisfy the summability relation
\begin{equation}  \label{Eq:Pre:Trace_Class_Type_Summability}
\sum_{k=1}^{\infty} |\sigma_{k}|_{L^{\frac{2\gamma}{\gamma-1}}_{x}(L^{\infty}_{\rho,m})}^{2} < \infty,
\end{equation}
where we denote for $p \geq 1$
$$
 |\sigma_{k}|_{L^{p}_{x}(L^{\infty}_{\rho,m})} = \big | \sup_{\rho,m} |\sigma_{k}(\rho,m,\cdot ) |  \big |_{L^{p}_{x}}.
 $$
\end{Hyp}

A few remarks are in order.  A simple, but important point is that the initial data are described by a density/ momentum pair $(\rho_{0},m_{0})$ rather than a density/velocity pair. As a result, the initial velocity is undefined in the vaccum regions. Next we note that the assumptions on $\gamma$ imposed in Hypothesis \ref{Hyp:gamma} are slightly worse than in the original work of Lions \cite{lions1998mathematical} in dimension $3$, in that we ask for a strict inequality $\gamma > \frac{9}{5}$.  This simplifies the proof but is not necessary.  Using the tools developed by Feiresil in \cite{feireisl2001compactness}, \cite{feireisl2004dynamics} one can weaken the constraints to simply $\gamma > \frac{d}{2}$ in all dimensions.

Regarding Hypothesis \ref{Hyp:color}, a simple example to have in mind is a noise of the form $\rho \dot{W}$ where $\dot{W}(x,t)= \sum_{k=1}^{\infty}\sqrt{\lambda_{k}}e_{k}(x)$ for some orthonormal basis $\{ e_{k}\}_{k=1}^{\infty}$ of $[L^{2}_{x}]^{d}$ where 
$$
\sum_{k=1}^{\infty} \lambda_{k}|e_{k}|_{L^{\frac{2\gamma}{\gamma-1}}}^{2} < \infty.
$$
In this case, a pathwise approach to the solvability of system \eqref{Eq:I:SCNS_System} is possible, as shown in \cite{feireisl2013compressible}.  
 
The starting point for our analysis is a formal energy identity (see Section \ref{SubSec:P:5:Energy_Estimates}).  In the stochastic, compressible framework, the kinetic and potential energy are random processes which fluctutate due to noise and grow according to an Ito correction term.  To close an estimate on their moments(in $\omega$) and obtain an a priori bound for the SPDE, one is lead to the trace class type summability condition for the coefficients $\{ \sigma_{k}\}_{k=1}^{\infty}$ (Hypothesis \ref{Hyp:color}). This leads one to suspect that for reasonably smooth noise coefficients satisfying Hypothesis \ref{Hyp:color}, finite energy weak solutions to \eqref{Eq:I:SCNS_System} should exist.

Unfortunately, to construct these solutions, a simple Galerkin approximation procedure is not enough.  Instead, one requires several layers of approximating stochastic PDE's.  Within each layer, one establishes a priori bounds and implements a compactness method. Broadly speaking there are two principle steps, an application of a generalized version of the Skorohod theorem (Theorem \ref{Thm:Appendix:Jakubowksi_Skorohod}) to obtain a form of compactness and an alternative to the martingale representation theorem (Lemma \ref{Lem:Appendix:Three_Martingales_Lemma}) to pass limits in the stochastic integrals.  However, there are a number of subtleties in the implementation due to the limited control that the a priori bounds provide on the density.   

The literature devoted to the the deterministic, compressible system is extensive; the most fundamental for our work are the results of Lions \cite{lions1998mathematical} and Feiresil/Novotny/Petzeltova \cite{MR1867887}.  These works provide the main inspiration for the subtle analysis of the density and the particular approximating schemes, respectively.

There has also been an intensive study of the incompressible, stochastic Navier Stokes equations.  The works most relevant to our paper concern the construction of weak, martingale solutions.  See for instance \cite{bensoussan1995stochastic},\cite{capinski1994stochastic},\cite{flandoli1995martingale}, \cite{MR2118862}.  There are also a few articles concerning the non-homogenous, incompressible system; see \cite{sango2010density} and \cite{cutland2006stochastic}.  However, the literature concerning the stochastic, compressible system is rather scarce.  Some results are available in dimension one, see \cite{MR1760377}.  The most relevant for our analysis is the work of Berthelin/Vovelle \cite{berthelin2013stochastic} on the one dimensional compressible Euler equations.  This paper inspired the time splitting scheme used in the lowest level of our approximations.

 The paper \cite{feireisl2013compressible} studies the compressible Navier Stokes equations driven by a forcing of the form $\rho \Dot{W}$.  In this special case, one can change variables and work pathwise in $\omega$.  This technique is not generally available for other types of multiplicative noise.

The only existing literature on the compressible, stochastic Navier Stokes equations that is comparable to this work is a very interesting preprint of Breit/Hofmanova \cite{breithofm} .  We became aware of these results during the late stages of the write up of this work.  We emphasize that our results were independently concieved and obtained.  There are a number of similarities between their work and ours, but our hypotheses on the noise do not overlap, so neither result implies the other.  As a result, there are also several differences in the approximating schemes.  Finally, we should remark that our work is set on a bounded domain, rather than the torus.     

\subsection{Notion of Weak Solution and Statement of the Main Result}

In this article, we refer to $(\Omega,\mathcal{F},\p,\{ \mathcal{F}_{t}\}_{t=0}^{T}, \{ \beta_{k}\}_{k=1}^{\infty})$ as a {\it stochastic basis} provided $(\Omega,\mathcal{F},\p)$ is a probability space endowed with a collection $\{ \beta_{k}\}_{k=1}^{\infty}$ of one dimensional, independent Brownian motions adapted to $\{ \mathcal{F}_{t}\}_{t=0}^{T}$ (see Appendix A for more discussion).  

The solutions we construct are weak in both the analytic and the probabilistic sense.  Namely, a solution to \eqref{Eq:I:SCNS_System} is a pair $(\rho,u)$ satisfying the continuity and momentum equations in the analytically weak sense, relative to a stochastic basis of our choice. More precisely, weak solutions are defined as follows:
% Definition of Weak Solution
\begin{Def} \label{Def:Pre:Weak_Solutions}
A pair $(\rho,u)$ is a weak solution to the stochastic compressible Navier Stokes equations \eqref{Eq:I:SCNS_System} provided there exists a stochastic basis $(\Omega, \mathcal{F}, \{ \mathcal{F}^{t} \}_{t=0}^{T} , \p, \big \{ \beta_{k}\big \}_{k \geq 1})$ such that 
\begin{enumerate}
\item  \label{Def:Pre:Weak_Solutions:Item:Measurability}
The pair $\left (\rho,\rho u \right ): \Omega \times [0,T] \to [ L^{\gamma}_{x} \times L^{\frac{2\gamma}{\gamma+1}}_{x}]_{w}$ is an $\{\mathcal{F}^{t} \}_{t=0}^{T}$ progressively measurable stochastic process with $\p$ a.s. continuous sample paths. The velocity $u \in L^{2} \left (\Omega \times [0,T] ; H^{1}_{0,x} \right )$ is an equivalence class of $\{\mathcal{F}^{t} \}_{t=0}^{T}$ progressively measurable $H^{1}_{x}$ valued processes.\\
\item  \label{Def:Pre:Weak_Solutions:Item:Continuity_Equation}
Continuity Equation: For all $\phi \in C^{\infty}_{c}(D)$ and $t \in [0,T]$, the following equality holds $\p$ a.s.
\begin{equation} \label{Eq:Pre:Continuity_Equation}
\int_{D} \rho(t)\phi dx = \int_{D} \rho_{0} \phi dx + \int_{0}^{t}\int_{D} \rho u \cdot \nabla \phi dxds.
\end{equation}
\item \label{Def:Pre:Weak_Solutions:Item:Momentum_Equation}
Momentum Equation: For all $\phi \in \big [C^{\infty}_{c}(D) \big ]^{d}$ and $t \in [0,T]$, the following equality holds $\p$ a.s.
\begin{equation} \label{Eq:Pre:Momentum_Equation}
\begin{split}
&\int_{D} \rho u(t) \cdot \phi dx  = \int_{D} m_{0}\cdot \phi + \int_{0}^{t}\int_{D} [\rho u \tensor u-2\mu \nabla u-\lambda \Div u I ] : \nabla \phi + \rho^{\gamma}\Div \phi \, \, dxds\\
&+ \sum_{k=1}^{\infty}\int_{0}^{t}\int_{D} \rho \sigma_{k}(\rho , \rho u, x) \cdot \phi dx d \beta_{k}(s). 
\end{split}
\end{equation}
\item \label{Def:Pre:Weak_Solutions:Item:Energy_Estimate} The following energy estimate holds: for all $p \geq 1$
\begin{equation} \label{Eq:Pre:Energy_Estimate}
\E^{\p} \bigg [ |\sqrt{\rho}u |_{\LL{\infty}{2}}^{2p} + |\rho|_{\LL{\infty}{\gamma}}^{\gamma p} + |u|_{L^{2}_{t}(H^{1}_{0,x})}^{2p} \bigg ] < \infty .
\end{equation}
\end{enumerate}
\end{Def}

We can now give a precise statement of our main result:

\begin{Thm} \label{Thm:I:Main_Result}
Let $(\rho_{0},m_{0})$ be an initial density/momentum pair satisfying Hypothesis \ref{Hyp:data}.  Suppose the pressure $P(\rho)$ satisfies Hypothesis \ref{Hyp:gamma}. Let $\{ \sigma_{k}\}_{k \geq 1}$ be a collection of noise coefficients satisfying Hypotheses \ref{Hyp:gamma} and \ref{Hyp:lip}. 

Then there exists a weak solution $(\rho,u)$ to the stochastic compressible Navier Stokes system \eqref{Eq:I:SCNS_System} in the sense of Definition \ref{Def:Pre:Weak_Solutions}.          
\end{Thm}

\subsection{Outline of the Proof and Main Difficulties}
Let us proceed to an overview of the proof of Theorem \ref{Thm:I:Main_Result}.  The construction of weak solutions involves four layers of approximating schemes; labelled (from the lowest to the highest level) $\tau$, $n$, $\epsilon$, and $\delta$. 

In Section \ref{Section:Tau}, we prove existence for the lowest level of our approximating scheme.  Definition \ref{Def:Tau:Setup:Tau_Layer_Approximation} introduces the notion of a $\tau$ layer approximation to the compressible, stochastic Navier Stokes system \eqref{Eq:I:SCNS_System}.  The main result of the section, Theorem \ref{Thm:Tau:Setup:Tau_Layer_Existence}, shows that for each fixed $(n,\epsilon,\delta)$, one can construct a sequence $\{ (\rth,\uth) \}_{\tau >0}$ of $\tau$ layer approximations which are controlled uniformly (in all parameters, though we will only indicate dependence on one parameter at a time) in the sense of the energy estimate \eqref{Eq:Tau:Setup:Uniform_Bounds}.  

The proof is based on a time splitting method used by Berthelin/Vovelle in \cite{berthelin2013stochastic}. Half of the time, the determinstic system evolves, and the stochastic forcing is neglected.   The other half of the time, the density is frozen, and the system evolves only through the noise.  The evolution is sped up appropriately so there is consistency when the time splitting parameter is sent to zero (in Section \ref{Section:N}).  The main tools of this section are Propositions \ref{Prop:Tau:Determnistic_Existence_Result} and \ref{Prop:Tau:SPDE_Existence_for_the_Velocity}, which use classical fixed point arguments to obtain a basic existence result for both the deterministic and the stochastic systems, respectively.  The solutions are sufficiently regular so that Ito's formula may be applied, and the formal estimates of Section \ref{SubSec:P:5:Energy_Estimates} can be justified.  This leads to uniform bounds.

Section \ref{Section:N} is devoted to proving Theorem \ref{Thm:N:Setup:N_Layer_Existence}, an existence result for the second layer.  Definition \ref{Def:Tau:Setup:Tau_Layer_Approximation} introduces the notion of an $n$ layer approximation \eqref{Eq:I:SCNS_System}, and the goal is to construct a sequence $\{ (\rn,\un) \}_{n=1}^{\infty}$ of these which obey the uniform bounds \eqref{Eq:N:Setup:Uniform_Bounds}. 

The proof uses the existence theory at the $\tau$ layer together with a compactness method.  For each fixed $n \geq 1$, we apply Theorem \ref{Thm:Tau:Setup:Tau_Layer_Existence} and find a sequence $\{(\hat{\rho}_{\tau,n},\hat{u}_{\tau,n})\}_{\tau >0}$ of $\tau$ layer approximations obeying the uniform bounds \eqref{Eq:Tau:Setup:Uniform_Bounds}.  Our first step is to establish Proposition \ref{Prop:N:Skorohod}, which yields a limit point $(\rnh,\unh)$, together with a new sequence $\{(\rho_{\tau,n},u_{\tau,n})\}_{\tau >0}$ of $\tau$ layer approximations with improved compactness properties.  A caveat is that the new sequence is defined relative to a new probability space.  However, there exists a sequence $\{ \widehat{T}_{\tau}\}_{\tau >0}$ of measure preserving transformations (referred to by the author as recovery maps), which link the new and the old sequence by composition.  These mappings allow us to preserve information as we change probability spaces; in particular, ensuring that the new sequence solves the same equations, obeys the same uniform bounds and is generally unaltered in any of its arguments besides $\omega$.  Our main tool is Theorem \ref{Thm:Appendix:Jakubowksi_Skorohod}, a generalization of the classical result of Skorohod. It merges two recent extensions of the theorem; the first, due to Jakubowksi \cite{MR1453342}, permits random variables on a class of topological spaces and the other, due to by Vaart/Wellner \cite{MR1385671}, provides the recovery maps.  

A distinctive feature of the $\tau$ layer is that one requires stronger convergence in time than in other layers of the scheme.  This is required to compensate for the high frequency switching between the two types of evolution when $\tau$ is small.  This leads to some subtleties in the proof of tightness Lemma \ref{Lem:N:Tightness}. To obtain the necessary Holder estimate in time, one needs a probabilistic bootstrapping procedure to deal with the coupling between the density and the velocity.  

After Proposition \ref{Prop:N:Skorohod} is established, it is easy to pass to the limit in the parabolic equation on the new probability space (Lemma \ref{Lem:N:Id:Parabolic_Equation}).  To pass to the limit in the momentum equation (Lemma \ref{Lem:N:Id:Momentum_Equation}), we use a martingale method based on an appendix result \ref{Lem:Appendix:Three_Martingales_Lemma}, which provides a convienient characterization of a series of one dimensional stochastic integrals.  This was developed in \cite{MR3098063} as an alternative to the martingale representation theorem.  This method is used systematically throughout the paper when passing to the limit in the momentum equation at each layer.    

Section \ref{Section:Eps} is devoted to existence Theorem \ref{Thm:Eps:Setup:Eps_Layer_Existence}, the $\epsilon$ layer.  For each $\epsilon >0$ fixed, we can apply Theorem \ref{Thm:N:Setup:N_Layer_Existence} to obtain a sequence $\{ (\hat{\rho}_{n,\epsilon},\hat{u}_{n,\epsilon}) \}_{n=1}^{\infty}$ of $n$ layer approximations satisfying the uniform bounds \eqref{Eq:N:Setup:Uniform_Bounds}.  Again, the section splits into a compactness step, Proposition \ref{Prop:Eps:Skorohod}  and an identification step, Lemmas \ref{Lem:Eps:Id:Parabolic_Equation} and \ref{Lem:Eps:Id:Momentum_Equation}.  In this section, the spaces where the tightness Lemma \ref{Lem:Eps:Comp:Tightness} are proved become a bit more sophisticated.  In particular, we must use certain Banach spaces endowed with their weak or weak-$\star$ topology.  At this stage, and at all later compactness steps, the Jakubowski extension of the Skorohod theorem is essential. 

As $n \to \infty$, the most challenging term is $ \epsilon \nabla u_{n} \nabla \rho_{n}$ in the momentum equation, which corrects for the vanishing viscosity regularization in the energy balance.  To treat this difficulty, we adapt a technique of Feireisl in Lemma \ref{Lem:Eps:Id:Parabolic_Equation}, and upgrade the convergence of the density.  This allows us to use a martingale method again in Lemma \ref{Lem:Eps:Id:Momentum_Equation} and complete our stability analysis at this layer.  Finally, we appeal to standard lower semicontinuity arguments to uphold our uniform bounds.

In Section \ref{Section:Del} we build our final approximating scheme, a sequence $\{ (\rdh, \udh) \}_{\delta >0}$ of $\delta$ layer approximations to \eqref{Eq:I:SCNS_System}.  Again, the proof still splits broadly into two parts; a compactness step and an identification procedure.  However, as $\epsilon \to 0$ we encounter several new difficulties related to the pressure $\nabla (\re^{\gamma}+ \delta \re^{\beta})$.  The first is that basic energy bounds only provide moment estimates on the $L^{1}_{t,x}$ norm of the pressure.  To obtain tightness of the pressure sequence, we must improve these bounds.  In Proposition \ref{Prop:Del:Comp:Integrability_Gains}, we prove stochastic analogues of the integrability gains observed by Lions \cite{lions1998mathematical}.  Namely, we show that our weak solutions inherit additional integrability from the equation itself.  

An even more serious difficulty is passing to the limit in the pressure, which requires strong convergence of the density.  The compactness step, Proposition \ref{Prop:Del:Comp:Skorohod} is setup in a way that anticipates this problem.  We start with a preliminary identification step, passing to the limit in the continuity equation (Lemma \ref{Lem:Del:Id:Continuity_Equation}) and also the momentum equation (Lemma \ref{Lem:Del:Id:Momentum_Equation}), modulo a possible change in the pressure law. 

Next we improve upon our preliminary identification step and work towards the strong convergence of the density.  In Lemma \ref{Lem:Del:Strong:Weak_Continuity_Of_The_Effective_Viscous_Flux}, we prove a stochastic analogue of the weak continuity of the effective viscous pressure, originally discovered by Lions \cite{lions1998mathematical}.  Namely, the weak continuity holds after averaging out the contribution of the stochastic integral.  In Lemma \ref{Lem:Del:Strong_Convergence_Of_Density}, this result is used together with techniques from the theory of renormalized solutions of the transport equation \cite{MR1022305} to prove the strong convergence of the density and complete the $\delta$ layer existence proof.

At this stage, we have suceeded in constructing a sequence $ \{ (\rdh, \udh) \}_{\delta >0}$ of $\delta$ layer approximations to the stochastic Navier Stokes equations which obey the uniform bounds \eqref{Eq:Del:Setup:Uniform_Bounds}.  We now proceed to the prove of our main result, Theorem \ref{Thm:I:Main_Result}.  As usual, we begin with a compactness step, Proposition \ref{Prop:Pf:Comp:Skorohod}.  The statement of the theorem is rather technical, but natural if one anticipates again the difficulties with the pressure term.  In Proposition \ref{Prop:Pf:Comp:Integrability_Gains} we prove another integrability gain, which leads to moment bounds of the $L^{p}_{t,x}$ norm of the density.  The assumptions on $\gamma$ in Hypothesis \ref{Hyp:gamma} ensure $p \geq 2$.  This avoids difficulties with renormalizing the transport equation.

The identification procedure faces a new difficulty at this stage, regarding the nonlinear compositions in the multiplicative noise.  Namely, at each of the earlier stages in the analysis, we used the parameter $\delta$ to regularize the density and momentum before composing with the diffusion coefficient $\sigma_{k}$.  This was crucial for checking the two parts in our key identifiation Lemma \ref{Lem:Appendix:Three_Martingales_Lemma}.  In the $\delta$ layer existence proof, it allowed us to make a preliminary passage to the limit in the equation before proceeding to the proof of the strong convergence of the density.  In this final step, this is simply not possible.  Instead, we can only prove that the momentum process minus its drift is a martingale (Lemma \ref{Lem:Pf:Id:Momentum_Martingale}).  Nonetheless, we show that this is in fact enough information to prove again a stochastic analogue of weak continuity of the effective viscous flux.  Namely, Lemma \ref{Lem:Pf:Strong:Averaged_Ito_Product_Rule} uses the momentum martingale together with a regularization procedure to establish an averaged It\^{o} product rule, which is in turn enough to establish the weak continuity, Lemma \ref{Lem:Pf:Strong:Weak_Continuity_Of_The_Effective_Viscous_Pressure}.  This is used again to prove convergence of the density in Lemma \ref{Lem:Pf:Strong:Strong_Convergence_Of_Density} and strong convergence of the momentum in Lemma \ref{Lem:Pf:Strong:Lions_Lemma}.  Finally, we conclude the proof of our main result by passing the limit in the momentum equation, once and for all in Lemma \ref{Lem:Pf:Id:Momentum_Martingale}.

The remainder of the paper is organized as follows.  Section \ref{Section:Pre} contains a preliminary discussion of our result, while Sections \ref{Section:Tau}-\ref{Section:Del} are devoted to the proof of Theorem \ref{Thm:I:Main_Result}.   Section \ref{SubSec:P:1:Lit_Review} reviews the key ideas from the literature; Section \ref{SubSec:P:2:Notation} sets the notation; Section \ref{SubSec:P:5:Energy_Estimates} shows the formal energy estimates; Section \ref{SubSec:P:6:Discussion} discusses the hypotheses in more detail. Section \ref{Section:Tau} establishes existence at the $\tau$ layer; Section \ref{Section:N} establishes existence at the $n$ layer; Section \ref{Section:Eps} builds solutions to the $\epsilon$ layer; Section \ref{Section:Del} builds solutions to the $\delta$ Layer; Section \ref{Section:Pf} passes to the limit in the $\delta$ layer to complete the proof of Theorem \ref{Thm:I:Main_Result}.

\section{Preliminaries} \label{Section:Pre}
%!TEX root = main.tex
\subsection{Background/Literature Review}\label{SubSec:P:1:Lit_Review}
%!TEX root = main.tex
The section consists of some high level remarks regarding the existence theory for the deterministic, compressible system and the incompressible, stochastic system.  Both these theories strongly influenced the methodology in this paper, so we choose to review some of the main ideas.  The reader may wish to skim this section on a first reading since, strictly speaking, we do not quote directly from either of the theories.  

\subsubsection{Existence Theory for the Deterministic, Compressible System}
The literature on the deterministic system is extensive, and we will not attempt to give a complete discussion of the current status of the field.  Instead, we focus on the results that provide the guiding principles for our work.  The seminal work of P.L. Lions \cite{lions1998mathematical} initiated a large data global existence theory for finite energy weak solutions.  Let us give a very rough outline of the construction.  The proof splits into two parts; proving that the solution set is weakly compact and constructing several layers of approximating schemes.  That is, suppose $\{(\rho_{n},u_{n}) \}_{n=1}^{\infty}$ is a sequence of weak solutions(or well chosen approximate solutions) which are uniformly bounded in the natural energy space. The strategy is to show that if the initial data are stongly convergent, then the corresponding solutions must converge to a solution $(\rho,u)$ emanating from the limit point of the data. Since the pressure is a nonlinear function of the density, the only feasible way to proceed is by proving that the sequence of densities $\{\rho_{n}\}_{n=1}^{\infty}$ converge strongly to $\rho$.  However, the continuity equation is driven by too rough of a velocity field to provide any control on the densities in a positive Sobolev space.  Hence, this is a nontrivial task and the basic energy bounds alone are not enough.  Nonetheless, Lions found a more subtle mechanism in the nonlinear structure that gives compactness.

To motivate the proof, recall the method for obtaining compactness of the density in the Di Perna/Lions \cite{MR1022305} theory of the transport equation driven by a ``rough'' velocity field with bounded divergence.  One starts with a convienient renormalization(meaning just a smooth function to be applied to the density), for instance $\beta(\rho)=\rho^{2}$ and renormalizes the equation at the level of both the approximation and the level of the limiting solution, known a priori to be a renormalized solution of the transport equation in its own right.  If strong oscillations in the density sequence are present, the operations of composition with a nonlinear function and extraction of a weak limit do not generally commute.  However, the renormalized form allows one to track the evolution of this ``commutator'' $\overline{\rho^{2}}-\rho^{2}$ and a Gronwall argument shows that if compression effects are limited, strong convergence of initial densities implies the ``commutator'' vanishes for all later times.  Unfortunately, one cannot apply this method directly to the compressible Navier Stokes system because the known a priori bounds are not enough to rule out the possibility of extreme compression(or expansion).  To proceed, Lions made the crucial observation that a sort of ``monotoncity miracle'' occurs for particular pressure laws, and in some sense, it suffices that the so called effective viscous pressure $P(\rho)-(2\mu+\lambda)\Div u$ is ``slightly well behaved''(in the sense of a certain weak continuity property), even if the divergence of the velocity field alone is potentially unbounded.  The importance of this quantity had already been observed in a simpler context by D. Serre.  Moreover, the evolution of this quantity is readily available upon taking the divergence, followed by the inverse laplacian on both sides of the momentum equation.  By studying this quantity before and after a preliminary passage to the limit in the momentum equation, one is able to prove a subtle compactness result, known as the weak continuity of the effective viscous pressure, which is just barely enough to complete an analysis of a similar ``commutator'' as in the bounded divergence case, and hence conclude the strong convergence of the density. 

The original work of Lions considered power laws with $\gamma$ large enough to ensure that the continuity equation could be renormalized, see Hypothesis \ref{Hyp:gamma} below. Several years later, Feiresil introduced in \cite{feireisl2001compactness} some additional tools which, combined with Lions general strategy of proof, succeeded in weakening the hypothesis on $\gamma$ in dimensions two and three, to what seems to be the critical level \footnote{Below this level, one can just barely give a meaning to the flux term in the momentum equation, and Lions method seems to break down.} $\gamma >\frac{d}{2}$.  This is a nontrival task, since for low enough values of $\gamma$, one dips below the integrabililty required to classically renormalize the continuity equation.  More importantly for this paper,  with co-authors in \cite{MR1867887} Feiresil developed a somewhat simplified(but still rather long) approximation scheme, based on a Galerkin appromation for the velocity, a vanishing viscosity regularization for the continuity equation, and an artifical pressure regularization.

\subsubsection{Existence Theory for the Stochastic, Incompressible System}

There is also a fairly developed literature concerning the stochastic Navier Stokes equations for incompressible fluids, which we will not review in much depth.  Naturally, much more is known in dimension two, but the existence of weak solutions is known in any dimension.  In this regard, the primary inspiration for our work is Flandoli's construction of weak martingale solutions in \cite{flandoli2008introduction}; see also \cite{capinski1994stochastic}.  The main point we wish to emphasize is that these solutions are weak in both the analytic and the probabilistic sense.  Namely, one is allowed to input the probability space where the solution is built, along with a convienient choice of Brownian motions and a sufficiently large filtration(potentially larger than the information needed to assess the values of the Brownian motions alone).  These inputs are referred to as the stochastic basis, and after they have been fixed, one asks that the momentum equation holds in an analytically weak sense in space and in the Ito sense in time.

To understand the virtue of flexibility in the choice of a stochastic basis, recall Leray's construction of weak solutions to the deterministic, incompresible Navier Stokes equations.  The key point is that uniform bounds in $\LW{2}{1}{2} \cap C_{t}(H^{-1}_{x})$ allow one to apply the Aubin/Lions lemma and obtain strong compactness in $\LL{2}{2}$(and hence weak stability of the flux term), leading to a straightforward(from a modern point of view) weak compactness theory.  At a superficial level, in the stochastic case, there is an additional variable $\omega$, and the possibility of ``oscillations'' in this variable may block the compactness upgrade from the space/time bounds.  However, if one is content with only accessing the probability law of the solution, then there is a classical fix.  Namely, if one can show that the sequence of Galerkin approxmations becomes uniformly concentrated(up to a set of very small probability) on $\LW{2}{1}{2} \cap C_{t}(H^{-1}_{x})$, the Skorohod embedding(for random variables on complete separable metric spaces) guarantees the existence of a new sequence of random variables(with the same distribution) on the unit interval, along with a limit point, for which the usual $L^{2}_{t,x}$ convergence holds pointwise.  Essentially, under an appropriate change of variables one is able to convert information that only holds on average on the initial probability space, to information that holds in every state of the universe of a well chosen probability space.  One could visualize this in one dimension by noting that given a sequence of bumps sliding back and forth across the unit interval on smaller and smaller measure sets, if we rearrange the sequence based on the distribution of mass, one converts the typical counterexample to ``weak convergence implies pointwise convergence'' into a pointwise converging sequence, without altering its probability law. 
\subsection{Notation}\label{SubSec:P:2:Notation}
%!TEX root = main.tex
Recall that the equation is posed in the space/time domain $D \times [0,T]$.  Assume that $D$ is a connected, bounded open subset of $\R^{d}$ with smooth boundary.  The shorthand notation $\LL{q}{p}, \LW{q}{k}{p} , \WL{k}{q}{p}$ is used to denote the spaces $L^{q} \left ([0,T] ; L^{p}(D)\right ) ,  L^{q} \left ([0,T] ; W^{k,p}(D) \right ) , W^{k,q} \left ([0,T] ; L^{p}(D) \right )$ respectively, where each space is understood to be endowed with its strong topology.  We will often use the same notation to denote scalar functions in $\LL{q}{p}$ and vector valued functions(with $d$ components) in $[\LL{q}{p}]^{d}$, but the meaning will always be clear from the context.  To emphasize when one of the spaces above is endowed with its weak topology, we write $\LLw{q}{p}, \LWw{q}{k}{p}$.  Also, the abbreviation $\CLw{p}$ denotes the topological space of weakly continuous functions $f:[0,T] \to L^{p}(D)$.  The space $W^{k,p}_{0,x}$ is the closure of the smooth compactly supported functions, $C^{\infty}_{c}(D)$, with respect to the $W^{k,p}(D)$ norm. Moreover, we denote $W^{1,2}_{0,x}$ as $H^{1}_{0,x}$.  Given a probability space $(\Omega,\F,\p)$ and a Banach space $E$,  let $L^{p} \big (\Omega ; E \big )$ be the collection of equivalence classes of $\mathcal{F}$ measurable mappings $X : \Omega \to E$ such that the $p^{th}$ moment of the $E$ norm is finite.  Again, we write $L^{p}_{w} \big (\Omega ; E \big )$ when emphasizing that the space is endowed with its weak topology.  To define the sigma algebra generated by various random variables, we use a restriction operator $r_{t}: C \big ([0,T] ; E \big ) \to C \big ([0,t] ; E \big )$ which realizes a mapping $f : [0,T] \to E$ as a mapping  $r_{t}f:[0,t] \to E$.  The same notation is used for the restriction of an equivalence class $f \in L^{2} \big ([0,T] ; E \big ) $ to $r_{t}f \in L^{2} \big ([0,t] ; E \big )$.  We denote $\mathcal{A} = \nabla \Delta^{-1}$, understood to be well defined on compactly supported distributions in $\R^{d}$.  The symbol $\mathcal{B}$ is reserved for the Bogovski operator, see the remarks preceding Lemma \ref{Thm:Appendix:Bogovoski} for the definition of the operator, along with its basic properties.  Given two $d \times d$ matrices $A,B$, $A:B$ denotes a Frobenius matrix product.  The notation $A \leqs B$ denotes inequality up to an insignificant constant.  The notion of insignificance will be clear from the context. 
\subsection{Formal Energy Estimates}\label{SubSec:P:5:Energy_Estimates}
%!TEX root = main.tex
In this section, we present a formal derivation of the basic energy equality for the system \eqref{Eq:I:SCNS_System}.  The kinetic  and potential energy dissipate in the usual way, but also fluctuate due to noise and grow according to an Ito correction.  Nonetheless, we will see that Hypothesis \ref{Hyp:color} ensures that the moments(in $\omega$) are controlled by the initial energy.         

As there is no Ito term in the continuity equation,  we may apply the ordinary product rule(in time) to $\partial_{t}(\rho u)$, the product rule in space to $\Div(\rho u \tensor u)$, and use the equation for $\partial_{t}\rho$ to formally rewrite  
%$\partial_{t}(\rho u) + \Div(\rho u \tensor u) = \rho[\partial_{t}u + (u \cdot \nabla)u]$.
%\begin{equation}
%\partial_{t}(\rho u) + \Div(\rho u \tensor u) = \partial_{t}\rho\, u + \rho \partial_{t}u + \Div(\rho u) u + \rho (u \cdot \nabla) u = \rho[\partial_{t}u + (u \cdot \nabla)u]
%\end{equation}
\eqref{Eq:I:SCNS_System} as
\begin{equation*}
\begin{cases}
\partial_{t}\rho + \Div \left (\rho u \right ) = 0  \\
 \rho[\partial_{t}u + (u \cdot \nabla)u]+ \nabla(\rho^{\gamma}) -2\mu \Delta u-\lambda \nabla \Div u = \rho\sigma_{k}(\rho,\rho u,x)\dot{\beta}_{k}. \\
\end{cases}
\end{equation*}
Dividing the momentum equation by $\rho$ and noting the remaining noise term, we see $|du|^{2}= \sum_{k=1}^{\infty} |\sigma_{k}|^{2}$.
To derive the energy identity, multiply the momentum equation by $u$ and integrate over $D$.  Note that Ito's formula gives
 $\rho\partial_{t}u \cdot u = \rho \partial_{t}(\frac{|u|^{2}}{2}) - \frac{1}{2}\rho|du|^{2}$. Using the continuity equation again, 
\begin{align*}
&\rho[\partial_{t}u + (u \cdot \nabla)u] \cdot u = \rho \partial_{t}(\frac{|u|^{2}}{2}) + \rho u \cdot \nabla(\frac{|u|^{2}}{2}) -\frac{1}{2}\rho\sum_{k=1}^{\infty}|\sigma_{k}|^{2} \\
&=\partial_{t}(\frac{1}{2}\rho|u|^{2}) + \Div(\frac{1}{2}\rho |u|^{2}) -\frac{1}{2}\rho\sum_{k=1}^{\infty}|\sigma_{k}|^{2}. 
\end{align*}
Multiplying the continuity equation by $\frac{\gamma}{\gamma-1}\rho^{\gamma-1}$ and noting $\Div(\rho u)\frac{\gamma}{\gamma-1}\rho^{\gamma-1} = \Div(\frac{\gamma}{\gamma-1}\rho^{\gamma}u)-u \cdot \nabla(\rho^{\gamma})$ leads to the identity
$u \cdot \nabla(\rho^{\gamma})=\partial_{t}(\frac{1}{\gamma-1}\rho^{\gamma}) + \Div(\frac{\gamma}{\gamma-1}\rho^{\gamma}u)$.  Combining these observations, integrating the dissipation by parts and using the dirichlet boundary condition gives for all $t \in [0,T]$
\begin{equation}
\begin{split}
&\int_{D}\left ( \frac{1}{2}\rho |u|^{2}(t)+\frac{1}{\gamma-1}\rho^{\gamma}(t) \right )dx + \int_{0}^{t}\int_{D}2\mu |\nabla u|^{2}+\lambda (\Div u)^{2}dxds  \\
&=\int_{D}\left ( \frac{1}{2}\frac{|m_{0}|^{2}}{\rho_{0}}+\frac{1}{\gamma-1}\rho_{0}^{\gamma} \right )dx + \sum_{k=1}^{\infty}\int_{0}^{t}\int_{D} \rho u\cdot \sigma_{k}(\rho,\rho u,x) dx d\beta_{k}(s) \\
&+\frac{1}{2}\sum_{k=1}^{\infty}\int_{0}^{t}\int_{D}\rho |\sigma_{k}(\rho,\rho u,x)|^{2}dxds . 
\end{split}
\end{equation}
To control the moments of the LHS, note that the series of stochastic integrals on the RHS is a martingale with quadratic variation(at time $T$) given by $\sum_{k=1}^{\infty} \int_{0}^{T} (\int_{D}\rho u \cdot \sigma_{k}dx )^{2}ds$.  Hence we may use the Burkholder/Davis/Gundy inequality followed by H\"older and Hypothesis \ref{Hyp:color} to find for $p>1$
\begin{align*}
&\E \left [ \sup_{t \in [0,T]}  \big|\sum_{k=1}^{\infty} \int_{0}^{t}\int_{D}\rho u \cdot \sigma_{k}dxd\beta_{k}(s) \big |^{p} \right ]+\E \left [ \big |\sum_{k=1}^{\infty}\int_{0}^{T}\int_{D}\rho |\sigma_{k}|^{2}dxds \big |^{p} \right ] \\
&\leqs \E \left [    \big | \sum_{k=1}^{\infty} \int_{0}^{T} (\int_{D}\rho u \cdot \sigma_{k}dx )^{2}ds \big |^{p/2} \right ]+\E \left [ \big |\sum_{k=1}^{\infty}\int_{0}^{T}\int_{D}\rho |\sigma_{k}|^{2}dxds \big |^{p} \right ] \\
&\leqs \left ( \sum_{k=1}^{\infty}|\sigma_{k}|_{L^{\frac{2\gamma}{\gamma-1}}_{x}(L^{\infty}_{\rho,m})}^{2}\right )^{\frac{p}{2}} \E \left [   \big |\int_{0}^{T}|\rho u(s)|_{L^{\frac{2\gamma}{\gamma+1}}_{x}}^{2}ds \big |^{p/2} \right ] \\
&+\left ( \sum_{k=1}^{\infty}|\sigma_{k}|_{L^{\frac{2\gamma}{\gamma-1}}_{x}(L^{\infty}_{\rho,m})}^{2}\right )^{p}\E \left [\big |\int_{0}^{T}|\rho(s)|_{L^{\gamma}_{x}}ds\big |^{p} \right ] \\
& \leqs \E \left [|\rho|_{\LL{\infty}{\gamma}}^{\frac{p}{2}}|\sqrt{\rho}u|_{\LL{\infty}{2}}^{p} \right ]+ \E \left [|\rho|_{\LL{\infty}{\gamma}}^{p} \right ].
\end{align*}
Maximizing over $[0,T]$, then taking the expectation of the $p$ moment on both sides of the energy identity yields
\begin{align*}
&\E \left [|\sqrt{\rho}u|_{\LL{\infty}{2}}^{2p}+|\rho|_{\LL{\infty}{\gamma}}^{\gamma p} + |\nabla u|_{\LLs{2}}^{2p} \right ] \leqs \E \left [|\frac{m_{0}}{\sqrt{\rho_{0}}}|_{L^{2}_{x}}^{2p}+|\rho_{0}|_{L^{\gamma}_{x}}^{\gamma p} \right ] \\
&+ \E \left [|\rho|_{\LL{\infty}{\gamma}}^{\frac{p}{2}}|\sqrt{\rho}u|_{\LL{\infty}{2}}^{p} \right ]+ \E \left [|\rho|_{\LL{\infty}{\gamma}}^{p} \right ].
\end{align*}
Applying Cauchy's inequality with the exponent pairs $(2\gamma,2,\frac{2\gamma}{\gamma-1})$ and $(\gamma,\frac{\gamma}{\gamma-1})$ followed by Poincare yields
\begin{equation} 
\E \left [|\sqrt{\rho}u|_{\LL{\infty}{2}}^{2p}+|\rho|_{\LL{\infty}{\gamma}}^{\gamma p} + |u|_{\LHZ{2}{1}}^{2p} \right ] \leqs \E \left [|\frac{m_{0}}{\sqrt{\rho_{0}}}|_{L^{2}_{x}}^{2p}+|\rho_{0}|_{L^{\gamma}_{x}}^{\gamma p} \right ] + 1 .
\end{equation}
\subsection{A Few Technical Remarks}\label{SubSec:P:6:Discussion}
%!TEX root = main.tex

In \cite{lions1998mathematical}, Lions treats(among other possible assumptions) a forcing of the form $\rho f$ where $f \in \LL{1}{\frac{2\gamma}{\gamma-1}}$.  Hence, the spatial exponent in Hypothesis \ref{Hyp:color} is a familiar one.  However, note that if we were to treat a time dependent $\sigma_{k}$, the summability criterion would require a norm of $\sigma_{k}$ in $L^{2}_{t}(L^{\gamma}_{x}L^{\infty}_{\rho,m})$ in order to close the energy estimates as in Section \ref{SubSec:P:5:Energy_Estimates}.

Next we make a few remarks regarding the notion of weak solution, Definition \ref{Def:Pre:Weak_Solutions}.  Note that we add noise in the spirit of Krylov \cite{MR1377477} by working directly with a series of one dimensional stochastic integrals.  Let us check that Part \ref{Def:Pre:Weak_Solutions:Item:Measurability} of Definition \ref{Def:Pre:Weak_Solutions} combined with Hypotheses \ref{Hyp:lip}, \ref{Hyp:color} imply that the series in \eqref{Eq:Pre:Momentum_Equation} admits a well defined continuous martingale version.  For each $k$, define the process $f_{k}$ by the relation
\begin{align*}
&f_{k}(s) = \int_{D}\rho(x,s)\sigma_{k}(\rho(x,s),\rho u(x,s),x) \cdot \phi(x)dx .
%&f_{k}^{\delta}(s) = \int_{D}\rho(x,s)\sigma_{k}(\rho* \eta_{\delta}(x,s),(\rho u)* \eta_{\delta}(x,s),x) \cdot \phi(x)dx 
\end{align*}
%where $\eta_{\delta}$ is the standard mollifier(in space) and for each $s$, $\left (\rho(s),\rho u(s) \right )$ is understood to be extended by zero outside of $D$.
In view of Lemma \ref{Prop:Appendix:Defining_A_Series_Of_Stochastic_Integrals} it suffices to check that $f_{k}$ is an $\{ \mathcal{F}^{t} \}_{t=0}^{T}$ progressively measurable process and 
\begin{equation*} \label{eq:summRel}
\sum_{k=1}^{\infty} \E^{\p} \left [ \int_{0}^{T} f_{k}^{2}(s) ds\right ] <\infty.
\end{equation*}
The desired summability follows from H\"older and Hypothesis \ref{Hyp:color} since
\begin{align*}
&\E^{\p} \big [\int_{0}^{T} \left ( \int_{D}\rho(x,s)\sigma_{k}(\rho(x,s),\rho u(x,s),x) \cdot \phi(x)dx \right )^{2}ds\big ] \\
&\leq |\phi|_{L^{\infty}_{x}}|\sigma_{k}|_{L^{\frac{\gamma}{\gamma-1}}_{x}(L^{\infty}_{\rho,m})}^{2}\int_{0}^{T} \E^{\p} \left [|\rho(s)|_{L^{\gamma}}^{2} \right ]ds \leqs |\sigma_{k}|_{L^{\frac{2\gamma}{\gamma-1}}_{x}(L^{\infty}_{\rho,m})}^{2}.  
\end{align*}
Using the continuity Hypothesis \ref{Hyp:lip}, one may check the measurability of the following map(by regularization)
$$
(\rho,m) \in \left [ L^{\gamma}_{x} \times L^{\frac{2\gamma}{\gamma-1}}_{x} \right ]_{w}  \to \int_{D} \rho(x) \sigma_{k}(\rho(x),m(x),x)dx.
$$
Hence, $f_{k}$ inherits progressive measurability from the density/momentum pair $(\rho,\rho u)$, in view of Part \ref{Def:Pre:Weak_Solutions:Item:Measurability}.  

Also, the  careful reader may wonder whether there is some ambiguity in the notion of the momentum equation \ref{Def:Pre:Weak_Solutions:Item:Momentum_Equation} regarding measure zero sets(due to the fact that we do not work explicitly with infinite dimensional stochastic integrals).  However, this is not the case, and one can check that it is equivalent to ask for a universal set of full $\p$ measure where the weak form holds for all $\phi \in \left [ C_{c}^{\infty}(D) \right ]^{d}$ simultaneously.  This follows from a density argument provided one chooses a suitable modification of the stochastic integrals.     

We should also mention that it is unclear(to the author) whether the velocity field $u$ inherits even a weak form of continuity in time from the equation itself.  This is the reason why we do not ask that the velocity $u$ is a stochastic process in the usual sense(unlike the density/momentum pair), and is only identified up to equivalence relations in a class of non-anticipating processes.  Also note that the dirichlet boundary condition is understood in the weak sense.

\section{$\tau$ Layer Existence} \label{Section:Tau}
%!TEX root = main.tex
% Main Result of the Section
%!TEX root = main.tex
In this section, we build the first layer of our approximating scheme, the $\tau$ layer.  Each of the parameters $n,\epsilon,$ and $\delta$ are present in the notion of solution, Definition \ref{Def:Tau:Setup:Tau_Layer_Approximation} below, but they are frozen in this section, so we only indicate dependence of the approximating sequence on $\tau$, the time splitting parameter.  We partition the time interval $[0,T]$ into $\frac{T}{\tau}$ time intervals of length $\tau$, where $\frac{T}{\tau}$ is assumed to be an even integer.  Denoting $t_{j}=j \tau$, we define the functions $\htd$ and $\hts$ via
\begin{equation} \label{Eq:Tau:Setup:htd}
 \htd(s) = \sum_{j=0}^{\frac{T}{2\tau}-1}1_{(t_{2j},t_{2j+1}]}(s)=1-\hts(s). 
\end{equation} 
The main result of this section is the following:
\begin{Thm}  \label{Thm:Tau:Setup:Tau_Layer_Existence}
Let $(\hat{\Omega}, \Fh, \{\Fh^{t} \}_{t=0}^{T} , \ph , \{ \hat{\beta}^{k}\}_{k=1}^{n})$ be a stochastic basis and suppose $\{\Fh^{t} \}_{t=0}^{T}$ is the filtration generated by the collection of Brownian motions $\{ \hat{\beta}^{k}\}_{k=1}^{n}$. 

There exists a sequence $\{(\rth,\uth)\}_{\tau > 0}$ of $\tau$ layer approximations(in the sense of Definition \ref{Def:Tau:Setup:Tau_Layer_Approximation} below), relative to the given stochastic basis, such that for all $p\geq 1$
\begin{equation} \label{Eq:Tau:Setup:Uniform_Bounds}
\begin{split} 
&\sup_{\tau > 0} \E^{\ph} \bigg [ |\sqrt{\rth}\uth |_{\LL{\infty}{2}}^{2p} + |\rth|_{\LL{\infty}{\gamma}}^{\gamma p}+|\htd\uth|_{\LHZ{2}{1}}^{2p} \bigg ] < \infty \\ 
& \sup_{\tau >0}\E^{\ph} \bigg [|\delta^{\frac{1}{\beta}}\rth|_{\LL{\infty}{\beta}}^{\beta p} + |\htd \epsilon^{\frac{1}{2}} \nabla(\rth^{\frac{\gamma}{2}} + \delta^{\frac{1}{2}}\rth^{\frac{\beta}{2}})|_{\LL{2}{2}}^{2p} \bigg ] < \infty .
\end{split}
\end{equation}
\end{Thm}
Let us proceed to a precise definition of a $\tau$ layer approximation.  To do so, we introduce three elements of our approximating scheme: a finite dimensional space where the velocity evolves, a regularization of the multiplicative structure of the noise, and an artifical pressure.

%Definition of Finite Dimensional Spaces
Let $\{ (X_{n},\Pi_{n}) \}_{n=1}^{\infty}$ be a collection of finite dimensional subspaces of $[L^{2}_{x}]^{d}$, together with a sequence of linear operators $\Pi_{n} : [L^{2}_{x}]^{d} \to X_{n}$ which satisfy:
\begin{Hyp}\label{Hyp:Projections}
The space $X_{n}$ is a spanned by a finite number of compactly supported vector fields in $[C^{3}_{x}]^{d}$.  $\Pi_{n}: [L^{2}_{x}]^{d} \to X_{n}$ is a linear operator.  Let $s \in \{0,1,2,3\}$ and $1<p<\infty$.  For each $u \in [W^{s,p}_{x,0}]^{d}$ 
$$
\lim_{n \to \infty} |\Pi_{n}u-u|_{W^{s,d+1}_{x,0}}=0.
$$
\end{Hyp}
The pair $(X_{n},\Pi_{n})$ can be constructed using a wavelet expansion.  For more details on wavelet expansions in domains, see \cite{MR2250142}.
%
%
% Regularization of the Noise Coefficients 
Next, let $C^{+}_{x}$ be the cone of positive functions in $C_{x}$ and $\eta_{\delta}$ a standard mollifier.  Define the operator $\sigma_{k,\tau,n,\delta} : C^{+}_{x} \times [L^{1}_{x}]^{d} \to X_{n}$ by  
\begin{align*}
\sigma_{k,\tau,n,\delta}(\rho,m) = \Pi_{n} \circ \sigma_{k}(\rho * \eta_{\delta}(\cdot), [(\rho \wedge \frac{1}{\tau})\frac{m}{\rho} ]*\eta_{\delta}(\cdot),\cdot),
\end{align*}
where $(\rho,m)$ are understood to be extended by zero outside of $D$.  
%
%
% Definition of the Artificial Pressure

Finally, the original pressure in the momentum equation will be replaced by an ``artifical'' one of the form $P(\rho)= \rho^{\gamma}+\delta \rho^{\beta}$ for a sufficiently large power $\beta$.  Specifically, we require that
\begin{equation} \label{Eq:Tau:Artifical_Pressure_Beta_Constraints}
\beta > \text{max}(d,2\gamma,4) .
\end{equation}
%
%
%
%Definition of a Tau Layer Approximation
\begin{Def} \label{Def:Tau:Setup:Tau_Layer_Approximation}
A pair $(\rth,\uth)$ is defined to be a $\tau$ layer approximation to the compressible Stochastic Navier Stokes equations \eqref{Eq:I:SCNS_System} provided there exists a stochastic basis $(\Oth, \Fth, \{\Fth^{t} \}_{t=0}^{T} , \pth , \{ \hat{\beta}_{\tau}^{k}\}_{k=1}^{n})$ such that:
\begin{enumerate}
\item \label{Def:Tau:Setup:Tau_Layer_Approximation:Item:Brownian_Motions} 
The filtration $\{\Fth^{t} \}_{t=0}^{T}$ is generated by the collection $\{ \bthk \}_{k=1}^{n}$.    \\
\item \label{Def:Tau:Setup:Tau_Layer_Approximation:Item:Measurability} 
The pair $\left (\rth,\uth \right ): \hat{\Omega}_\tau \times [0,T] \to  L^{\beta}_{x} \times X_{n}$ is progressively measurable with respect to $\{\Fth^{t} \}_{t=0}^{T}$ with $\pth$ a.s. continuous sample paths. \\
\item \label{Def:Tau:Setup:Tau_Layer_Approximation:Item:Continuity_Equation} 
For all $\phi \in C^{\infty}(D)$ and $t \in [0,T]$, the following equality holds $\pth$ a.s.
\begin{equation} \label{Eq:Tau:Setup:Continuity_Equation}
	\int_{D} \rth(t)\phi dx = \int_{D} \rho_{0,\delta} \phi + \int_{0}^{t}\int_{D} 2\htd(s)[ \rth \uth \cdot \nabla \phi + \epsilon \rth \Delta \phi] dxds .
\end{equation} 
\item \label{Def:Tau:Setup:Tau_Layer_Approximation:Item:Momentum_Equation}
For all $\phi \in X_{n}$ and $t \in [0,T]$, the following equality holds $\pth$ a.s.
\begin{equation} \label{Eq:Tau:Setup:Momentum_Equation}
\begin{split}
&\int_{D} \rth \uth(t) \cdot \phi dx  = \int_{D} m_{0,\delta} \cdot \phi + \int_{0}^{t}\int_{D}2\htd [\rth \uth \tensor \uth-2\mu\nabla \uth-\lambda \Div \uth I] : \nabla \phi \\
&+ \int_{0}^{t} 2 \htd \big [ (\rth^{\gamma}+\delta \rth^{\beta}) \Div \phi - \epsilon \nabla \uth \nabla \rth \cdot \phi \big ] dxds\\
&+ \sum_{k=1}^{n} \int_{0}^{t}\int_{D} \sqrt{2} \hts \rth \sigma_{k,\tau,n,\delta}(\rth,\rth\uth) \cdot \phi dx d \bthk(s) .
\end{split}
\end{equation}
\item \label{Def:Tau:Setup:Tau_Layer_Approximation:Item:Energy_Identity}
For all $t \in [0,T]$, the following approximate energy identity holds $\pth$ a.s.
\begin{equation} \label{Eq:Tau:Setup:Energy_Identity}
\begin{split}
&\int_{D}\left [ \frac{1}{2}\rth |\uth|^{2}(t) + \frac{1}{\gamma -1} \rth^{\gamma}(t) + \frac{\delta}{\beta -1} \rth^{\beta}(t) \right ]dx \\
&+ \int_{0}^{t}\int_{D} 2\htd \bigg [ 2\mu|\nabla \uth|^{2}+\lambda (\Div \uth )^{2} + \epsilon(\gamma \rth^{\gamma-2}+ \delta \beta \rth^{\beta -2})|\nabla \rth|^{2} \bigg ] dxds \\
&= \sum_{k=1}^{n} \int_{0}^{t}\int_{D}\sqrt{2}\hts \rth \uth\cdot  \sigma_{k,\tau,n,\delta}(\rth,\rth\uth)dxd\bthk(s) \\ 
&+ \sum_{k=1}^{n}  \int_{0}^{t}\int_{D} 2\hts \rth |\sigma_{k,\tau,n,\delta}(\rth,\rth\uth)|^{2} dxdt +E_{n}(0) \\
& \sup_{n} E_{n}(0) \leq E_{\delta}(0) = \frac{1}{2}\int_{D} \left [ \frac{|m_{0,\delta}|^{2}}{\rho_{0,\delta}} + \frac{1}{\gamma -1}\rho_{0,\delta}^{\gamma} \right ]dx .
\end{split}
\end{equation}
\end{enumerate}
\end{Def}
In the definition above, we have replaced the initial data $(\rho_{0},m_{0})$ by the pair $(\rho_{0,\delta},m_{0,\delta})$ which satisfy
\begin{Hyp}\label{Hyp:regularized_Data} For each $\delta >0$, $\rho_{0,\delta} \in C^{\infty}(D)$ and
\begin{equation} \label{Eq:Tau:Neumann_Regularized_Data_Plus_Growth_Bounds}
\nabla \rho_{0,\delta} \cdot \textbf{n} \mid_{\partial D} = 0\quad \quad 0< \delta < \rho_{0,\delta}(x) \leq \delta^{-\frac{1}{2\beta}} \quad \text{for} \quad x \in D .
\end{equation}
The sequence $\{ \rho_{0,\delta}\}_{\delta >0}$ converges strongly to $\rho_{0}$ in the sense that
\begin{equation} \label{Eq:Tau:Convergence_of_Regularized_Data}
\lim_{\delta \to 0} |\rho_{0,\delta} - \rho_{0}|_{L^{\gamma}(D)} + | \{x \in D \mid \rho_{0,\delta} < \rho_{0} \} | = 0.
\end{equation} 
The regularized initial momentum $\{ m_{0,\delta} \}_{\delta >0}$ are defined by the relation
\begin{equation} \label{Eq:Tau:Regularized_Momentum}
m_{0,\delta}=
\begin{cases}
m_{0} &\quad \text{if} \quad \rho_{0,\delta}(x) \geq \rho_{0}(x)\\
0 &\quad \text{if} \quad \rho_{0,\delta}(x) < \rho_{0}(x) . 
\end{cases}
\end{equation} 
\end{Hyp}
Now we build up to the proof of Theorem \ref{Thm:Tau:Setup:Tau_Layer_Existence}, establishing some preliminary results in Sections \ref{SubSec:Tau:1:Det} and \ref{SubSec:Tau:2:Stoch} below then proving the Theorem in Section \ref{SubSec:Tau:2:Proof}.  The essence of the proof is an inductive construction of a $\tau$ layer approximation.  On the time interval $(0,\tau]$, the noise does not contribute to the weak form, so a pathwise application of a deterministic result from Section \ref{SubSec:Tau:1:Det} will suffice.  On the interval $(\tau,2\tau]$, the density remains frozen and the noise is the sole contribution to the weak form, leading to a simple SPDE for the velocity.  In Section \ref{SubSec:Tau:2:Stoch}, we write this SPDE down and use the regularized multiplicative noise structure to reduce the existence to a classical fixed point problem.     

\subsection{Machinery from the Deterministic Theory} \label{SubSec:Tau:1:Det}
%!TEX root = main.tex
Given $\rho \in C^{+}_{x}$, define the operator $\mathcal{M}[\rho]: X_{n} \to X_{n}^{*} $ by the relation
\begin{align*}
 \langle \mathcal{M}[\rho]u,\eta \rangle = \int_{D} \rho(x) u(x) \cdot \eta(x) dx
\end{align*}
 for $u,\eta \in X_{n}$.  The proof of the lemma below is left to the reader. 
\begin{Lem} \label{Lem:Tau:Properties_of_M}
For each $\rho \in C^{+}_{x}$, $\mathcal{M}[\rho]: X_{n} \to X_{n}^{*}$ is an invertible(linear) mapping and 
$$ \big |\mathcal{M}^{-1}[\rho]  \big |_{\mathcal{L}(X_{n}^{*},X_{n})} \leq |\rho^{-1}|_{C_{x}}. $$
Moreover, for each $\rho_{1},\rho_{2} \in C^{+}_{x}$ the inverse satisfies  the following continuity estimate:
$$ \big |\mathcal{M}^{-1}[\rho_{1}]-\mathcal{M}^{-1}[\rho_{2}]  \big |_{\mathcal{L}(X_{n}^{*},X_{n})} \leq  |\rho_{1}^{-1}|_{C_{x}}  |\rho_{2}^{-1}|_{C_{x}} |\rho_{1}-\rho_{2}|_{L^{1}_{x}}. $$
\end{Lem}
 Let us also introduce the mapping $\mathcal{N} : C_{x} \times X_{n} \to X_{n}^{*}$ by the relation
\begin{align*}
 \langle \mathcal{N}[\rho,u],\eta \rangle = \int_{D} \big [ \rho u \tensor u-2\mu \nabla u+(\rho^{\gamma}+\delta \rho^{\beta} - \lambda \Div u)I \big ] : \nabla \eta -\epsilon \nabla u \nabla \rho \cdot \eta  \, \, dx .
\end{align*} 
\begin{Prop} \label{Prop:Tau:Determnistic_Existence_Result}
Let $s<T$ be initial and final times and  suppose initial data $(\rho_{in},u_{in}) \in C^{+}_{x} \times C^{1}_{x}$ are given.  Then there exists a unique pair $(\rho , u) \in C \left ( (s,T] ; C^{2}(D) \cap C^{+}(D) \right ) \times C \left ((s,T] ; X_{n} \right )$ satisfying the system
\begin{equation} \label{Eq:Tau:Deterministic_Evolution}
\begin{cases}
\partial_{t}\rho= 2\epsilon \Delta \rho - 2\Div(\rho u)    &\quad \text{in} \quad D \times (s,T] \\
\frac{\partial \rho}{\partial n} = 0  &\quad \text{in} \quad \partial D \times (s,T] \\ 
\rho(s) = \rho_{in}   &\quad \text{in} \quad D \\
u(t)=\mathcal{M}^{-1}  [\rho(t)] \circ \left ( m_{in}^{*} +\int_{s}^{t} 2 \, \mathcal{N} \big [u(r),\rho(r)\big ] dr \right ) &\quad \text{in} \quad (s,T]  \\
\end{cases}
\end{equation}
where $m_{in}^{*} \in X_{n}^{*}$ is defined for $\eta \in X_{n}$ via the relation
$$
 \langle m_{in}^{*},\eta \rangle =\int_{D}\rho_{in}u_{in} \cdot \eta .
$$
If $u_{in} \in X_{n}$, then $u(s)=u_{in}$.  Moreover, the solution map $\mathcal{S}: C^{+}(D) \times X_{n} \to C \left ( (s,T] ; C^{+}(D) \right ) \times C \left ((s,T] ; X_{n} \right )$ is continuous.
\end{Prop}
\begin{proof}
The proof uses a  straightforward combination of the contraction mapping principle, the $\LW{p}{2}{p}$ estimates for the parabolic Neumann problem, the maximum principle, and basic a priori bounds for the system \eqref{Eq:Tau:Deterministic_Evolution}.  A similar result is established in \cite{feireisl2004dynamics} using the Schauder fixed point theorem, though the uniqueness is not proven.  The details are left to the reader.
\end{proof}

\subsection{A Classical SPDE Result}\label{SubSec:Tau:2:Stoch}
%!TEX root = main.tex
Let $(\Omega, \F, \p, \{ \F^{t}\}_{t=0}^{T} , \{ \beta_{k}\}_{k=1}^{n} )$ be a stochastic basis such the filtration $\{ \F^{t}\}_{t=0}^{T}$ is generated by the collection $\{ \beta_{k}\}_{k=1}^{n}$.  Suppose $s<T$ are two times and $\rho$ is an $\F_{s}$ measurable, $C^{+}_{x}$ valued random variable.  An $X_{n}$ valued, $\{ \F^{t}\}_{t=s}^{T}$ progressively measurable process $u$ is defined to be a solution to the SPDE
\begin{equation} \label{Eq:Tau:SPDE_for_the_Velocity}
\begin{cases}
\partial_{t}u = \sum_{k=1}^{n} \sigma_{k,\tau,n,\delta}(\rho, \rho u) \dot{\beta}_{k}(t) & \quad \text{in} \quad (s,T] \times D \\
u(s)=u_{in} & \quad \text{in} \quad D
\end{cases}
\end{equation}
provided that for all $t \in [s,T]$ the following equality(in $X_{n}$) holds $\p$ a.s. 
\begin{equation} \label{Eq:Tau:SPDE_for_the_Velocity:Integral_Form}
u(t)=u_{in} + \sum_{k=1}^{n}\int_{s}^{t}\sigma_{k,\tau,n,\delta}\left (\rho,\rho u(r) \right )d\beta_{k}(r).
\end{equation}
\begin{Prop} \label{Prop:Tau:SPDE_Existence_for_the_Velocity}
Given a stochastic basis and a density $\rho$ as above, there exists a unique solution $u \in L^{2} \big (\Omega ; C \left ([s,T] ; X_{n} \right ) \big )$ to \eqref{Eq:Tau:SPDE_for_the_Velocity} in the sense of \eqref{Eq:Tau:SPDE_for_the_Velocity:Integral_Form}. 
\end{Prop}
\begin{proof}
Define a ``random diffusion coefficient'' $G : X_{n} \times \Omega \to X_{n}$ via
\begin{equation}
G(u,\omega) = \sum_{k=1}^{n} \sigma_{k,\tau,n,\delta}(\rho(\omega),\rho(\omega)u).
\end{equation}
Note that Hypothesis \ref{Hyp:Projections} gives $L^{2}_{x}$ stability of the projections, which may be combined with the continuity Hypothesis \ref{Hyp:lip} to check that $G$ is lipschitz in $u$, uniformly in $\omega$.  Indeed, for $u,v \in X_{n}$
\begin{align*}
&|G(u,\omega)-G(v,\omega)|_{X_{n}} \leqs (\int_{D}|[(\rho(\omega) \wedge \frac{1}{\tau})(u-v)] * \eta_{\delta}|^{2}dx)^{\frac{1}{2}} \\ 
&\leqs |\rho(\omega) \wedge \frac{1}{\tau}(u-v)|_{L^{2}_{x}}
\leqs \frac{1}{\tau} |u-v|_{L^{2}_{x}} \leqs |u-v|_{X_{n}}.
\end{align*}
Here we have used that $X_{n}$ is a finite dimensional space.  Hence, the Proposition may be established in the classical way via the contraction mapping principle.
\end{proof}

\subsection{Proof of Theorem \ref{Thm:Tau:Setup:Tau_Layer_Existence} }\label{SubSec:Tau:2:Proof}
%!TEX root = main.tex
We are now prepared to establish an existence theorem for the lowest level of our scheme. 

\begin{proof} Let $(\hat{\Omega},\ph,\{\hat{\beta}_{k}\}_{k=1}^{n}, \{\hat{\mathcal{F}}_{t} \}_{t=0}^{T})$ be a stochastic basis and assume the filtration $\{\hat{\mathcal{F}}_{t} \}_{t=0}^{T}$ is generated by the Brownian motions $\{\hat{\beta}_{k}\}_{k=1}^{n}$. We will define the solution inductively.  Namely, suppose that $(\rth,\uth)$ have been constructed to satisfy the continuity equation \eqref{Eq:Tau:Setup:Continuity_Equation}, the momentum equation \eqref{Eq:Tau:Setup:Momentum_Equation}, and the energy identity \eqref{Eq:Tau:Setup:Energy_Identity} on the time interval $[0,t_{2j}]$.  To extend the solution to the interval $(t_{2j},t_{2j+1}]$, apply Proposition \ref{Prop:Tau:Determnistic_Existence_Result} to find a unique pair $(\rho,u)$ satisfying:
\begin{equation}
\begin{cases}
\partial_{t}\rho + 2\Div(\rho u) - 2\epsilon \Delta \rho = 0 &\quad \text{in} \quad D \times (t_{2j},t_{2j+1}] \\
\frac{\partial \rho}{\partial n} = 0  &\quad \text{in} \quad \partial D \times (t_{2j},t_{2j+1}] \\ 
\rho(t_{2j}) = \rth(t_{2j})   &\quad \text{in} \quad D \\
u(t)=\mathcal{M}^{-1}  [\rho(t)] \circ \left ( \rth \uth (t_{2j})^{*} +\int_{t_{2j}}^{t} 2 \, \mathcal{N} \big [u(s),\rho(s)\big ] ds \right ) &\quad \text{in} \quad D \times (t_{2j},t_{2j+1}]  \\
u(t_{2j}) = \uth(t_{2j})   &\quad \text{in} \quad D . \\
\end{cases}
\end{equation} 
To extend the solution to the interval $(t_{2j+1},t_{2j+2}]$ we appeal to Proposition \ref{Prop:Tau:SPDE_Existence_for_the_Velocity} to find a unique pair $(\rho,u)$ satisfying 
\begin{equation}
\begin{cases}
\partial_{t}\rho = 0 &\quad \text{in} \quad D \times (t_{2j+1},t_{2j+2}] \\
\partial_{t}u = \sqrt{2}\sum_{k=1}^{n}\sigma_{k,\tau,n,\delta}(\rho,\rho u)\dot{\beta}_{k} & \quad \text{in} \quad D \times (t_{2j+1},t_{2j+2}]. \\
\end{cases}
\end{equation}
Using the Ito Formula and the inductive hypothesis, one may check that \eqref{Eq:Tau:Setup:Continuity_Equation}-\eqref{Eq:Tau:Setup:Energy_Identity} continue to hold for $t \in [t_{2j},t_{2j+2}]$.  To prove the uniform bounds, begin with the energy identity \eqref{Eq:Tau:Setup:Energy_Identity}.  One can estimate the stochastic integral terms and the Ito correction with the same manipulations as in the formal proof provided in Section \ref{SubSec:P:5:Energy_Estimates}.  The only additional detail is to note that Hypothesis \ref{Hyp:Projections} and the uniform boundedness principle imply the projection operators $\Pi_{n}$ are bounded(uniformly in $n$) as linear operators from $L^{\frac{2\gamma}{\gamma-1}}_{x}$ to $L^{\frac{2\gamma}{\gamma-1}}_{x}$. The desired measurability, part \ref{Def:Tau:Setup:Tau_Layer_Approximation:Item:Measurability} of Definition \ref{Def:Tau:Setup:Tau_Layer_Approximation}, follows from the continuity of the solution map to the deterministic problem(guaranteed by Proposition \ref{Prop:Tau:Determnistic_Existence_Result}), together with the fact the that we obtain a stochastically strong(measurable with respect to the same filtration as the Brownian motions) solution during each time interval where the stochastic forcing evolves. 
\end{proof}

\section{$n$ Layer Existence}\label{Section:N}
%!TEX root = main.tex
%!TEX root = main.tex
In this section, we apply Theorem \ref{Thm:Tau:Setup:Tau_Layer_Existence} to build the next layer of the approximating scheme, the $n$ layer.  Our goal is to establish the following: 
\begin{Thm}  \label{Thm:N:Setup:N_Layer_Existence}
There exists a sequence $\{(\rnh,\unh)\}_{n=1}^{\infty}$ of $n$ layer approximations(in the sense of Definition \ref{Def:N:Setup:N_Layer_Approximation} below), relative to a collection of stochastic bases $\{ (\Onh, \Fnh, \{ \Fnh^{t}\}_{t=0}^{T}, \pnh, \{ \bnhk\}_{k=1}^{n}) \}_{n=1}^{\infty}$, such that for all $p\geq 1$
\begin{equation} \label{Eq:N:Setup:Uniform_Bounds}
\begin{split}
&\sup_{n \geq 1} \E^{\pnh} \bigg [ |\sqrt{\rnh}\unh |_{\LL{\infty}{2}}^{2p} + |\rnh|_{\LL{\infty}{\gamma}}^{\gamma p}+ |\unh|_{\LHZ{2}{1}}^{2p} \bigg ] < \infty \\
&\sup_{n \geq 1} \E^{\pnh} \bigg [ |\delta^{\frac{1}{\beta}}\rnh|_{\LL{\infty}{\beta}}^{\beta p} + | \epsilon^{\frac{1}{2}} \nabla(\rnh^{\frac{\gamma}{2}} + \delta^{\frac{1}{2}}\rnh^{\frac{\beta}{2}})|_{\LL{2}{2}}^{2p} \bigg ] < \infty. \\
\end{split}
\end{equation}
\end{Thm}

Let us introduce the $n$ layer regularization of the multiplicative noise structure. Define an operator $\sigma_{k,n,\delta} : L^{1}_{x} \times [L^{1}_{x}]^{d} \to X_{n}$ via the relation
\begin{equation}
\sigma_{k,n,\delta}(\rho,m) = \Pi_{n} \circ \sigma_{k}(\rho * \eta_{\delta}(\cdot), m * \eta_{\delta}(\cdot),\cdot).
\end{equation}

\begin{Def} \label{Def:N:Setup:N_Layer_Approximation}
A pair $(\rnh,\unh)$ is defined to be an $n$ layer approximation to \eqref{Eq:I:SCNS_System} provided there exists a stochastic basis $(\hat{\Omega}_{n}, \hat{\mathcal{F}}_{n}, \{\hat{\mathcal{F}}_{n}^{t} \}_{t=0}^{T} , \hat{\p}_{n}, \{ \bnhk\}_{k=1}^{n})$ such that 
\begin{enumerate}
\item \label{Def:N:Setup:N_Layer_Approximation:Item:Measurability}
The pair $\left (\rnh,\rnh \unh \right ): \Onh \times [0,T] \to L^{\beta}_{x} \times L^{\beta}_{x}$ is an $\{\Fnh^{t} \}_{t=0}^{T}$ progressively measurable stochastic process with $\pnh$ a.s. continuous sample paths.  The velocity $\unh \in L^{2} \left (\Onh \times [0,T] ; H^{1}_{0,x} \right )$ is an equivalence class of $\{\Fnh^{t} \}_{t=0}^{T}$ progressively measurable $H^{1}_{x}$ valued processes.  \\
\item 
For all $\phi \in C^{\infty}(D)$ and all times $t \in [0,T]$ the following equality holds $\pnh$ a.s.
\begin{equation} \label{Eq:N:Setup:Continuity_Equation}
\int_{D} \rnh(t)\phi dx = \int_{D} \rho_{0,\delta} \phi + \int_{0}^{t}\int_{D} [ \rnh \unh \cdot \nabla \phi + \epsilon \rnh \Delta \phi] dxds .
\end{equation} 
\item \label{Def:N:Weak_Solutions:Item:Momentum_Equation}
For all $\phi \in X_{n}$ and all times $t \in [0,T]$ the following equality holds $\pnh$ a.s.
\begin{equation} \label{Eq:N:Setup:Momentum_Equation}
\begin{split}
&\int_{D} \rnh \unh(t) \cdot \phi dx  = \int_{D} m_{0,\delta}\cdot \phi + \int_{0}^{t}\int_{D} [\rnh \unh \tensor \unh-2\mu \nabla \unh-\lambda \Div \unh I] : \nabla \phi \\ 
&+\int_{0}^{t}\int_{D} \big [ (\rnh^{\gamma}+\delta \rnh^{\beta}) \Div \phi-\epsilon \nabla \unh \nabla \rnh \cdot \phi  \big ] dxds \\
&+ \sum_{k=1}^{n} \int_{0}^{t}\int_{D} \rnh \sigma_{k,n,\delta}(\rnh,\rnh \unh) \cdot \phi dx d \bnhk(s).
\end{split}
\end{equation}
\item  \label{Def:N:Setup:Energy_Identity} For all $t \in [0,T]$, the following approximate energy identity holds $\pnh$ a.s. 
\begin{equation}  \label{Eq:N:Setup:Energy_Identity}
\begin{split}
&\int_{D} [ \frac{1}{2}\rnh |\unh|^{2}(t) + \frac{1}{\gamma -1} \rnh^{\gamma}(t) + \frac{\delta}{\beta -1} \rnh^{\beta}(t)  ]dx \\   
&+ \int_{0}^{t}\int_{D} \big [ 2 \mu|\nabla_x \unh|^{2} + \lambda \Div \unh^{2}+\epsilon(\gamma \rnh^{\gamma-2} + \delta \beta \rnh^{\beta -2})|\nabla \rnh|^{2} \big ] dxds  \\
&= \sum_{k=1}^{n} \int_{0}^{t}\int_{D}\rnh \unh \cdot \sigma_{k,n,\delta}(\rnh,\rnh \unh)dxd\bthk(s) \\ 
&+ \frac{1}{2}\sum_{k=1}^{n}\int_{0}^{t}\int_{D} \rth(s)|\sigma_{k,n,\delta}(\rnh,\rnh\unh)|^{2} dxds + E_{n}(0). \\
& \sup_{n} E_{n}(0) \leq E_{\delta}(0) = \frac{1}{2}\int_{D} \left [\frac{|m_{0,\delta}|^{2}}{\rho_{0,\delta}} + \frac{1}{\gamma -1}\rho_{0,\delta}^{\gamma} \right ]dx .
\end{split}
\end{equation}
\end{enumerate}
\end{Def}

For each $n$ fixed we apply Theorem \ref{Thm:Tau:Setup:Tau_Layer_Existence} to obtain a sequence of $\tau$ layer approximations $\{(\hat{\rho}_{\tau,n},\hat{u}_{\tau,n})\}_{\tau >0}$.  In Section \ref{SubSec:N:1:Compactness}, we prove a compactness result for this sequence and extract a candidate $n$ layer approximation $(\rnh,\unh)$ built on a convenient choice of probability space $(\Onh,\Fnh,\pnh)$.  In Section \ref{SubSec:N:2:Identification}, we use the compactness result to verify $(\rnh,\unh)$ is an $n$ layer approximation in the sense of Definition \ref{Def:N:Setup:N_Layer_Approximation}. 

\subsection{$\tau \to 0$ Compactness Step} \label{SubSec:N:1:Compactness}
%!TEX root = main.tex
Now we proceed to our compactness step:
\begin{Prop} \label{Prop:N:Skorohod}
Let $\widehat{W} = \{ \hat{\beta}^{k}\}_{k=1}^{n}$ denote the collection of Brownian motions in the stochastic basis supporting our sequence $\{(\rth,\uth)\}_{\tau >0}$.  

There exists a probability space $(\Onh,\Fnh,\pnh )$, a limit point $(\rnh,\unh)$, and a sequence of ``recovery" maps $\{\hat{T}_{\tau}\}_{\tau > 0}$ 
$$ \hat{T}_{\tau}:(\Onh,\Fnh,\pnh) \to (\hat{\Omega},\Fh,\ph)$$
with the properties listed below:
\begin{enumerate}
\item	\label{Prop:N:Skorohod:Item:Recovery_Maps} The measure $\ph$ can be recovered by pushing forward $\pnh$ with $\hat{T}_{\tau}$. \\
\item \label{Prop:N:Skorohod:Item:Preserving_the_Equation}
The new sequence $\{(\rt,\ut)\}_{\tau > 0}$ defined by  $ \left(\rt,\ut \right )=\left (\rth ,\uth \right ) \circ \hat{T}_{\tau}$ constitutes a $\tau$ layer approximation relative to the stochastic basis 
$(\Onh,\Fnh,\pnh,\{ \mathcal{F}^{t}_{\tau}\}_{t=0}^{T},W_{\tau} )$, where
$$ W_{\tau} := \widehat{W} \circ T_{\tau} \quad \quad \quad \F^{t}_{\tau} := \sigma \left (r_{t}\rt, r_{t}\ut , r_{t}(\rt \ut), r_{t}W_{\tau} \right ).$$  Moreover, the initial data are recovered in the sense that $(\rnh(0),\unh(0))=(\rho_{0,\delta},\mathcal{M}^{-1}[\rho_{0,\delta}]m_{0,\delta})$.
\item \label{Prop:N:Skorohod:Item:Preserving_the_Bounds} The following uniform bounds hold for all $p \geq 1$
\begin{equation} \label{Eq:N:Skorohod:Preserved_Tau_Layer_Bounds}
\begin{split} 
&\sup_{\tau > 0} \E^{\pnh} \bigg [ |\sqrt{\rt}\ut |_{\LL{\infty}{2}}^{2p} + |\rt|_{\LL{\infty}{\gamma}}^{\gamma p}+|\htd\ut|_{\LHZ{2}{1}}^{2p} \bigg ] < \infty \\ 
& \sup_{\tau >0}\E^{\pnh} \bigg [|\delta^{\frac{1}{\beta}}\rt|_{\LL{\infty}{\beta}}^{\beta p} + |\htd \epsilon^{\frac{1}{2}} \nabla(\rt^{\frac{\gamma}{2}} + \delta^{\frac{1}{2}}\rt^{\frac{\beta}{2}})|_{\LL{2}{2}}^{2p} \bigg ] < \infty .
\end{split}
\end{equation}
\item \label{Prop:N:Skorohod:Item:Pointwise_Convergences}As $\tau \to 0$, the following convergences hold $\pnh$ a.s. 
\begin{align}
\label{Eq:N:Skorohod:Pointwise_Convergence:Density}\rt &\to \rnh \quad &\text{in}   \quad   &C_{t}(L^{\beta}_{x}) \cap \LW{\beta}{1}{\beta}\\
\label{Eq:N:Skorohod:Pointwise_Convergence:Velocity}\ut &\to \unh \quad &\text{in}   \quad  &C_{t}(X_{n}) \\
\label{Eq:N:Skorohod:Pointwise_Convergence:Brownians}W_{\tau}& \to \widehat{W}_{n} \quad & \text{in} \quad  &\big [C_{t}(X_{n}) \big ]^{n} .
\end{align}
\end{enumerate}
\end{Prop}
%%%%%%%%%%%%%%%%%
%%%%%%%%%%%%%%%%%
%%%%%%%%%%%%%%%%%
%%%%%%%%%%%%%%%%%
The proof of Proposition \ref{Prop:N:Skorohod} begins with a tightness lemma.
\begin{Lem} \label{Lem:N:Tightness}
The sequence of induced measures $\{ \hat{\p} \circ (\rth,\uth,\widehat{W})^{-1} \}_{\tau > 0}$ are tight on $C_{t}(L^{\beta}_{x}) \cap \LW{\beta}{1}{\beta} \times \big [ C_{t}(X_{n}) \big ] \times [C_{t} ]^{n}$.
\end{Lem}
\begin{proof}  Note that it suffices to show the tightness of each component separately. Tightness of $\hat{\p} \circ \widehat{W}^{-1}$ is an immediate consequence of Arzela-Ascoli and the usual $L^{2}(\hat{\Omega}; C^{1/3}_{t})$ bound on each one dimensional Brownian motion.  Next we will check that
\begin{equation} \label{Eq:N:Tightness:Velocity_Bounds}
\lim_{M \to \infty}\sup_{\tau > 0} \ph \big ( |\uth|_{C_{t}(X_{n})} > M\big ) = 0.
\end{equation}
Multiplying and dividing by the density gives the pathwise upper bound 
$$
\sup_{t \in [0,T]}\int_{D} |\uth|^{2} \leq |\rth^{-1}|_{\LLs{\infty}}
\sup_{t \in [0,T]}\int_{D} \rth|\uth|^{2}.
$$
Also, note that if $X,Y : \hat{\Omega} \to \R$ are two positive random variables, then
\begin{equation} \label{Eq:N:Tightness:Product_Bounds}
\ph(XY>M) \leq \ph(X > \sqrt{M})+ \ph(Y > \sqrt{M}).
\end{equation}
Combining these observations yields
\begin{align*}
%\ph \big ( |\uth|_{C_{t}(X_{n})} > M\big ) \leq 
&\ph \big ( \sup_{t \in [0,T]} \int_{D} |\uth|^{2} > M\big ) \leq \ph \big ( \sup_{t \in [0,T]} \int_{D} \rth |\uth|^{2} > \sqrt{M}\big ) +
\ph \big ( |\rth^{-1}|_{\LLs{\infty}}> \sqrt{M}\big ).  
\end{align*}
Using the $L^{2}(\hat{\Omega})$ bounds on kinetic energy implied by \eqref{Eq:Tau:Setup:Uniform_Bounds},  we can choose $M$ to make the first probability small, uniformly in $\tau$. 

To treat the second term, recall the splitting scheme from Section \ref{Section:Tau} defining the evolution of $\rth$. On time intervals $(t_{2j},t_{2j+1}]$,  $\rth$ solves a parabolic equation with drift $\uth$ and remains constant on the intervals $(t_{2j+1},t_{2j+2}]$.  Iteratively apply the maximum principle then use the equivalence of the $X_{n}$ and $C_{x}^{1}$ norms.  This controls the second probability from above by
\begin{align*}
&\ph \left ( |\rho_{0,\delta}^{-1}|_{L^{\infty}_{x}} \text{exp}( \int_{0}^{T}\htd(t)|\Div \uth(t)|_{L^{\infty}_{x}}dt > \sqrt{M}) \right ) \\
&\leq \ph \bigg (  \int_{0}^{T}|\htd(t)\uth(t)|_{X_{n}})dt> C_{n}^{-1}\text{log}(M |\rho_{0,\delta}^{-1}|_{L^{\infty}_{x}}^{-1})\bigg ).  
\end{align*}
Applying a H\"older(in time) and the $L^{2}_{t}(X_{n})$ bounds on the velocity implied by \eqref{Eq:Tau:Setup:Uniform_Bounds}, we can make this second probability uniformly arbitrarily small also.  Hence, \eqref{Eq:N:Tightness:Velocity_Bounds} is established.

We can now bootstrap \eqref{Eq:N:Tightness:Velocity_Bounds} and prove the tightness of $\{\hat{\p} \circ \rth^{-1} \}_{\tau >0}$ on $\CL{\beta} \cap \LW{\beta}{1}{\beta}$. To this end, we use Lemma \ref{Lem:Appendix:Mixed_Parabolic_Hyperbolic_Estimate} from the appendix.  For simplicity, we will omit dependence of the estimate on the initial density, since it has been smoothed out already.  Start by defining the exponent $q$ via the interpolation condition $\frac{1}{q}=\frac{1}{2\beta} + \frac{1}{2(\beta+1)}$ to obtain the following estimate: 
\begin{align*} 
&|\partial_{t} \rth|_{\LLs{\beta}} + |\rth|_{\LW{\beta}{2}{\beta}} \leqs |\htd \Div(\rth\uth)|_{\LLs{q}} \leqs |\uth |_{C_{t}(X_{n})} |\htd \rth|_{\LW{q}{1}{q}} \\
&\leqs |\uth |_{C_{t}(X_{n})} |\htd \rth|_{\LLs{\beta+1}}^{\frac{1}{2}}|\rth|_{\LW{\beta}{2}{\beta}}^{\frac{1}{2}}.
\end{align*}
Applying Cauchy's inequality, we may close the estimate then interpolate once more to obtain
\begin{equation}\label{Eq:N:Tightness:Parabolic_Est1}
\begin{split}
&\nonumber |\partial_{t} \rth|_{\LLs{\beta}} + |\rth|_{\LW{\beta}{2}{\beta}} \leqs  |\uth |_{C_{t}(X_{n})}^{2} |\htd \rth|_{\LLs{\beta+1}}  \leqs |\uth |_{C_{t}(X_{n})}^{2} |\htd \rth^{\beta}|_{\LLs{\frac{\beta+1}{\beta}}}^{\frac{1}{\beta}} \\ \nonumber
&\leqs |\uth |_{C_{t}(X_{n})}^{2}|\htd \rth^{\beta}|_{\LL{\infty}{1}}^{\frac{\theta}{\beta}} |\htd \rth^{\beta}|_{\LL{1}{\frac{d}{d-2}}}^{\frac{1-\theta}{\beta}} \\
&\leqs |\uth |_{C_{t}(X_{n})}^{2}|\htd \rth|_{\LL{\infty}{\beta}}^{\theta} |\htd \rth^{\beta/2}|_{\LL{2}{\frac{2d}{d-2}}}^{\frac{1-\theta}{2\beta}} \\ \    
&\leqs |\uth |_{C_{t}(X_{n})}^{2}|\htd \rth|_{\LL{\infty}{\beta}}^{\theta} |\htd \rth^{\beta/2}|_{\LW{2}{1}{2}}^{\frac{1-\theta}{2\beta}}. 
\end{split}
\end{equation}
Note that $\theta$ is defined by the relation $\frac{\beta}{\beta+1}=\theta(1-\frac{2}{d}) + (1-\theta)$.  Bootstrapping this estimate once yields for all $r<\beta$
\begin{equation}\label{Eq:N:Tightness:Parabolic_Est2}
 |\partial_{t} \nabla \rth|_{\LLs{r}} + |\nabla \rth|_{\LW{r}{2}{r}} \leqs |\nabla \Div(\rth \uth)|_{\LLs{\beta}} \leqs |\uth|_{C_{t}(X_{n})} |\rth|_{\LW{\beta}{2}{\beta}}.
\end{equation}
Choosing $r$ large enough to ensure the embedding $W^{1,r}_{x} \hookrightarrow L^{\beta}_{x}$ is compact, we may conclude from Arzela-Ascoli and Aubin-Lions the following set is compact in $\CL{\beta} \cap \LW{\beta}{1}{\beta}$
\begin{align*}
\{ f \in \LW{\beta}{1}{\beta} \mid |\partial_{t}f|_{\LL{\beta}{\beta}} +|\partial_{t}\nabla f|_{\LL{r}{r}}+ |f|_{\LW{\beta}{2}{\beta}} \leq M \}.
\end{align*}
Combining \eqref{Eq:N:Tightness:Velocity_Bounds} together with the uniform estimates \eqref{Eq:Tau:Setup:Uniform_Bounds} to control the RHS of \eqref{Eq:N:Tightness:Parabolic_Est1} and \eqref{Eq:N:Tightness:Parabolic_Est2}, we obtain the desired tightness by Chebyshev for $M$ large enough.   Our final step is to show 
\begin{equation}
\lim_{M \to \infty} \sup_{\tau} \hat{\p} \left ([\uth]_{C^{1/3}_{t}(X_{n})} \geq M \right ) = 0.
\end{equation}
 Note that the brackets indicate we are considering the H\"older seminorm, since the uniform norm has already been handled above.  Recalling the operators introduced in Section \ref{SubSec:Tau:1:Det}, define the $X_{n}^{*}$ valued processes $\{ I_{\tau}^{D}(t) \}_{t \geq 0}$,$\{ I_{\tau}^{S}(t) \}_{t \geq 0}$ via
 \begin{align*}
 &I_{\tau}^{D}(t) =  \int_{0}^{t} 2\htd(r) \mathcal{N} \big [\uth(r),\rth(r)\big ] dr. \\
 &\langle I_{\tau}^{S}(t),\phi \rangle =\int_{0}^{t} \int_{D} \sqrt{2}\hts(r) \rth(r) \sigma_{k,\tau,n,\delta}(\rth(r),\rth \uth(r)) \cdot \phi dx d \hat{\beta}_{k}(r). 
 \end{align*}
for $\phi \in X_{n}$.  For each $s<t$ the momentum equation yields
\begin{equation}
\begin{split}
&\uth(t)-\uth(s) =\mathcal{M}^{-1}[\rth(t)] \left ( I_{\tau}^{D}(t)-I_{\tau}^{D}(s)+I_{\tau}^{S}(t)-I_{\tau}^{S}(s) \right ) \\
&+ \left (\mathcal{M}^{-1}[\rth(t)]-\mathcal{M}^{-1}[\rth(s)] \right ) \circ \left ( m_{0,\delta}^{*} +I_{\tau}^{D}(s) + I_{\tau}^{S}(s) \right ).
\end{split} 
\end{equation}
 Using Lemma \ref{Lem:Tau:Properties_of_M} and the maximum principle, we obtain the $\hat{\p}$ a.s. estimate
\begin{equation} \label{Eq:N:Tightness:Semi_Norm_Bound}
\begin{split}
&[\ut]_{ C^{1/3}_{t}(X_{n})} \leq e^{C_{n} T |\uth|_{C_{t}(X_{n})}} \left [ |m_{0,\delta}^{*}|_{X_{n}^{*}}+ [I_{\tau}^{D}]_{C^{1/3}_{t}(X_{n}^{*})} +[I_{\tau}^{S}]_{C^{1/3}_{t}(X_{n}^{*})} \right ] \\
&+ e^{C_{n} T |\uth|_{C_{t}(X_{n})}}\left [ |\rth|_{C^{1/3}_{t}(L^{1}_{x})} \left ( |I_{\tau}^{D}|_{C_{t}(X_{n}^{*})} +|I_{\tau}^{S}|_{C_{t}(X_{n}^{*})} \right )   \right ].
\end{split}
\end{equation}
In view of the estimates for the density above, this reduces the problem to controlling the probability that $I_{\tau}^{S}$ and $I_{\tau}^{D}$ have a large Holder norm.  To estimate $|I_{\tau}^{D}|_{C^{1/3}_{t}(X_{n}^{*})}$, note first that for $\rho \in L^{1}_{x}$ and $u \in X_{n}$
\begin{align*}
|\mathcal{N}(\rho, u)|_{X_{n}^{*}} \leq C_{n} \left (|\rho|_{L^{1}_{x}}|u|_{X_{n}}^{2}+|u|_{X_{n}}+|\rho|_{L^{\gamma}}^{\gamma}+\delta|\rho|_{L^{\beta}}^{\beta}+|u|_{X_{n}}|\rho|_{W^{1,1}_{x}} \right ).
\end{align*}
Applying H\"older's inequality in time yields for all $s<t$
\begin{align*}
&\int_{s}^{t}|\mathcal{N}(\rth(r), \uth(r))|_{X_{n}^{*}} \leqs (t-s)^{1-\frac{1}{\beta}}
\left [ |\rth|_{\LL{\infty}{1}}|\uth|^{2}_{C_{t}(X_{n})} + |\rth|_{\LL{\infty}{\gamma}}^{\gamma} + |\rth|_{\LL{\infty}{\beta}}^{\beta} \right ] \\
&+(t-s)^{1-\frac{1}{\beta}} \left [ |\uth|_{C_{t}(X_{n})}|\rth|_{\LW{\beta}{1}{1}} \right ].
\end{align*}
Certainly $1-\frac{1}{\beta} >\frac{1}{3}$.  Hence, we may combine \eqref{Eq:N:Tightness:Velocity_Bounds}, \eqref{Eq:N:Tightness:Parabolic_Est1}, and the uniform bounds \eqref{Eq:Tau:Setup:Uniform_Bounds} to obtain
\begin{equation} \label{Eq:N:Tightness:Det_Integrals} 
\lim_{M \to \infty} \sup_{\tau} \hat{\p} \left (|I_{\tau}^{D}|_{C^{1/3}_{t}(X_{n}^{*})}>M \right ) =0.
\end{equation}
To estimate $|I_{\tau}^{S}|_{C^{1/3}_{t}(X_{n}^{*})}$, fix a $\phi \in X_{n}$. Apply the BDG inequality, the boundedness of the projections, and the summability Hypotheses \ref{Hyp:color} for the noise coefficients to obtain for all $p \geq 2$
\begin{align*}
&\E^{\hat{\p}} \left [ |\int_{s}^{t} \int_{D} \sqrt{2}\hts(r) \rth(r) \sigma_{k,\tau,n,\delta}(\rth(r),\rth\uth(r)) \cdot \phi dx d \hat{\beta}_{k}(r)|^{p} \right ] \\
& \leqs \E^{\hat{\p}} \left [ |\int_{s}^{t}(\int_{D} \rth(r) \sigma_{k,\tau,n,\delta}(\rth(r),\rth\uth(r)) \cdot \phi dx)^{2}dr|^{\frac{p}{2}} \right ] \\
&\leqs |\sigma_{k}|_{L^{\frac{\gamma}{\gamma-1}}_{x}(L^{\infty}_{\rho,m})}^{p} \E^{\hat{\p}} \big [ |\int_{s}^{t}|\rth(r)|_{L^{\gamma}_{x}}^{2}|^{p/2} \big ] 
\leqs (t-s)^{p/2} \sup_{\tau} \E^{\hat{\p}} \big [ |\rth|_{\LL{\infty}{\gamma}}^{p}\big ] .
\end{align*}
This yields for all $s<\frac{1}{2}$, $p \geq 2$ and $\phi \in X_{n}$
\begin{align*}
\sup_{\tau > 0}\E^{\hat{\p}} \big [ | \langle I_{\tau}^{S},\phi \rangle|_{W^{s,p}_{t}}^{p} \big ] \leq \sup_{\tau >0}\E^{\hat{\p}}\big [ |\rth|_{\LL{\infty}{\gamma}}^{p}\big ] .
\end{align*}
Choose $s,p$ such that $W^{s,p}_{t} \hookrightarrow C^{1/3}_{t}$. Since $X_{n}$ is a finite dimensional, the uniform bounds \eqref{Eq:Tau:Setup:Uniform_Bounds} imply
\begin{equation} \label{Eq:N:Tightness:Stoch_Integrals}
\lim_{M \to \infty} \sup_{\tau} \hat{\p} \left (|I_{\tau}^{S}|_{C^{1/3}_{t}(X_{n}^{*})}>M \right ) =0.
\end{equation}
Starting with the identity \eqref{Eq:N:Tightness:Semi_Norm_Bound} and using \eqref{Eq:N:Tightness:Det_Integrals}, \eqref{Eq:N:Tightness:Stoch_Integrals} and some elementary estimates similar to \eqref{Eq:N:Tightness:Product_Bounds} give the tightness of the laws $ \{ \hat{\p} \circ \uth^{-1} \}_{\tau >0}$ on $C_{t}(X_{n})$ by Arzela-Ascoli.
 \end{proof}
 %%%%%%%%%%%%
 %%%%%%%%%%%%
 %%%%%%%%%%%%

Next we apply the tightness result above together with a version of the Skorohod Theorem \ref{Thm:Appendix:Jakubowksi_Skorohod} to complete our compactness step.
\begin{proof}[Proof of Proposition \ref{Prop:N:Skorohod}]: \quad  Let us define the sequence of random variables $ \{ \hat{X}_{\tau} \}_{\tau > 0}$ via the relation
$$
\hat{X}_{\tau} = (\rth,\uth,\widehat{W}).
$$
These random variables induce a tight sequence of laws on the metric space 
$$E=C_{t}(L^{\beta}_{x}) \cap \LW{\beta}{1}{\beta} \times \big [ C_{t}(X_{n}) \big ] \times [ C_{t}(X_{n}) \big ]^{n}$$
 by Lemma \ref{Lem:N:Tightness}.  In view of Remark \ref{Rem:Appendix:Examples_of_Jakubowksi_Spaces}, $(E,\tau)$ is a Jakubowski space and we may apply Theorem \ref{Thm:Appendix:Jakubowksi_Skorohod} to obtain a new probability space $(\Onh,\Fnh,\pnh)$, a sequence of recovery maps $\{\hat{T}_{\tau} \}_{\tau > 0}$ and a limiting random variable $\hat{X}_{n}=(\rnh, \unh, \hat{W}_{n})$ such that parts \ref{Prop:N:Skorohod:Item:Recovery_Maps} and \ref{Prop:N:Skorohod:Item:Pointwise_Convergences} of the Proposition hold. To check the uniform bounds \eqref{Eq:N:Skorohod:Preserved_Tau_Layer_Bounds} on $(\rt,\ut)$, simply use the recovery maps together with the bounds on the original probability space \eqref{Eq:Tau:Setup:Uniform_Bounds}. This is rigorous because the following functional is is continuous from $(E,\tau)$ to $(\R,\mathcal{B}(\R))$ and hence measurable.
\begin{equation} \label{Eq:N:Skorohod_Proof_Measurability_of_Energy}
\begin{split}
&(\rho,u,W) \to |\sqrt{\rho}u |_{\LL{\infty}{2}}^{2p} + |\rho|_{\LL{\infty}{\gamma}}^{\gamma p}+|\delta^{\frac{1}{\beta}}\rho|_{\LL{\infty}{\beta}}^{\beta p} + |\htd u|_{\LW{2}{1}{2}}^{2p} \\
&+ |\htd \epsilon^{\frac{1}{2}} \nabla(\rho^{\frac{\gamma}{2}} + \delta^{\frac{1}{2}}\rho^{\frac{\beta}{2}})|_{\LL{2}{2}}^{2p} .
\end{split}
\end{equation}
 To check part 2, note that $\hat{\p}(\hat{T}_{\tau}(\hat{\Omega}_{n}))=1$, so the recovery maps allow us to preserve the continuity equation.  Recovering the momentum equation and checking adaptedness with respect to the new filtration requires a small argument regarding the recovery of stochastic integrals.  This is not used directly in the proof below so we omit it, for more details see \cite{bavnas2013convergent}. 
\end{proof}

\subsection{$\tau \to 0$ Identification Step} \label{SubSec:N:2:Identification}
%!TEX root = main.tex
Recall that $r_{t}$ denotes a restriction operator.  Consider $r_{t}$ on the spaces dictated by Proposition \ref{Prop:N:Skorohod} and define a filtration $\{ \hat{\mathcal{F}}_{t}^{n} \}_{t=0}^{T}$ by
\begin{equation}
\hat{\mathcal{F}}_{t}^{n} = \sigma \big ( r_{t}\rnh , r_{t}\unh, r_{t} \big ( \rnh \unh \big ), r_{t}\widehat{W}_{n}\big ). 
\end{equation}

\begin{Lem} \label{Lem:N:Id:Parabolic_Equation}
The pair $(\rnh,\unh)$ satisfies the parabolic equation, \eqref{Eq:N:Setup:Continuity_Equation} of Definition \ref{Def:N:Setup:N_Layer_Approximation}.
\end{Lem}
\begin{proof} \quad By part \ref{Prop:N:Skorohod:Item:Preserving_the_Equation} of Proposition \ref{Prop:N:Skorohod},  the parabolic equation \eqref{Eq:Tau:Setup:Continuity_Equation} is satisfied by $(\rt,\ut)$ on $\Onh$ a.s. with respect to $\pnh$. Applying the pointwise convergences \eqref{Eq:N:Skorohod:Pointwise_Convergence:Density}-\eqref{Eq:N:Skorohod:Pointwise_Convergence:Velocity} along with the weak convergence of $\htd \to \frac{1}{2}$, we easily pass the limit in the weak form.
\end{proof}
\begin{Lem} \label{Lem:N:Id:Momentum_Equation}
The pair $(\rnh,\unh)$ satisfies the projected momentum equation, \eqref{Eq:N:Setup:Momentum_Equation} of Definition \ref{Def:N:Setup:N_Layer_Approximation}.
\end{Lem}
%%%%%%% Proof of tau to zero momentum convergence
%%%%%%%
%%%%%%%
\begin{proof} \quad Given $\phi \in X_{n}$,  we define the $\{ \Fnh^{t}\}_{t=0}^{T}$ adapted, continuous stochastic process $\{ \widehat{M}_{t}^{n}\}_{t=0}^{T}$ via 
\begin{align*}
&\widehat{M}_{t}^{n}(\phi)= \int_{D} \hat{\rho}_{n}\hat{u}_{n}(t) \cdot \phi dx  - \int_{D} m_{0}^{\delta}\cdot \phi dx - \int_{0}^{t}\int_{D} [\hat{\rho}_{n}\hat{u}_{n}\tensor \hat{u}_{n}-2\mu\nabla \hat{u}_{n} -\lambda \Div \unh] : \nabla \phi \, dxds \\
&-\int_{0}^{t}\int_{D} \big [(\hat{\rho}_{n}^{\gamma}+\delta \hat{\rho}_{n}^{\beta})\Div \phi -\epsilon \nabla \hat{u}_{n} \nabla \hat{\rho}_{n} \cdot \phi \big ] dxds.  
\end{align*}
In an analogous manner, we define the processes $\{ \widehat{M}_{t}^{\tau}(\phi)\}_{t=0}^{T}$ and $\{ M_{t}^{\tau}(\phi)\}_{t=0}^{T}$ via
\begin{align*}
&\widehat{M}_{t}^{\tau}(\phi)= \int_{D} \rth \uth(t) \cdot \phi dx  - \int_{D} m_{0}^{\delta}\cdot \phi dx - \int_{0}^{t}\int_{D}2 \htd   [\rth\uth\tensor \uth-2\mu\nabla \uth -\lambda \Div \uth] : \nabla \phi dxds \\
&+\int_{0}^{t}\int_{D}2 \htd \big [ (\rth^{\gamma}+\delta \rth^{\beta})\Div \phi \, + 2\epsilon \nabla \uth \nabla \rth \cdot \phi  \big ] dxds .  \\
&M_{t}^{\tau}(\phi)= \int_{D} \rt \ut(t) \cdot \phi dx  - \int_{D} m_{0}^{\delta}\cdot \phi dx - \int_{0}^{t}\int_{D}2 \htd  [\rt \ut \tensor \ut-2\mu\nabla \ut -\lambda \Div \ut] : \nabla \phi \\
& +\int_{0}^{t}\int_{D} \htd \big [ (\rt^{\gamma}+\delta \rt^{\beta})\Div \phi \,+2\epsilon \nabla \ut \nabla \rt \cdot \phi \big ] dxds. 
\end{align*}
Our plan is to check the criterion laid forth in Lemma \ref{Lem:Appendix:Three_Martingales_Lemma}, in order to identify
\begin{align*}
\widehat{M}_{t}^{n}(\phi) = \sum_{k=1}^{n} \int_{0}^{t}\int_{D} \rnh \sigma_{k,n,\delta}(\rnh,\rnh\unh) \cdot \phi dxd\bnhk(s).
\end{align*}
This implies the momentum equation \eqref{Eq:N:Setup:Momentum_Equation} holds.  Let us fix in advance two arbitrary times $s<t$ and a continuous functional $\gamma$ 
\begin{align*}
\gamma : C \big ([0,s]; L^{\beta}_{x} \big ) \times C \big ( [0,s] ; X_{n} \big ) \times  C \big ([0,s] ; L^{\beta}_{x} \big ) \times C \big ( [0,s] ; X_{n} \big ) \to \R
\end{align*}
%\label{lem: TauSkp},\label{SubLem: Tau Layer/Transport Maps},\label{SubLem: Tau Layer/Upholding the Equation},\label{SubLem: Tau Layer/Energy Bounds},\label{eq:tauPw1}-\label{eq:tauPw3}
which will be used repeatedly below.  We will repeately use the fact that in order to verify a process $\{ N_{t}\}_{t=0}^{T}$ is a $\{\hat{\mathcal{F}}_{t}^{n} \}_{t=0}^{T}$ martingale on $(\hat{\Omega}_{n},\Fnh,\pnh)$, it suffices to verify
$$
\E^{\pnh} \bigg [ \gamma\big ( r_{s} \rnh,r_{s}\unh, r_{s} (\rnh \unh) , r_{s} \widehat{W}_{n}\big ) \big (N_{t}-N_{s}\big) \bigg ] =0 .
$$

We start by using the Levy Characterization to verify$ \, \{\bnhk(t)\}_{t=0}^{T}$ is an $\{ \Fnh^{t}\}_{t=0}^{T}$ Brownian Motion. Applying the pointwise convergences \eqref{Eq:N:Skorohod:Pointwise_Convergence:Density}-\eqref{Eq:N:Skorohod:Pointwise_Convergence:Brownians} together with the uniform bounds \eqref{Eq:N:Skorohod:Preserved_Tau_Layer_Bounds} we find that
\begin{align*}
&\E^{\pnh} \bigg [ \gamma\big ( r_{s} \rnh,r_{s}\unh, r_{s} (\rnh \unh) , r_{s} \widehat{W}_{n}\big ) \big ( \bnhk(t)-\bnhk(s) \big) \bigg ] \\
&= \lim_{\tau \to 0} \E^{\pnh} \bigg [ \gamma\big ( r_{s} \rt, r_{s} \ut , r_{s} (\rt \ut) , r_{s} W_{\tau}\big ) \big ( \btk(t)-\btk(s) \big) \bigg ]. 
\end{align*}
 Using Part \ref{Prop:N:Skorohod:Item:Recovery_Maps} of Proposition \ref{Prop:N:Skorohod} with a change of variables, then the martingale property of $\hat{\beta}_{k}$, we deduce
%%%% Brownian martingale
%%%%
\begin{align*}
&\lim_{\tau \to 0} \E^{\pnh} \bigg [ \gamma\big ( r_{s} \rt, r_{s} \ut , r_{s} (\rt \ut) , r_{s} W_{\tau}\big ) \big ( \btk(t)-\btk(s) \big) \bigg ] \\
&= \lim_{\tau \to 0} \E^{\hat{\p}} \bigg [ \gamma\big ( r_{s}\rth, r_{s} \uth , r_{s} (\rth \uth) , r_{s} \widehat{W}\big ) \big ( \hat{\beta}_{k}(t)-\hat{\beta}_{k}(s) \big) \bigg ] = 0.
\end{align*}
Similarly, one verifies 
\begin{align*}
&\E^{\pnh} \bigg [ \gamma\big ( r_{s} \rnh, r_{s} \un , r_{s} \rnh \unh , r_{s} \widehat{W}_{n}\big ) \big ( \bnhk(t)^{2}-\bnhk(s)^{2}-t+s \big) \bigg ]=0.
\end{align*}
%%%%% Momentum Martingale
%%%%%
Next we check that $ \{ \widehat{M}_{t}^{n}(\phi) \}_{t=0}^{T}$  is an $\{ \hat{\mathcal{F}}_{t}^{n} \}_{t=0}^{T}$ martingale with quadratic variation
$$ \sum_{k=1}^{n} \int_{0}^{t} \left (\int_{D} \rnh \sigma_{k,\tau,n,\delta}(\rnh,\rnh\unh) \cdot \phi \right )^{2}ds. 
$$
Recall that $\htd \to \frac{1}{2}$ weakly.  Hence, by using \eqref{Eq:N:Skorohod:Pointwise_Convergence:Density}-\eqref{Eq:N:Skorohod:Pointwise_Convergence:Brownians} together with \eqref{Eq:N:Skorohod:Preserved_Tau_Layer_Bounds} ;  followed by \eqref{Eq:Tau:Setup:Momentum_Equation} of Definition \ref{Def:Tau:Setup:Tau_Layer_Approximation:Item:Momentum_Equation}, we obtain:
\begin{align*}
&\E^{\pnh} \bigg [ \gamma\big ( r_{s} \rnh, r_{s}\unh,r_{s} (\rnh \unh) , r_{s} \widehat{W}_{n}\big ) \big ( \widehat{M}_{t}^{n}(\phi)-\widehat{M}_{s}^{n}(\phi) \big) \bigg ] \\
&= \lim_{\tau \to 0} \E^{\pnh} \bigg [\gamma\big ( r_{s} \rt, r_{s}\ut , r_{s} (\rt \ut) , r_{s} W_{\tau}\big ) \big ( M_{t}^{\tau}(\phi)-M_{s}^{\tau}(\phi) \big) \bigg ] \\
& \lim_{\tau \to 0} \E^{\hat{\p}} \bigg [ \gamma\big ( r_{s}\rth, r_{s} \uth , r_{s} (\rth \uth) , r_{s} \widehat{W}_{\tau}\big ) \big ( \widehat{M}_{t}^{\tau}(\phi)-\widehat{M}_{s}^{\tau}(\phi) \big) \bigg ] = 0 
\end{align*}
In the remaining analysis we will suppress the arguments of $\gamma$.  Similarly, using the same facts as above along with the weak convergence $\hts \to \frac{1}{2}$ we find
\begin{align*}
&\E^{\pnh} \bigg [ \gamma  \bigg ( \big (\widehat{M}_{t}^{n}(\phi) \big )^{2}-\big ( \widehat{M}_{s}^{n}(\phi) \big )^{2}- \sum_{k=1}^{n} \int_{s}^{t}(\int_{D} \rnh \sigma_{k,\tau,n,\delta}(\rnh,\rnh\unh) \cdot \phi dx)^{2}dr \bigg) \bigg ] \\
&= \lim_{\tau \to 0} \E^{\pnh} \bigg [ \gamma \bigg ( \big (M_{t}^{\tau}(\phi) \big )^{2}-\big ( M_{s}^{\tau}(\phi) \big )^{2}- \sum_{k=1}^{n}\int_{s}^{t} 2\hts(r)(\int_{D} \rt \sigma_{k,\tau,n,\delta}(\rt,\rt\ut) \cdot \phi dx)^{2}dr \bigg) \bigg ]\\
&= \lim_{\tau \to 0}\E^{\hat{\p}} \bigg [ \gamma \bigg ( \big (\widehat{M}_{t}^{\tau}(\phi) \big )^{2}-\big ( \widehat{M}_{s}^{\tau}(\phi) \big )^{2}- \sum_{k=1}^{n} \int_{s}^{t}2\hts(r)(\int_{D} \rth \sigma_{k,\tau,n,\delta}(\rth,\rth \uth) \cdot \phi dx)^{2}dr \bigg) \bigg ] = 0. 
\end{align*}
In essentially the same way, we check 
\begin{align*}
&\E^{\pnh} \bigg [ \gamma \bigg ( \widehat{M}_{t}^{n}(\phi) \bnhk(t) - \widehat{M}_{s}^{n}(\phi) \bnhk(s)- \int_{s}^{t}\int_{D} \rnh \sigma_{k,\tau,n,\delta}(\rnh,\unh) \cdot \phi dxdr \bigg) \bigg ]=0. 
%&= \lim_{\tau \to 0} \E^{\pnh} \bigg [ \gamma\big ( r_{s}\rt, r_{s} (\rt \ut) , r_{s} W_{\tau}\big )  \bigg ( M_{t}^{\tau}(\phi) \btk(t) - M_{s}^{\tau}(\phi) \btk(s)- \sqrt{\lambda_{j}}\int_{s}^{t}\int_{D} \rt e_{j} \cdot \phi dxdr \bigg) \bigg ]\\
%&\lim_{\tau \to 0} \E^{\hat{\p}} \bigg [ \gamma\big ( r_{s}\rth, r_{s} \rth \uth , r_{s} \widehat{W}_{\tau}\big ) \bigg ( \widehat{M}_{t}^{\tau}(\phi) \hat{\beta}_{k}(t) - \widehat{M}_{s}^{\tau}(\phi) \hat{\beta}_{k}^{n}(s)- \sqrt{\lambda_{j}}\int_{s}^{t}\int_{D} \rth e_{j} \cdot \phi dxdr \bigg) \bigg ] = 0 \\
\end{align*}
 Using Lemma \ref{Lem:Appendix:Three_Martingales_Lemma}, we conclude.
\end{proof} 
%%%%%%%%%%%%%%%%%%%%%%%%%%%%%%%%%%%%%%%%%%%%%%%%%%%%%%%%%%%%%%%%%%%%%%%%%%%%%%%%%%%%%%%%%%%%%%%%%%%%%%%%%%%%%%%%%%%%%%%%%%%%%%%%%%%%%%%%%%%%%%%%%%%%%%%%%%%%%%%%%%
\begin{Lem} \label{Lem:N:Id:Energy_Identity}
The pair $(\rnh,\unh)$ satisfies the energy identity, \eqref{Eq:N:Setup:Energy_Identity} from Definition \ref{Def:N:Setup:N_Layer_Approximation}.
\end{Lem}
\begin{proof}  Define the $\{ \Fnh^{t}\}_{t=0}^{T}$ continuous, adapted stochastic process $\{ \widehat{E}_{t}^{n}\}_{t=0}^{T}$ via
\begin{align*}
&\widehat{E}_{t}^{n} =\int_{D} [ \frac{1}{2}\rnh |\unh|^{2}(t) + \frac{1}{\gamma -1} \rnh^{\gamma}(t) + \frac{\delta}{\beta -1} \rnh^{\beta}(t)  ]dx \\
&+ \int_{0}^{t}\int_{D} \left [ 2 \mu|\nabla \unh|^{2} + \lambda (\Div \unh)^{2}+\epsilon(\rnh^{\gamma-2} + \delta \rnh^{\beta -2})|\nabla \rnh|^{2} \right ] dxds \\
&- \sum_{k=1}^{n} \int_{0}^{t}\int_{D} \rth(s)|\sigma_{k,\tau,n,\delta}(\rth,\rth\uth)|^{2} dxds-E_{n}(0)  
\end{align*}
We may now implement the same method as in Lemma \ref{Lem:N:Id:Momentum_Equation}, in order to identify $\{ \widehat{E}_{t}^{n} \}_{t \geq 0}$ as
$$
\widehat{E}_{t}^{n}=\sum_{k=1}^{n}\int_{0}^{t}\int_{D}\rnh \unh\cdot \sigma_{k,\tau,n,\delta}(\rth,\rth \uth)dxd\bthk(s).
$$ 
One can now perform the same manipulations as in the formal estimates in order to deduce the uniform bounds.
\end{proof}

\begin{proof}[Proof of Theorem \ref{Thm:N:Setup:N_Layer_Existence}]
For each $n \geq 1$, we obtain an $n$ layer approximation $(\rnh,\unh)$ using our compactness step, Proposition \ref{Prop:N:Skorohod}.  Indeed, we can check each part of Definition \ref{Def:N:Setup:N_Layer_Approximation} as follows: Part \ref{Def:N:Setup:N_Layer_Approximation:Item:Measurability} follows immediately from the definition of the filtration $\{ \hat{\mathcal{F}}_{t}^{n} \}_{t=0}^{T}$, the continuity and momentum equations are satisfied in view of Lemmas \ref{Lem:N:Id:Parabolic_Equation} and \ref{Lem:N:Id:Momentum_Equation}, while the approximate energy identity \ref{Lem:N:Id:Energy_Identity} follows from Lemma \ref{Lem:N:Id:Energy_Identity}.  To obtain the uniform bounds \ref{Eq:N:Setup:Uniform_Bounds}, use the energy identity along with the method described in Section \ref{SubSec:P:5:Energy_Estimates}.
\end{proof}

\section{$\epsilon$ Layer Existence}\label{Section:Eps}
%!TEX root = main.tex
%!TEX root = main.tex
This section is devoted to the $\epsilon$ layer existence theory; sending $n \to \infty$ our goal is to prove:
\begin{Thm}  \label{Thm:Eps:Setup:Eps_Layer_Existence}
There exists a sequence $\{(\reh,\ueh)\}_{\epsilon >0}$ of $\epsilon$ layer approximations(in the sense of Definition \ref{Def:Eps:Setup:Eps_Layer_Approximation} below), relative to a collection of stochastic bases $\{ (\Oeh, \Feh, \{ \Feh^{t}\}_{t=0}^{T}, \peh, \{ \behk\}_{k=1}^{\infty}) \}_{ \epsilon >0}$, such that for all $p\geq 1$
\begin{equation} \label{Eq:Eps:Setup:Uniform_Bounds}
\begin{split}
&\sup_{\epsilon >0} \E^{\peh} \bigg [ |\sqrt{\reh}\ueh |_{\LL{\infty}{2}}^{2p} + |\reh|_{\LL{\infty}{\gamma}}^{\gamma p}+ |\ueh|_{\LHZ{2}{1}}^{2p} \bigg ] < \infty \\
&\sup_{\epsilon >0} \E^{\peh} \bigg [ |\delta^{\frac{1}{\beta}}\reh|_{\LL{\infty}{\beta}}^{\beta p} + | \epsilon^{\frac{1}{2}} \nabla(\reh^{\frac{\gamma}{2}} + \delta^{\frac{1}{2}}\reh^{\frac{\beta}{2}})|_{\LL{2}{2}}^{2p} \bigg ] < \infty. \\
\end{split}
\end{equation}
\end{Thm}

For each $k$, introduce the operator $\sigma_{k,\delta} : L^{1}_{x} \times [L^{1}_{x}]^{d} \to C^{0}_{x}$ via the prescription
\begin{align*}
\sigma_{k,\delta}(\rho,m) = \sigma_{k}(\rho * \eta_{\delta}(\cdot),m * \eta_{\delta}(\cdot),\cdot).
\end{align*}
\begin{Def} \label{Def:Eps:Setup:Eps_Layer_Approximation} 
A pair $(\reh,\ueh)$ is an $\epsilon$ layer approximation to \eqref{Eq:I:SCNS_System} provided there exists a stochastic basis $(\Oeh, \Feh, \{\Feh^{t} \}_{t=0}^{T} , \pe, \big \{ \behk \big \}_{k\geq 1})$ such that 
\begin{enumerate}
\item \label{Def:Eps:Setup:Eps_Layer_Approximation:Measurability}
The pair $\left (\reh,\reh \ueh \right ) : \Oeh \times [0,T] \to [ L^{\beta}_{x} \times L^{\frac{2\beta}{\beta+1}}_{x}]_{w}$ is an $\{\Feh^{t} \}_{t=0}^{T}$ progressively measurable stochastic proces with $\peh$ a.s. continuous sample paths.  The velocity $\ueh \in L^{2} \left (\Oeh \times [0,T] ; H^{1}_{0,x} \right )$ belongs to the equivalence classes of $\{\Feh^{t} \}_{t=0}^{T}$ progressively measurable $H^{1}_{0,x}$ valued processes.  \\
\item \label{Def:Eps:Setup:Eps_Layer_Approximation:Continuity_Equation}
For all $\phi \in C^{\infty}(D)$ and all $t \in [0,T]$, the following equality holds $\peh$ a.s.
\begin{equation} \label{Eq:Eps:Setup:Continuity_Equation}
	\int_{D} \reh(t)\phi dx = \int_{D} \rho_{0,\delta} + \int_{0}^{t}\int_{D} [ \reh \ueh \cdot \nabla \phi + \epsilon \reh \Delta \phi] dxds.
\end{equation}
\item \label{Def:Eps:Setup:Eps_Layer_Approximation:Momentum_Equation}
For all $\phi \in \big [C_{c}^{\infty}(D) \big ]^{d}$ and all $t \in [0,T]$, the following equality holds $\peh$ a.s.
\begin{equation} \label{Eq:Eps:Setup:Momentum_Equation}
\begin{split}
&\int_{D} \reh \ueh(t) \cdot \phi dx  = \int_{D} m_{0,\delta}\cdot \phi + \int_{0}^{t}\int_{D} [\reh \ueh \tensor \ueh-2\mu\nabla \ueh-\lambda \Div \ueh I] : \nabla \phi dxds \\
&+\int_{0}^{t}\int_{D} \big [ (\reh^{\gamma}+\delta \reh^{\beta})\Div \phi-\epsilon \nabla \ueh \nabla \reh \cdot \phi \big ] dxds \\
&+ \sum_{k=1}^{\infty}\int_{0}^{t}\int_{D} \reh \sigma_{k,\delta}(\reh,\reh \ueh) \cdot \phi dx d \behk(s). 
\end{split}
\end{equation}
\end{enumerate}
\end{Def}
For each $\epsilon >0$ fixed, we can apply Theorem \ref{Thm:N:Setup:N_Layer_Existence} to obtain a sequence 
$\{ (\hat{\rho}_{n,\epsilon},\hat{u}_{n,\epsilon}) \}_{n=1}^{\infty}$ of $n$ layer approximations satisfying the uniform bounds \eqref{Eq:N:Setup:Uniform_Bounds}. In Section \ref{SubSec:Eps:Compactness} we switch probability spaces and use the recovery maps to define a new sequence 
$\{ (\rho_{n,\epsilon},u_{n,\epsilon}) \}_{n=1}^{\infty}$ and obtain compactness.  We extract a limit point $(\reh,\ueh)$ then verify our candidate is an $\epsilon$ layer approximation in Section \ref{SubSec:Eps:Identification}. 
\subsection{$n \to \infty$ Compactness Step} \label{SubSec:Eps:Compactness}
%!TEX root = main.tex 
Next we establish the following compactness result:
\begin{Prop} \label{Prop:Eps:Skorohod} 
There exists a probability space $(\Oeh,\Feh,\peh)$, along with a sequence of recovery maps $\{\widehat{T}_{n}\}_{n=1}^{\infty}$ and limit points $\left (\reh,\ueh,\overline{\sqrt{\reh} \ueh}, \overline{\reh^{\gamma}} + \delta \overline{\reh^{\beta}} \right )$ \\
$$
\widehat{T}_{n} : (\Oeh,\Feh,\peh) \to (\Onh,\Fnh,\pnh)
$$
such that the following hold:
\begin{enumerate}
\item	\label{Prop:Eps:Skorohod:Item:Recovery_Maps}
For each $n$, the measure $\pnh$ may be recovered from $\peh$ by pushing forward $\widehat{T}_{n}$.  \\
\item \label{Prop:Eps:Skorohod:Item:Preserving_The_Equation}
The new sequence $\{(\rn,\un)\}_{n=1}^{\infty}$ defined by  $ \left(\rn,\un \right )=\left (\rnh ,\unh \right ) \circ \widehat{T}_{n}$ constitutes an $n$ layer approximation relative to the stochastic basis $(\Oeh,\Feh,\peh,\{ \mathcal{F}^{t}_{n}\}_{t=0}^{T},W_{n} )$, where
$$ W_{n} := \widehat{W}_{n} \circ \widehat{T}_{n} \quad \quad \quad \F^{t}_{n} := \sigma \left (r_{t}\rn, r_{t}\un , r_{t}(\rn \un), r_{t}W_{n} \right )$$
\item The following uniform bounds hold for all $p \geq 1$
\begin{equation} \label{Prop:Eps:Skorohod:Item:Preserving_The_Bounds}
\begin{split}
&\sup_{n} \E^{\peh} \bigg [ |\sqrt{\rn}\un |_{\LL{\infty}{2}}^{2p} + |\rn|_{\LL{\infty}{\gamma}}^{\gamma p}+|\un|_{L^{2}_{t}(H^{1}_{0,x})}^{2p}  \bigg ] < \infty \\
& \sup_{n} \E^{\peh} \bigg [ |\delta^{\frac{1}{\beta}}\rn|_{\LL{\infty}{\beta}}^{\beta p}+ | \epsilon^{\frac{1}{2}} \nabla(\rn^{\frac{\gamma}{2}} + \delta^{\frac{1}{2}}\rn^{\frac{\beta}{2}})|_{\LL{2}{2}}^{2p} \bigg ] < \infty 
\end{split}
\end{equation}
\item The following convergences hold $\peh$ a.s.
\begin{align} 
\label{Eq:Eps:Skorohod:Pointwise_Convergence:Density}& \rn \to \reh \quad &\text{in} \quad  &C_{t}\big ( [L^{\beta}_{x}]_{w} \big ) \cap \LL{2}{2}\\
\label{Eq:Eps:Skorohod:Pointwise_Convergence:Velocity}& \un \to \ueh \quad &\text{in} \quad &[L^{2}_{t}(H^{1}_{0,x})]_{w}  \\
\label{Eq:Eps:Skorohod:Pointwise_Convergence:Momentum}& \rn \un \to \reh \ueh \quad & \text{in} \quad &\CLw{\frac{2\beta}{\beta+1}} \\
\label{Eq:Eps:Skorohod:Pointwise_Convergence:Brownian_Motions}& W_{n} \to \widehat{W}_{\epsilon} \quad &\text{in} \quad & [C_{t}]^{\infty} 
\end{align}
\item The following additional convergences hold
\begin{align}
\label{Eq:Eps:Skorohod:Weak_Convergence:Brownian_Motions:Kinetic_Energy}&\sqrt{\rho_{n}}u_{n} \to \overline{\sqrt{\reh} \ueh }  \quad &\text{in} \quad &L^{p}_{w^{*}} \big ( \Oeh ; \LL{\infty}{2}\big ) \\
\label{Eq:Eps:Skorohod:Weak_Convergence:Brownian_Motions:Velocity}& u_{n} \to \ueh  \quad &\text{in}  \quad &L^{p}_{w} \big (\Oeh ; L^{2}_{t}(H^{1}_{0,x}) \big ) \\
\label{Eq:Eps:Skorohod:Weak_Convergence:Brownian_Motions:Density}& \rho_{n} \to \reh  \quad &\text{in} \quad &L^{p}_{w^{*}} (\Oeh ; \LL{\infty}{\beta} ) \cap L^{p}_{w} \big ( \Oeh ; \LW{2}{1}{2} \big ) 
\end{align}
\end{enumerate}
\end{Prop}
The first step is another tightness lemma.  Enumerate a dense subset $\{ \phi_{k}\}_{k=1}^{\infty}$ of $L^{\frac{2\beta}{\beta+1}}_{x}$.  Define the sequence of random variables $\{  (\hat{X}_{n} , \hat{Y}_{n}) \}_{n\geq 1}$ via the prescription 
\begin{align*}
	&\hat{X}_{n}=\left (\rnh,\unh,\rnh \unh,\{ \bnhk\}_{k=1}^{n}\right ) \\
	&\hat{Y}_{n} = \{ \langle \rnh\unh,\phi_{k} \rangle_{L^{2}_{x}}\}_{k=1}^{n}.
\end{align*}
Our convention is that given a topological vector space $G$, a finite sequence $\{ x_{k}\}_{k=1}^{n}$ is viewed as an element of $G^{\infty}$ where $x_{j}=0$ for $j \geq n$.  These random variables induce measures on the following topological spaces
\begin{align*} \label{Eq:Eps:Comp:Tightness_Space}
	&E = \CLw{\beta} \cap \LLs{2} \times [L^{2}_{t}(H^{1}_{0,x})]_{w} \times  \left [ \LL{\infty}{\frac{2\beta}{\beta+1}} \right ]_{w^{*}}  \times [C_{t}]^{\infty} \\
	&F=\left [ C_{t} \right ]^{\infty} .
\end{align*}
\begin{Lem} \label{Lem:Eps:Comp:Tightness}
The sequence of induced measures $\{\hat{\p}_{n} \circ (\hat{X}_{n}, \hat{Y}_{n})^{-1}\}_{n\geq 1}$ is tight on $E \times F$.
\end{Lem}
\begin{proof} \quad It suffices to consider each component of $\hat{X}_{n}$ separately.  The tightness of $\{ \pnh \circ (\unh,\rnh \unh)^{-1} \}_{n \geq 1}$ follows immediately from the bounds \eqref{Eq:N:Setup:Uniform_Bounds} and Banach Alaoglu. To treat the collection of SBM, note
\begin{equation} 
	\sup_{k,n: \, k \leq n} \E^{\pnh} \big [ |\bnhk|_{C^{\frac{1}{3}}_{t}}^{2} \big ]  < \infty .
\end{equation}
For each $M >0$, the set below is compact in $\left [C_{t} \right ]^{\infty}$ by Arzela-Ascoli and Tychnoff. 
%\begin{equation*}
%\pnh(\widehat{W}_{n} \in \prod_{j=1}^{\infty} \{ f \in C_{t}  \mid |f|_{C^{\frac{1}{3}}_{t}} \leq \sqrt{M 2^{j}}  \}) \leq \sum_{j=1}^{n} \pnh (|\bnhk|_{C_{t}^{1/3}} \geq \sqrt{M {2^{j}}}) \leqs \frac{1}{M}
%\end{equation*}
\begin{equation} 
\prod_{j=1}^{\infty} \{ f \in C_{t}  \mid |f|_{C^{\frac{1}{3}}_{t}} \leq M2^{j}  \}
\end{equation}
Choosing $M>0$ appropriately and summing a geometric series gives the desired tightness of $\{\pnh \circ \widehat{W}_{n}^{-1}\}_{n \geq 1}$.  Recall that $\rnh(0)=\rho_{0,\delta}$ by Part \ref{Prop:N:Skorohod:Item:Preserving_the_Equation} of Proposition \ref{Prop:N:Skorohod}. Since $\beta >d$ we may choose a $\theta \in (0,1)$ and define a $q >2$ by the relation $\frac{1}{q}=\theta(\frac{1}{\beta}+\frac{1}{2}-\frac{1}{d}) +(1-\theta)(\frac{1}{2\beta}+\frac{1}{2})$.  Maximal regularity results for parabolic equations and interpolation give the $\pnh$ a.s. inequality
\begin{equation} \label{Lem:Eps:Comp:Tightness:Parabolic_Estimate}
\begin{split}
&|\partial_{t} \rnh|_{\LW{q}{-1}{q}} + \epsilon |\rnh|_{\LW{q}{1}{q}} 
\leqs |\rho_{0,\delta}|_{L^{\beta}_{x}}+|\rnh \unh|_{\LLs{q}} \\
&\leqs 1+|\rnh \unh|_{\LL{2}{\frac{2^{*}\beta}{\beta + 2^{*}}}}^{\theta} |\rnh \unh |_{\LL{\infty}{\frac{2\beta}{\beta+1}}}^{1-\theta} \\
&\leqs 1+ |\rnh|_{\LL{\infty}{\beta}}^{\frac{1}{2}(1+\theta)}|\unh|_{\LL{2}{2^{*}}}^{1-\theta} |\sqrt{\rnh} \unh|_{\LL{\infty}{2}}^{1-\theta} . \\
\end{split}
\end{equation}
Hence, the LHS is uniformly controlled in $L^{2}(\Onh)$ in view of the uniform bounds \eqref{Eq:N:Setup:Uniform_Bounds} and H\"older(in $\omega$).   Corollary \ref{Cor:Appendix:Time_Derivative_In_A_Negative_Sobolev_Space} and Aubin-Lions imply that for each $M >0$, the following set is compact in $\CLw{\beta} \cap \LLs{2}$.
$$
\{f \in \LL{\infty}{\beta} \cap \LLs{2} \mid |\partial_{t}f|_{\LW{q}{-1}{q}} + |f|_{\LW{q}{1}{q}} + |f|_{\LL{\infty}{\beta}} \leq \sqrt{M} \}
$$
Using the uniform bounds on $\{\rnh\}_{n=1}^{\infty}$ in $L^{2}(\Onh ; \LL{\infty}{\beta})$ gives the tightness of $\{\pnh \circ \rnh^{-1}\}_{n=1}^{\infty}$ for an appropriate choice of $M>0$.  To address the sequence $\pnh \circ \hat{Y}_{n}^{-1}$, let $s<t$ and $\phi \in X_{n}$ be arbitrary and use the momentum equation \eqref{Eq:N:Setup:Momentum_Equation} to decompose $\langle \rnh \unh(t)-\rnh(s)\unh(s), \phi \rangle_{L^{2}_{x}}$ into three terms: the stochastic integrals, the energy correction, and the rest.  To estimate the stochastic integrals, we use the BDG inequality together with the stability of the projection operators, Hypothesis \ref{Hyp:Projections}, via
%Recall that for a one dimensional stochastic integral, when we apply BDG over the time interval $[s,t]$ we first do a trivial bound by a sup that contains the quantity of interest and then we estimate that 
\begin{equation} \label{Lem:Eps:Comp:Tightness:Stochastic_Integrals} 
\begin{split}
	& \E^{\pnh} \bigg [ \big (\sum_{k=1}^{n} \int_{s}^{t}\int_{D} \rnh \sigma_{k,n,\delta}(\rnh,\rnh \unh)\cdot \phi dxd\hat{\beta}_{n}^{k} \big )^{p} \bigg ] \\
	&\leqs  \E^{\pnh} \left [  \big (\sum_{k=1}^{\infty} \int_{s}^{t}(\int_{D}\rnh \sigma_{k,n,\delta}(\rnh,\rnh \unh)\cdot \phi dx)^{2}dr \big )^{\frac{p}{2}} \right ] \\ 
	&\leqs |\phi|_{L^{\infty}_{x}}^{p}(\sum_{k=1}^{\infty}|\sigma_{k}|_{L^{\gamma'}_{x}(L^{\infty}_{\rho,m})}^{2})^{\frac{p}{2}} \E^{\pnh}\left [ \big (\int_{s}^{t} |\rnh(r)|^{2}_{L^{\gamma}_{x}}dr \big )^{\frac{p}{2}} \right ] \\
	&\leqs  (t-s)^{\frac{p}{2}} |\phi|_{L^{\infty}_{x}}^{p} \E^{\pnh} \left [ |\rnh|_{\LL{\infty}{\gamma}}^{p} \right ] .
\end{split}
\end{equation}
To estimate the energy correction, we use \eqref{Lem:Eps:Comp:Tightness:Parabolic_Estimate} to obtain the following inequality $\pnh$ a.s.
\begin{align*}
&|\epsilon \int_{s}^{t} \int_{D} \nabla \unh \nabla \rnh\cdot  \phi dxds | 
\leq \epsilon |\phi|_{L^{\infty}_{x}}\int_{s}^{t} |\nabla \unh(r)|_{L^{2}_{x}} |\nabla \rnh(r)|_{L^{2}_{x}}dr \\
 & \leq \epsilon (t-s)^{\frac{1}{2}-\frac{1}{q}} |\phi|_{L^{\infty}_{x}} |\unh|_{L^{2}_{t}(H^{1}_{0,x})} |\nabla \rnh|_{\LLs{q}} . \\
\end{align*}
To treat the remaining terms, H\"older and Sobolev yield the $\pnh$ a.s. inequality
\begin{equation} \label{Lem:Eps:Comp:Tightness:Remainder}
\begin{split}
&\int_{s}^{t} \int_{D} \left [ \rnh \unh \tensor \unh - 2\mu \nabla u +(\rnh^{\gamma}+\delta \rnh^{\beta}-\Div \unh)I \right ] : \nabla \phi dxdr \\
&\leq C(t-s)^{\frac{1}{2}} |\nabla \phi|_{L^{\infty}}\left [ |\unh|_{\LW{2}{1}{2}}(|\sqrt{\rnh} \unh|_{\LL{\infty}{2}}|\rnh|_{\LL{\infty}{\gamma}}^{1/2} +1) \right ]  \\
&+C(t-s)^{\frac{1}{2}} |\nabla \phi|_{L^{\infty}}\left [|\rnh|_{\LL{\infty}{\gamma}}^{\gamma} + |\rnh|_{\LL{\infty}{\beta}}^{\beta} \right ].
\end{split}
\end{equation}
Combining \eqref{Lem:Eps:Comp:Tightness:Stochastic_Integrals}-\eqref{Lem:Eps:Comp:Tightness:Remainder} and using \eqref{Eq:N:Setup:Uniform_Bounds} together with \eqref{Lem:Eps:Comp:Tightness:Parabolic_Estimate} we obtain for all $k$ and $p \geq 1$
\begin{equation} \label{eq: modEst}
	  \sup_{n \geq k}\E^{\pnh} \left [ \big|\langle \rnh\unh(t)-\rnh\unh(s),\phi_{k} \rangle_{L^{2}_{x}} \big |^{p} \right ] \leq C |\phi_{k}|_{C^{1}_{x}}^{p} |t-s|^{p(\frac{1}{2}-\frac{1}{q})} .
\end{equation}
Combining this observation with the Sobolev embedding theorem for fractional sobolev spaces(in time), for any $\alpha < \frac{1}{2}-\frac{1}{q}$ there exists a $p$ such that
\begin{equation}
\sup_{n \geq k}\E^{\pnh} \left [ \big | \langle \rnh \unh, \phi_{k} \rangle_{C_{t}^{\alpha}} \big |^{p}\right ]  \leqs  |\phi_{k}|_{C^{1}_{x}}^{p} .
\end{equation}
For each $M >0$, define the set $K_{M}$ via
\begin{equation*}
K_{M} = \prod_{j=1}^{\infty} \{f \in C_{t} \mid \, \, |f|_{C^{\alpha}_{t}} \leq M^{\frac{1}{p}}2^{\frac{j}{p}}|\phi_{j}|_{C^{1}_{x}} \} .
\end{equation*}
In view of Arzela-Ascoli and Tychonoff, $K_{M}$ is a compact set, and Chebyshev yields
\begin{align*}
&\pnh \left (\hat{Y}_{n} \notin K_{M} \right) 
\leq \sum_{k=1}^{n} \pnh \left ( \langle \rnh \unh , \phi_{k} \rangle_{C^{\alpha}_{t}} \geq M^{\frac{1}{p}}2^{\frac{k}{p}}|\phi_{k}|_{C^{1}_{x}}\right) \\
&\leq M^{-1}\sum_{k=1}^{n}2^{-k}|\phi_{k}|_{C^{1}_{x}}^{-p}\sup_{n \geq k}\E^{\pnh} \left [ \big | \langle \rnh \unh, \phi_{k} \rangle_{C_{t}^{\alpha}} \big |^{p}\right ] \leq M^{-1}. 
\end{align*}
 This implies the tightness of the sequence $\{ \pnh \circ \hat{Y}_{n}^{-1} \}_{n \geq 1}$.
\end{proof}

%%%%%%%%%%%%%%%% N TO INFINITY PROOF OF SKOROHOD PACKAGE
%%%%%%%%%%%%%%%%
%%%%%%%%%%%%%%%%
%%%%%%%%%%%%%%%%
Now we can complete the proof of our compactness step.
\begin{proof}[Proof of Proposition \ref{Prop:Eps:Skorohod}] : \quad In view of Remark \ref{Rem:Appendix:Examples_of_Jakubowksi_Spaces}, $E \times F$ is a Jakubowski space.  Thus, we may apply Theorem \ref{Thm:Appendix:Jakubowksi_Skorohod} in order to obtain a sequence of maps $\{\hat{T}_{n} \}_{n=1}^{\infty}$ 
  $$\hat{T}_{n}: (\Oeh,\Feh,\peh) \to (\Onh,\Fnh,\pnh)$$ 
 and a limiting random variable $(\hat{X}_{\epsilon},\hat{Y}_{\epsilon})=(\reh, \hat{m}_{\epsilon} , \ueh, \hat{W}_{\epsilon}, \{ \hat{Y}_{\epsilon}^{k}\}_{k=1}^{\infty})$.  Moreover, the properties listed in Theorem \ref{Thm:Appendix:Jakubowksi_Skorohod} imply directly Part \ref{Prop:Eps:Skorohod:Item:Recovery_Maps} of the Proposition and guarantee that 
 \begin{equation} \label{Eq:Eps:Comp:Consistency}
 \begin{split}
& (\rn , \un , \rn \un , W_{n}, \{ \langle \rn \un, \phi_{k} \rangle \}_{k=1}^{n} )=(\rnh , \unh , \rnh \unh , \widehat{W}_{n}, \{ \langle \rnh \unh, \phi_{k} \rangle \}_{k=1}^{n} ) \circ \hat{T}_{n} \\
&=\left (\rnh \circ \hat{T}_{n}, \, \unh \circ \hat{T}_{n}, \, \rnh \circ \hat{T}_{n} \, \unh \circ \hat{T}_{n} , \, \{ \bnhk \circ \hat{T}_{n}\}_{k=1}^{n},\{ \langle \rnh \unh \circ \hat{T}_{n}, \phi_{k} \rangle \}_{k=1}^{n} \right ) \\
 &= (\hat{X}_{n}, \hat{Y}_{n}) \circ \hat{T}_{n} \to (\hat{X}_{\epsilon}, \hat{Y}_{\epsilon}). 
 \end{split}
 \end{equation}
The limit is understood to hold $\peh$ almost surely in each of the topologies where the tightness was proven. In particular, we obtain the pointwise convergences \eqref{Eq:Eps:Skorohod:Pointwise_Convergence:Density}, \eqref{Eq:Eps:Skorohod:Pointwise_Convergence:Velocity}, and  \eqref{Eq:Eps:Skorohod:Pointwise_Convergence:Brownian_Motions}.  To obtain \eqref{Eq:Eps:Skorohod:Pointwise_Convergence:Momentum}, we need a few more arguments. By \eqref{Eq:Eps:Comp:Consistency}, $\rn \un$ converges to $\hat{m}_{\epsilon}$ in $ [\LL{\infty}{\frac{2\beta}{\beta+1}} ]_{w^{*}}$ \, $\peh$  almost surely.  Moreover, for each $k,n$ satisfying $k \leq n$,  $ \langle \rn \un , \phi_{k} \rangle $ converges to $\hat{Y}_{\epsilon}^{k}$ in $C_{t}$. Hence, for each fixed $\omega$ in a set of full $\peh$ measure, we may appeal to appendix lemma \ref{Lem:Appendix:Compact_Sets_of_Weakly_Continuous_Functions} in order to obtain the convergence of $\rn \un(\omega)$ to $\hat{m}_{\epsilon}(\omega)$ in $\CLw{\frac{2\beta}{\beta+1}}$. Indeed, the first remark yields the necessary uniform bounds and the second gives the equicontinuity.  It is important to remark that since we identified the limit point $\hat{m}_{\epsilon}$ in advance, there is no need to pass to a subsequence and we do not run into any trouble with picking different subsequences for different $\omega$.  Note that \eqref{Eq:Eps:Skorohod:Pointwise_Convergence:Density} implies that $\rn \to \reh$ strongly in $\LW{2}{-1}{2}$, which may be combined with \eqref{Eq:Eps:Skorohod:Pointwise_Convergence:Velocity} to identify $\hat{m}_{\epsilon} =\reh \ueh$ in the sense of distributions, $\peh$ almost surely, yielding \eqref{Eq:Eps:Skorohod:Pointwise_Convergence:Momentum}.\\
 It may be checked with a regularization argument that the energy functional defined in \eqref{Eq:N:Skorohod_Proof_Measurability_of_Energy} remains measurable with respect to the new topology introduced in this section.  Hence, we may recover the uniform bounds \eqref{Prop:Eps:Skorohod:Item:Preserving_The_Bounds} from the prior probability space.  Finally, we may use these bounds along with Banach-Alaogolu theorem to obtain \eqref{Eq:Eps:Skorohod:Weak_Convergence:Brownian_Motions:Kinetic_Energy}-\eqref{Eq:Eps:Skorohod:Weak_Convergence:Brownian_Motions:Density}
\end{proof}
%%%%
%%%%
%%%%
%%%%

\subsection{$n \to \infty$ Identification Step} \label{SubSec:Eps:Identification}
%!TEX root = main.tex
Consider the restriction operator on the spaces dictated by Proposition \ref{Prop:Eps:Skorohod} and define a filtration by
\begin{equation}
\hat{\mathcal{F}}_{t}^{\epsilon} = \sigma \big (r_{t}\reh ,r_{t}\ueh, r_{t} \big ( \reh \ueh \big ), r_{t}\widehat{W}_{\epsilon}\big ) 
\end{equation}
%%%%%%%%%%% CONTINUITY EQUATION
%%%%%%%%%%%
%%%%%%%%%%%
%%%%%%%%%%% 
%%%%%%%%%%%
\begin{Lem} \label{Lem:Eps:Id:Parabolic_Equation}
The pair $(\reh,\ueh)$ satisfies the parabolic equation, \eqref{Eq:Eps:Setup:Continuity_Equation} of Definition \ref{Def:Eps:Setup:Eps_Layer_Approximation}.  Moreover, we have the following convergence upgrade: for all $p \geq 1$
\begin{equation} \label{Eq:Eps:Id:Strong_Convergence}
\lim_{n \to \infty} \E^{\peh} \big [ |\rn-\reh|^{p}_{\LW{2}{1}{2}}\big ]=0
\end{equation}
\end{Lem}
\begin{proof} In view of Part \ref{Prop:Eps:Skorohod:Item:Preserving_The_Equation} of Proposition \ref{Prop:Eps:Skorohod}, $(\rn,\un)$ satisfy the same parabolic equation on the new probability space $\Oeh$ almost surely with respect to $\peh$. Appealing to the pointwise convergences \eqref{Eq:Eps:Skorohod:Pointwise_Convergence:Density}-\eqref{Eq:Eps:Skorohod:Pointwise_Convergence:Velocity}, we may pass to the limit in the weak form $\peh$ a.s. and obtain the same equation for $(\reh,\ueh)$.  To prove the convergence upgrade, begin by appealing to Lemma \ref{Lem:Appendix:Energy_Identity_Parabolic_Neumann} and obtain the following energy identity for all $t \in [0,T]$,  $\peh$ a.s.
\begin{align}
\label{Eq:Eps:Id:Energy1}& \int_{D} \rho_{n}^{2}(t)dx + \epsilon \int_{0}^{t}\int_{D}|\nabla \rho_{n}|^{2}dxdt =\int_{D} \rho_{0,\delta}^{2}dx - \int_{0}^{t}\int_{D}\Div u_{n} \rho_{n}^{2}dxdt. \\
\label{Eq:Eps:Id:Energy2}& \int_{D} \hat{\rho}_{\epsilon}^{2}(t)dx + \epsilon \int_{0}^{t}\int_{D}|\nabla \hat{\rho}_{\epsilon}|^{2}dxdt =\int_{D} \rho_{0,\delta}^{2}  - \int_{0}^{t}\int_{D}\Div \hat{u}_{\epsilon} \hat{\rho}_{\epsilon}^{2}dxdt. 
\end{align}
Using again \eqref{Eq:Eps:Skorohod:Pointwise_Convergence:Density}-\eqref{Eq:Eps:Skorohod:Pointwise_Convergence:Velocity}, we can pass limits on the RHS of \eqref{Eq:Eps:Id:Energy1} and conclude from the LHS of \eqref{Eq:Eps:Id:Energy2} that $\peh$ a.s.
\begin{align}
& \lim_{n \to \infty} \epsilon \int_{0}^{T}\int_{D}|\nabla \rho_{n}|^{2}dxdt = \epsilon \int_{0}^{T}\int_{D}|\nabla \hat{\rho}_{\epsilon}|^{2}dxdt. 
\end{align}
Combining this observation with the pointwise convergence \eqref{Eq:Eps:Skorohod:Pointwise_Convergence:Density} and the uniform bounds \eqref{Eq:N:Setup:Uniform_Bounds}, one obtains
\begin{align*}
\lim_{n \to \infty} |\rn|_{L^{p}\big ( \Oeh ; \LW{2}{1}{2} \big ) } = |\reh|_{L^{p}\big ( \Oeh ; \LW{2}{1}{2} \big ) }. 
\end{align*}
Hence, we may upgrade the weak convergence \eqref{Eq:Eps:Skorohod:Pointwise_Convergence:Momentum} and obtain \eqref{Eq:Eps:Id:Strong_Convergence} as desired.
\end{proof}
%%%%%% MOMENTUM EQUATION
%%%%%%
 \begin{Lem} \label{Lem:Eps:Id:Momentum_Equation}
The pair $(\reh,\ueh)$ satisfies the energy corrected momentum equation \ref{Eq:Eps:Setup:Momentum_Equation} from Definition \ref{Def:Eps:Setup:Eps_Layer_Approximation}.
\end{Lem}
\begin{proof} We follow the same general strategy as in Lemma \ref{Lem:N:Id:Momentum_Equation}.  Indeed, first note that the exact same proof works in order to check that  $\{ \behk\}_{k=1}^{\infty}$ is a collection of $\{ \Feh^{t} \}^{T}_{t=0}$ independent brownian motions.  Next, for each $ \phi \in \cup_{n=1}^{\infty} X_{n}$ we introduce a continuous $\{ \Feh^{t} \}_{t=0}^{T}$ adapted stochastic process $\{\widehat{M}_{t}^{\epsilon}(\phi) \}_{t=0}^{T}$ defined by the relation
\begin{align*}
&\widehat{M}_{t}^{\epsilon}(\phi)= \int_{D} \hat{\rho}_{\epsilon}\hat{u}_{\epsilon}(t) \cdot \phi dx  - \int_{D} m_{0,\delta}\cdot \phi - \int_{0}^{t}\int_{D} [\hat{\rho}_{\epsilon}\hat{u}_{\epsilon}\tensor \hat{u}_{\epsilon}-2\mu \nabla \hat{u}_{\epsilon}-\lambda \Div \ueh I ] : \nabla \phi  dxds \\
&-\int_{0}^{t}\int_{D} \big [ (\hat{\rho}_{\epsilon}^{\gamma}+\delta \hat{\rho}_{\epsilon}^{\beta})\Div \phi-\epsilon \nabla \hat{u}_{\epsilon} \nabla \hat{\rho}_{\epsilon} \cdot \phi \big ] dxds . 
\end{align*}
In view of the pointwise convergences \eqref{Eq:Eps:Skorohod:Pointwise_Convergence:Density}-\eqref{Eq:Eps:Skorohod:Pointwise_Convergence:Velocity},  the following limits hold $\peh$ a.s. for all $t \in [0,T]$ 
\begin{align}\label{Eq:Eps:Id:Pw1}
&\lim_{n\to \infty} \int_{D} \rn \un(t) \cdot \phi dx =  \int_{D}\hat{\rho}_{\epsilon}\hat{u}_{\epsilon}(t) \cdot \phi dx \\
& \lim_{n \to \infty} \int_{0}^{t}\int_{D}\left [ 2 \mu \nabla \un + \lambda \Div \un I \right ] : \nabla \phi dxds = \int_{0}^{t}\int_{D} \left [2 \mu \nabla \ueh + \lambda \Div \ueh \right ] I : \nabla \phi dxds .
\end{align}
Noting the compact embedding $L^{\frac{2\gamma}{\gamma+1}} \hookrightarrow W^{-1,2}_{x}$, we may upgrade \eqref{Eq:Eps:Skorohod:Pointwise_Convergence:Momentum} with Lemma \ref{Lem:Appendix:Bounded_Operator_Upgrade} and obtain $\peh$ a.s. $\rn\un \to \reh \ueh$ in $\LW{2}{-1}{2}$.  Combining with \eqref{Eq:Eps:Skorohod:Pointwise_Convergence:Velocity} we have a weak/strong pairing and obtain $\peh$ a.s. for all $t \in [0,T]$
\begin{equation}\label{Eq:Eps:Id:Pw2}
\lim_{n \to \infty} \int_{0}^{t}\int_{D} \rn \un \tensor \un : \nabla \phi dxdt = \int_{0}^{t}\int_{D} \reh \ueh \tensor \ueh :\nabla \phi dxdt. \\
\end{equation}
Combining \eqref{Eq:Eps:Id:Strong_Convergence}, the strong convergence upgrade for the density, with the weak convergence of the velocity \eqref{Eq:Eps:Skorohod:Pointwise_Convergence:Velocity}, yields(along a subsequence) $\peh$ a.s. for all $t \in [0,T]$ 
\begin{equation}\label{Eq:Eps:Id:Pw3}
\lim_{n \to \infty} \int_{0}^{t}\int_{D} \epsilon \nabla u_{n} \nabla \rho_{n} \cdot \phi dxds = \int_{0}^{t}\int_{D} \epsilon \nabla \hat{u}_{\epsilon} \nabla \hat{\rho}_{\epsilon} \cdot \phi dxds.
\end{equation}
Recalling the interpolation argument in the proof of the $\tau \to 0$ tightness Lemma \ref{Lem:N:Tightness}, we may use the uniform bounds \eqref{Eq:N:Setup:Uniform_Bounds} to obtain further: for all $p \geq 2$
\begin{equation} \label{Eq:Eps:Id:Integrability_Gains}
\sup_{n} \E^{\peh} |\rn|^{p}_{\LL{\beta+1}{\beta+1}} < \infty.
\end{equation}
Combining this observation with \eqref{Eq:Eps:Skorohod:Pointwise_Convergence:Density} gives $\peh$ a.s. for all $t \in [0,T]$
\begin{equation}
\lim_{n \to \infty} \int_{0}^{t}\int_{D} \left (\rn^{\gamma} + \delta \rn^{\beta} \right ) \Div \phi dxds = \int_{0}^{t}\int_{D} \left (\reh^{\gamma} + \delta \reh^{\beta} \right ) \Div \phi dxds.  
\end{equation}
These remarks allow us to proceed as in Lemma \ref{Lem:N:Id:Momentum_Equation} and conclude that $\{ \widehat{M}_{t}^{\epsilon}(\phi)\}_{t=0}^{T}$ is an $\{ \Feh^{t}\}_{t=0}^{T}$ martingale.  
Indeed, a dominated convergence argument reduces the problem to checking that for all $k \geq 1$, $\peh$ a.s. for all $t \in [0,T]$ we have the convergence
\begin{equation}
\lim_{n \to \infty} \int_{D} \rn(t) \sigma_{k,n,\delta}(\rn(t),\rn \un(t)) \cdot \phi dx = \int_{D} \rn(t) \sigma_{k,\delta}(\rn(t),\rn \un(t)) \cdot \phi dx.  
\end{equation}
To check this, note first that $\sigma_{k,\delta} : L^{\beta}_{x} \times [L^{\frac{2\beta}{\beta+1}}_{x}]^{d} \to L^{\frac{\beta}{\beta-1}}_{x}$ is a compact operator.  Since $\sigma_{k,n,\delta} = \Pi_{n} \circ \sigma_{k,\delta}$ we may use the stability of the projections in $L^{\frac{\beta}{\beta-1}}_{x}$ and \eqref{Eq:Eps:Skorohod:Pointwise_Convergence:Density},\eqref{Eq:Eps:Skorohod:Pointwise_Convergence:Velocity} in order to conclude that $\peh$ a.s. for all $t \in [0,T]$ we have $\sigma_{k,n,\delta}(\rho_{n}(s),\rho_{n} u_{n}(s)) \to \sigma_{k,\delta}(\rho(s),\rho u(s))$  in $L^{\frac{\beta}{\beta-1}}_{x}$.  Hence we have a weak/strong pairing and we may conclude.

With these additional convergences at hand, it is straightforward to implement the method in Lemma \ref{Lem:N:Id:Momentum_Equation} and identify
\begin{align*}
\hat{M}^{\epsilon}_{t}(\phi) = \sum_{k=1}^{\infty} \int_{0}^{t}\int_{D} \sigma_{k,\delta}(\reh,\reh\ueh) \cdot \phi dx d\behk(s).
\end{align*} 
Proceeding by a density argument, using Hypothesis \eqref{Hyp:Projections} and appropriately redefining the stochastic integrals on sets of measure zero if necessary, we obtain a single $\peh$ full measure set where the energy corrected momentum equation holds for all $\phi \in \left [C^{\infty}(D) \right ]^{d}$ and $t \in [0,T]$.
\end{proof}

\subsection{Conclusion of the Proof}
%!TEX root = main.tex
\begin{proof}[Proof of Theorem \ref{Thm:Eps:Setup:Eps_Layer_Existence}] For each $\epsilon >0$, we obtain an $\epsilon$ layer approximation $(\reh,\ueh)$ using our compactness step, Proposition \ref{Prop:Eps:Skorohod} together with Lemmas \ref{Lem:Eps:Id:Parabolic_Equation} and \ref{Lem:Eps:Id:Momentum_Equation}. In view of \eqref{Eq:Eps:Skorohod:Weak_Convergence:Brownian_Motions:Kinetic_Energy}-\eqref{Eq:Eps:Skorohod:Weak_Convergence:Brownian_Motions:Density} and the lower semicontinuity of the relevant norms, we obtain for each $\epsilon > 0$
\begin{align*}
 &\E^{\peh} \bigg [ |\overline{\sqrt{\reh}\ueh} |_{\LL{\infty}{2}}^{2p} + |\reh|_{\LL{\infty}{\gamma}}^{\gamma p}+|\delta^{\frac{1}{\beta}}\reh|_{\LL{\infty}{\beta}}^{\beta p} + |\ueh|_{\LW{2}{1}{2}}^{p} \bigg ] \\
& \leq \liminf_{n \to \infty}\E^{\peh} \bigg [ |\sqrt{\rn}\un |_{\LL{\infty}{2}}^{2p} + |\rn|_{\LL{\infty}{\gamma}}^{\gamma p}+|\delta^{\frac{1}{\beta}}\rn|_{\LL{\infty}{\beta}}^{\beta p} + |\un|_{\LW{2}{1}{2}}^{p} \bigg ] .
 \end{align*}
Moreover, in view of \eqref{Eq:Eps:Skorohod:Pointwise_Convergence:Density}-\eqref{Eq:Eps:Skorohod:Pointwise_Convergence:Velocity} we may identify $\sqrt{\reh}\ueh =\overline{\sqrt{\reh}\ueh} $ in the sense of distributions $\peh$ almost surely.  Appealing to \eqref{Eq:N:Setup:Uniform_Bounds}, we may maximize the inequality in $\epsilon$ and obtain the uniform bounds \eqref{Eq:Eps:Setup:Uniform_Bounds}.
\end{proof}

\section{$\delta$ Layer Existence} \label{Section:Del}
%!TEX root = main.tex
%!TEX root = main.tex
In this section, we build our final approximating scheme; the $\delta$ layer.  Sending $\epsilon \to 0$, our plan is to prove:
\begin{Thm}  \label{Thm:Del:Setup:Del_Layer_Existence}
There exists a sequence $\{(\rdh,\udh)\}_{\delta > 0}$ of $\delta$ layer approximations (in the sense of Definition \ref{Def:Del:Setup:Delta_Layer_Approximation} below), relative to a collection of stochastic bases $\{ (\Odh, \Fdh, \{ \Fdh^{t}\}_{t=0}^{T}, \pdh, \{ \bdhk\}_{k=1}^{\infty}) \}_{\delta >0}$, such that for all $p\geq 1$
\begin{equation} \label{Eq:Del:Setup:Uniform_Bounds}
\sup_{\delta >0} \E^{\pdh} \bigg [ |\sqrt{\rdh}\udh |_{\LL{\infty}{2}}^{2p} + |\rdh|_{\LL{\infty}{\gamma}}^{\gamma p}+|\delta^{\frac{1}{\beta}}\rdh|_{\LL{\infty}{\beta}}^{\beta p} + |\udh|_{\LHZ{2}{1}}^{2p} \bigg ] < \infty . \\
\end{equation}
\end{Thm}
Let us proceed to a precise definition a of $\delta$ layer approximation.
\begin{Def} \label{Def:Del:Setup:Delta_Layer_Approximation}  
A pair $(\rdh,\udh)$ is defined to be a $\delta$ layer approximation to \eqref{Eq:I:SCNS_System} provided there exists a stochastic basis $\left ( \Odh, \Fdh, \{ \Fdh^{t} \}_{t=0}^{T} , \pdh, \big \{ \bdhk \big \}_{k=1}^{\infty} \right )$ such that 
\begin{enumerate}
\item  \label{Def:Del:Setup:Delta_Layer_Approximation:Measurability} 
The pair $\left (\rdh,\rdh \udh \right ): \Odh \times [0,T] \to [ L^{\beta}_{x} \times L^{\frac{2\beta}{\beta+1}}_{x}]_{w}$ is an $\{\Fdh^{t} \}_{t=0}^{T}$ progressively measurable stochastic process with $\pdh$ a.s. continuous sample paths.  The velocity $\udh \in L^{2} \left (\Odh \times [0,T] ; H^{1}_{0,x} \right )$ belongs to the equivalence classes of $\{\Fdh^{t} \}_{t=0}^{T}$ progressively measurable $H^{1}_{0,x}$ valued processes. \\
\item \label{Def:Del:Setup:Delta_Layer_Approximation:Continuity_Equation} 
For all $\phi \in C_{c}^{\infty}(D)$ and for all $t \in [0,T]$ the following equality holds $\pdh$ a.s.
\begin{equation} \label{Eq:Del:Setup:Continuity_Equation} 
\int_{D} \rdh(t)\phi dx = \int_{D} \rho_{0,\delta}\cdot \phi + \int_{0}^{t}\int_{D} \rdh \udh \cdot \nabla \phi dxds .  
\end{equation}
\item \label{Def:Del:Setup:Delta_Layer_Approximation:Momentum_Equation}
For all $\phi \in \big [C_{c}^{\infty}(D) \big ]^{d}$ and for all $t \in [0,T]$, the following equality holds $\pdh$ a.s.
\begin{equation} \label{Eq:Del:Setup:Momentum_Equation} 
\begin{split}
&\int_{D} \rdh \udh(t) \cdot \phi dx  = \int_{D} m_{0,\delta}\cdot \phi + \int_{0}^{t}\int_{D} [\rdh \udh \tensor \udh-2\mu \nabla \udh-\lambda \Div \udh I] : \nabla \phi dxds \\
&+ \int_{0}^{t}\int_{D}(\rdh^{\gamma}+\delta \rdh^{\beta})\Div \phi dxds+ \sum_{k=1}^{\infty} \int_{0}^{t}\int_{D} \rdh \sigma_{k,\delta}(\rdh,\rdh \udh) \cdot \phi dx d \bdhk(s).
\end{split}
\end{equation}
\end{enumerate}
\end{Def}
For each $\delta$ fixed, we can apply Theorem \ref{Thm:Eps:Setup:Eps_Layer_Existence} to obtain a sequence $\{ (\hat{\rho}_{\epsilon,\delta},\hat{u}_{\epsilon,\delta}) \}_{\epsilon >0}$ of $\epsilon$ layer approximations satisfying the uniform bounds \eqref{Eq:Eps:Setup:Uniform_Bounds}. In the next section, we prove Proposition \ref{Prop:Del:Comp:Skorohod}, which allows us to switch probability spaces and use the recovery maps to define a new sequence $\{ (\rho_{\epsilon,\delta},u_{\epsilon,\delta}) \}_{\epsilon >0}$ and obtain compactness.  However, in order prove tightness of the pressure sequence, we need further estimates on the moments of the density.  This is proved in Proposition \ref{Prop:Del:Comp:Integrability_Gains}.  Next we extract a limit point $(\reh,\ueh)$ and work to verify our candidate is an $\epsilon$ layer approximation.  This involves both a preliminary identification step, Lemmas \ref{Lem:Del:Id:Continuity_Equation} and \ref{Lem:Del:Id:Momentum_Equation} and an elaborate procedure(modelled on the work of Lions \cite{lions1998mathematical}) for proving the strong convergence of the density in Lemmas \ref{Lem:Del:Strong:Weak_Continuity_Of_The_Effective_Viscous_Flux} and \ref{Lem:Del:Strong_Convergence_Of_Density}.
\subsection{$\epsilon \to 0$ Compactness Step}
%!TEX root = main.tex
The main goal of this subsection is to prove the following compactness result:
\begin{Prop} \label{Prop:Del:Comp:Skorohod} 
There exists a new probability space $(\Odh,\Fdh,\pdh)$, along with a sequence of recovery maps  $\{\widehat{T}_{\epsilon} \}_{\epsilon > 0}$ and limit points 
$ \left (\rdh,\udh,\overline{\sqrt{\rdh} \udh },\overline{\rdh^{\gamma}} + \delta \overline{\rdh^{\beta}},  \overline{\rdh \log{\rdh}}, \overline{\rdh \Div \udh} \right )$ \\
$$
\hat{T}_{\epsilon}:(\Odh,\Fdh,\pdh) \to (\Oeh,\Feh,\peh) 
$$
such that the following hold:
\begin{enumerate}
\item	\label{Prop:Del:Comp:Skorohod:Item:Recovery_Maps} 
 The original probability measure $\peh$ may be recovered from $\pdh$ by pushing forward $\widehat{T}_{\epsilon}$. 
\item  \label{Prop:Del:Comp:Skorohod:Item:Preserving_The_Equation} 
The new sequence $\{(\re,\ue)\}_{\epsilon> 0}$ defined by  $ \left(\re,\ue \right )=\left (\reh ,\ueh \right ) \circ \widehat{T}_{\epsilon}$ constitutes an $\epsilon$ layer approximation relative to the stochastic basis $(\Odh,\Fdh,\pdh,\{ \mathcal{F}^{t}_{\epsilon}\}_{t=0}^{T},W_{\epsilon} )$, where
$$ W_{\epsilon} := \widehat{W}_{\epsilon} \circ \widehat{T}_{\epsilon} \quad \quad \quad \F^{t}_{\epsilon} := \sigma \left (r_{t}\re, r_{t}\ue , r_{t}(\re \ue), r_{t}W_{\epsilon} \right ).$$
\item The following uniform bounds hold for all $p \geq 1$
\begin{equation} \label{Prop:Del:Comp:Skorohod:Item:Preserving_The_Bounds}
\sup_{\epsilon>0} \E^{\pdh} \bigg [ |\sqrt{\re}\ue |_{\LL{\infty}{2}}^{2p} + |\re|_{\LL{\infty}{\gamma}}^{\gamma p}+|\delta^{\frac{1}{\beta}}\re|_{\LL{\infty}{\beta}}^{\beta p} + |\ue|_{\LW{2}{1}{2}}^{2p} \bigg ] < \infty. \\
\end{equation}
%+ | \epsilon^{\frac{1}{2}} \nabla(\re^{\frac{\gamma}{2}} + \delta^{\frac{1}{2}}\re^{\frac{\beta}{2}})|_{\LL{2}{2}}^{2p}
\item \label{Prop:Del:Comp:Skorohod:Item:Pointwise_Convergence} The following convergences hold $\pdh$ a.s.
\begin{align} 
\label{Eq:Del:Comp:Skorohod:Pointwise_Convergence:Density}&\re \to \rdh \quad &\text{in} \quad  &C_{t}\big ( [L^{\beta}_{x}]_{w} \big ) \\
\label{Eq:Del:Comp:Skorohod:Pointwise_Convergence:Velocity}&\ue \to \udh \quad &\text{in} \quad &\LWw{2}{1}{2}  \\
\label{Eq:Del:Comp:Skorohod:Pointwise_Convergence:Momentum}& \re \ue \to \rdh \udh \quad & \text{in} \quad &\CLw{\frac{2\beta}{\beta+1}} \\
\label{Eq:Del:Comp:Skorohod:Pointwise_Convergence:Pressure}& \re^{\gamma}+\delta\re^{\beta} \to \overline{\rdh^{\gamma}} +\delta \overline{\rdh^{\beta}} \quad & \text{in} \quad & \big [L^{1+\frac{1}{\beta}}_{t,x} \big ]_{w} \\
%\label{eq:skPwe6}& \re \log{\re} \to \overline{\rdh \log{\rdh}} \quad & \text{in} \quad & \LLw{\infty}{2} \\
%\label{eq:skPwe7}& \re \Div \ue \to \overline{\rdh \Div \udh} \quad & \text{in} \quad & \LLw{2}{\frac{2\beta}{\beta+2}} \\
\label{Eq:Del:Comp:Skorohod:Pointwise_Convergence:Brownian_Motion}&W_{\epsilon} \to \widehat{W}_{\delta} \quad &\text{in} \quad &[C_{t}]^{\infty}. 
\end{align}
\item \label{Prop:Del:Comp:Skorohod:Item:Weak_Convergences} The following additional convergences hold
\begin{align}
\label{Eq:Del:Comp:Skorohod:Weak_Convergences:Kinetic_Energy}&\sqrt{\re}\ue \to \overline{\sqrt{\rdh} \udh } & \quad \text{in} \quad &L^{p}_{w} \big ( \Odh ; \LL{\infty}{2}\big ) \\
\label{Eq:Del:Comp:Skorohod:Weak_Convergences:Density}& \re \to \rdh & \quad \text{in} \quad &L^{p}_{w{*}} \big ( \Odh ; \LL{\infty}{\beta}\big ) \\
\label{Eq:Del:Comp:Skorohod:Weak_Convergences:Velocity}& \ue \to \udh & \quad \text{in}  \quad &L^{p}_{w} \big (\Odh ; \LW{2}{1}{2} \big ) \\
%\label{eq:EpsWeak5}& {\color{red}\re^{\gamma}+\delta\re^{\beta} \to \overline{\rdh^{\gamma}}  +\delta \overline{\rdh^{\beta}} }\quad & \text{in} \quad & L^{p}  ( \Odh ; \big [L^{1+\frac{1}{\beta}}_{t,x} \big ]_{w}  ) \\
\label{Eq:Del:Comp:Skorohod:Weak_Convergences:Log_Renorm_Term}& \re \Div \ue \to \overline{\rdh \Div \udh} \quad & \text{in} \quad & L^{p}_{w}  ( \Omega ; \LL{2}{\frac{2\beta}{\beta+2}} ) \\
\label{Eq:Del:Comp:Skorohod:Weak_Convergences:Entropy}&\re \log{\re} \to \overline{\rdh \log{\rdh}} & \quad \text{in} \quad &L^{p}_{w^{*}} \big ( \Odh ; \LL{\infty}{2}\big ). \\ 
\end{align}
\end{enumerate}
\end{Prop}
To prove the tightness, we start by proving the following integrability gains:
\begin{Prop}  \label{Prop:Del:Comp:Integrability_Gains} 
The following estimate holds uniformly in $\epsilon > 0$ for all $p \geq 1$
\begin{equation} \label{Eq:Del:Comp:Integrability_Gains}  
\sup_{\epsilon > 0} \E^{\peh} \bigg [ \big | \int_{0}^{T}\int_{D} \reh^{\gamma+1} + \delta \reh^{\beta + 1} dxdt \big |^{p} \bigg] < \infty.
\end{equation}
\end{Prop}
%%%%%%%%%% EPSILON LAYER INTEGRABILITY GAINS
%%%%%%%%%%
%%%%%%%%%%
%%%%%%%%%%
%%%%%%%%%%
% \label{eq:BogBounded}  \label{eq:BogDiv} \label{eq: Beta Constraints}
\begin{proof}  For regular domains $D$, one can define a sort of ``inverse divergence'', known as the Bogovski operator $\mathcal{B}$.  The properties of $\mathcal{B}$ are recalled in appendix lemma \ref{Thm:Appendix:Bogovoski}.  Define the following ``random test function''
\begin{equation}
\hat{\varphi}_{\epsilon} = \mathcal{B}[\reh - \frac{1}{|D|}\int_{D}\reh dx].
\end{equation}
The parabolic equation and the dirichlet boundary condition for the velocity yields the $\peh$ a.s. equality
\begin{equation}
\partial_{t}\hat{\varphi}_{\epsilon}= \epsilon \nabla \reh - \mathcal{B}[\Div(\reh \ueh)].
\end{equation}
Since the weak form of the momentum equation is stated in terms of deterministic test functions, ``testing'' $\hat{\varphi}_{\epsilon}$ requires an appeal to a version of the Ito product rule.  The equality below can be justified with a somewhat lengthy, but straightforward regularization argument(which we omit) in the spirit of \cite{MR2639716} or \cite{MR3098063}.  For all times $t \in [0,T]$ we have $\peh$ a.s.
\begin{equation} 
\begin{split}
&\int_{D} \reh \ueh(t) \cdot \hat{\varphi}_{\epsilon}(t)dx = \int_{D}m_{0,\delta} \cdot \hat{\varphi}_{\epsilon}(0)dx + \int_{0}^{t}\int_{D}  \reh \ueh \cdot \partial_{t}\hat{\varphi}_{\epsilon} +[\reh \ueh \tensor \ueh - 2\mu \nabla \ueh]: \nabla \hat{\varphi}_{\epsilon}dxds \\
&+\int_{0}^{t}\int_{D} (\reh^{\gamma} + \delta \reh^{\beta}-\lambda \Div \ueh) I : \nabla \hat{\varphi}_{\epsilon}- \epsilon \nabla \ueh \nabla \reh \cdot \hat{\varphi}_{\epsilon}dxds \\ 
&+ \sum_{k=1}^{\infty} \int_{0}^{t}\int_{D}\reh \sigma_{k,\delta}(\reh,\reh \ueh) \cdot \hat{\varphi}_{\epsilon} dxd\behk(s). 
\end{split}
\end{equation}
For our purposes, we will use the identity above at time $t=T$.  By definition of the Bogovski operator, we have
\begin{equation} 
\int_{0}^{T}\int_{D}(\reh^{\gamma} + \delta \reh^{\beta}) I : \nabla \hat{\varphi}_{\epsilon}dxds= \int_{0}^{T}\int_{D}(\reh^{\gamma}+\delta \reh^{\beta})(\reh - \frac{1}{|D|}\int_{D}\reh dx) dxds. 
\end{equation}
We can now rearrange and obtain
\begin{equation} 
\begin{split}
&\int_{0}^{T}\int_{D} \reh^{\gamma+1} + \delta \reh^{\beta + 1} dxdt = \int_{D} \big [\reh \ueh(T) \cdot \hat{\varphi}_{\epsilon}(T) -m_{0,\delta} \cdot\hat{\varphi}_{\epsilon}(0) \big ] \\
&+\int_{0}^{T}\int_{D} [2\mu \nabla \ueh+\lambda \Div \ueh I-\reh \ueh \tensor \ueh] : \nabla \hat{\varphi}_{\epsilon} dxds \\
& + \int_{0}^{T}\int_{D} (\reh^{\gamma}+\delta \reh^{\beta})\oint_{D}\reh +\epsilon \nabla \ueh \nabla \reh \cdot \hat{\varphi}_{\epsilon}dx + \int_{0}^{T}\int_{D} \big [ \mathcal{B}[\Div (\reh \ueh)]  -\epsilon \nabla \reh  \big ] \cdot \reh \ueh dxds\\
&-\sum_{k=1}^{\infty} \int_{0}^{T}\int_{D}\reh \sigma_{k}(\reh,\reh\ueh) \cdot \hat{\varphi}_{\epsilon} dxd\behk(s).  
\end{split}
\end{equation}
%%%% Momentum at a point
We proceed by estimating the $p^{th}$ moments on both sides of this equality.   In view of Theorem \ref{Thm:Appendix:Bogovoski}, the $\beta$ constraints \eqref{Eq:Tau:Artifical_Pressure_Beta_Constraints}, and the Sobolev embedding $W^{1,\beta}_{x} \hookrightarrow L^{\infty}_{x}$,  we obtain
 \begin{equation} 
\begin{split}
	& \E^{\peh} \left [ \big | \int_{D} (\reh \ueh(T)-m_{0,\delta})\hat{\varphi}_{\epsilon}(T)dx \big |^{p} \right] \\
	&\leqs     \E^{\peh}[ (|\reh \ueh|_{\LL{\infty}{\frac{2\gamma}{\gamma+1}}}^{p}+|m_{0,\delta}|_{L^{\frac{2\gamma}{\gamma+1}}_{x}}^{p})  |\reh|_{\LL{\infty}{\beta}}^{p}] \\
	&\leqs  \E^{\peh}[  |\reh \ueh |_{\LL{\infty}{\frac{2\gamma}{\gamma+1}}}^{2p}+|m_{0,\delta}|_{L^{\frac{2\gamma}{\gamma+1}}_{x}}^{2p}]^{\frac{1}{2}} \E^{\peh}[|\reh|_{\LL{\infty}{\beta}}^{2p}]^{\frac{1}{2}}.
\end{split}
\end{equation}
%%%% Dissipation
Using Theorem \ref{Thm:Appendix:Bogovoski} and \eqref{Eq:Tau:Artifical_Pressure_Beta_Constraints}, we obtain
  \begin{equation} 
 \begin{split}	
 	& \E^{\peh} \bigg [ \big | \int_{0}^{T}\int_{D} 2 \mu \nabla \ueh+\lambda \Div \ueh : \nabla \hat{\varphi}_{\epsilon} dxds \big |^{p} \bigg ]
	\leqs \E^{\peh} \big [ | \ueh |_{\LW{2}{1}{2}}^{p} |\nabla \hat{\varphi}_{\epsilon}|_{\LL{2}{2}}^{p} \big ] \\
	&\leqs  \E^{\peh} \big [ |\ueh |_{\LW{2}{1}{2}}^{2p} \big ]^{\frac{1}{2}}  \E^{\peh} \big [ |\reh |_{\LL{\infty}{\beta}}^{2p} \big ]^{\frac{1}{2}}.  
\end{split}
\end{equation}
%%%% Tensor Product
Note that \eqref{Eq:Tau:Artifical_Pressure_Beta_Constraints} implies the embedding $L^{\beta}_{x}  \hookrightarrow L^{\frac{d \beta }{2\beta-d}}_{x}$  so applying Theorem \ref{Thm:Appendix:Bogovoski} yields 
\begin{equation} \label{eq:IntGainT3}
\begin{split}
&\E^{\peh} \bigg [ \big |\int_{0}^{T}\int_{D} \reh \ueh \tensor \ueh : \nabla \hat{\varphi}_{\epsilon}dxds \big |^{p}\bigg ]
	\leqs \E^{\peh } \bigg[ \big |\int_{0}^{T} |\reh |_{L^{\beta}_{x}}|u_{\epsilon}|_{L^{\frac{2d}{d-2}}}^{2}|\nabla \hat{\varphi}_{\epsilon}|_{L^{\frac{d \beta}{2\beta -d}}}ds \big |^{p} \bigg] \\
&\leqs   \E^{\peh} \bigg [ \big |\int_{0}^{T} |\reh|_{L^{\beta}_{x}}^{2}|\ueh|_{L^{\frac{2d}{d-2}}}^{2}ds \big |^{p} \bigg ] \leqs  \E^{\peh}[ |\reh|_{\LL{\infty}{\beta}}^{2p}|\ueh|_{\LW{2}{1}{2}}^{2p}] \\
&\leqs \E^{\peh}[ |\reh|_{\LL{\infty}{\beta}}^{4p}]^{\frac{1}{2}}  \E^{\peh}[ \ueh |_{\LW{2}{1}{2}}^{4p}]^{\frac{1}{2}}. 
\end{split}
\end{equation}
Applying H\"older yields
%%%%%% Pressure times mean value
\begin{equation} 	
	\E^{\peh} \bigg [ \big | \int_{0}^{T}(\oint_{D}\reh(s)dx)\int_{D} (\delta\reh^{\beta} + \reh^{\gamma}) dxds \big |^{p} \bigg ]
	\leqs  \E^{\peh}[\delta^{p}|\reh|_{\LL{\infty}{\beta}}^{(\beta + 1)p}+|\reh|_{\LL{\infty}{\gamma}}^{(\gamma + 1)p}]. 
\end{equation}
%%%% Divergence term from continuity eqn
Defining $r$ by the relation $\frac{1}{r} =\frac{1}{2}+\frac{1}{d}-\frac{1}{\beta}$ and applying Theorem \ref{Thm:Appendix:Bogovoski}, then using H\"older, the embedding  $L^{\frac{d \beta }{2\beta-d}}_{x} \hookrightarrow L^{\beta}_{x}$, and \eqref{Eq:Tau:Artifical_Pressure_Beta_Constraints}, we obtain
\begin{equation} 
\begin{split}
	&\E^{\peh} \bigg [ \big | \int_{0}^{T}\int_{D} \reh \ueh \cdot \mathcal{B}[\Div(\reh \ueh)]dxdt \big |^{p} \bigg ]
	\leqs \E^{\peh}  \bigg [ \big | \int_{0}^{T} |\reh|_{L^{\beta}_{x}}|\ueh|_{L^{\frac{2d}{d-2}}_{x}} |\mathcal{B}[\Div(\reh \ueh)]|_{L^{r}_{x}}dt \big |^{p} \bigg ]\\
	&\leqs  \E^{\peh} \bigg [ \big | \int_{0}^{T} |\reh|_{L^{\beta}_{x}}|\ueh|_{L^{\frac{2d}{d-2}}_{x}} |\reh \ueh|_{L^{r}}dt \big |^{p} \bigg ]
	\leqs \E^{\peh}  \bigg [ \big | \int_{0}^{T} |\reh|_{L^{\beta}_{x}}|\ueh|_{L^{\frac{2d}{d-2}}_{x}}^{2} |\reh|_{L^{\frac{\beta d}{2\beta -d}}}dt \big |^{p} \bigg ] \\
	&\leqs \E^{\peh} [ |\reh|_{\LL{\infty}{\beta}}^{4p}]^{\frac{1}{2}} \E^{\peh} [ |\ueh|_{\LW{2}{1}{2}}^{4p}]^{\frac{1}{2}}.
\end{split}
\end{equation}
%%%%% Energy correction term
Using again the Sobolev embedding of $W^{1,\beta}_{x} \hookrightarrow L^{\infty}_{x}$, we can estimate the energy correction as follows:
\begin{equation} 
\begin{split}
&\E^{\peh} \bigg [ \big | \int_{0}^{T}\int_{D} \epsilon \nabla \ueh \nabla \reh \cdot \hat{\varphi}_{\epsilon}dxdt \big |^{p} \bigg ] \\
&\leqs \epsilon^{\frac{p}{2}}\E^{\peh}[|\ueh|_{\LW{2}{1}{2}}^{3p}]^{\frac{1}{3}} \E^{\peh}[ |\sqrt{\epsilon} \nabla \reh |_{\LL{2}{2}}^{3p}]^{\frac{1}{3}} \E^{\peh}[ |\reh|_{\LL{\infty}{\beta}}^{3p}]^\frac{1}{3}. 
\end{split}
\end{equation}
For the artificial viscosity, we use H\"older followed by \eqref{Eq:Tau:Artifical_Pressure_Beta_Constraints} to obtain
%%%% Laplacian term
\begin{equation} 
\begin{split}
	& \E^{\peh} \bigg [ \big | \int_{0}^{T}\int_{D}\epsilon \reh \ueh \cdot \nabla \reh dxdt \big |^{p} \bigg ]
	%\leq \epsilon^{p} \E^{\peh} \bigg [ \big |\int_{0}^{T} |\reh|_{L^{r}_{x}}|\ueh|_{L^{\frac{2d}{d-2}}} |\nabla \reh |_{L^{2}_{x}}dxdt \big |^{p} \bigg ] \\
	\leqs \epsilon^{p/2} \E^{\peh} [|\sqrt{\epsilon}\nabla \reh |_{\LL{2}{2}}^{p}|\reh|_{\LL{\infty}{\beta}}^{p}|\ueh|_{\LW{2}{1}{2}}^{p} ] \\
	& \leqs \epsilon^{p/2} \E^{\peh}[ |\sqrt{\epsilon}\nabla \reh|_{\LL{2}{2}}^{3p}]^{\frac{1}{3}} \E^{\peh} [\reh|_{\LL{\infty}{\beta}}^{3p}]^{\frac{1}{3}} \E^{\peh}[|\ueh|_{\LW{2}{1}{2}}^{3p} ]^{\frac{1}{3}} .
	%\quad \quad \frac{1}{r} + \frac{1}{2}-\frac{1}{d} + \frac{1}{2} = 1 \quad \to \quad r=d<\beta \\
\end{split}
\end{equation}
%%%% Stochastic integrals
Finally, we use the BDG inequality, the summability Hypothesis \ref{Hyp:color}, \eqref{Eq:Tau:Artifical_Pressure_Beta_Constraints}, and the Sobolev embedding of $W^{1,\beta}_{x} \hookrightarrow L^{\infty}_{x}$ to estimate the series of stochastic integrals as follows:
%%%%%% Stochastic Integral
\begin{equation} 
\begin{split}
&\E^{\peh} \bigg [\big| \sum_{k=1}^{\infty} \int_{0}^{T}\int_{D}\reh \sigma_{k,\delta}(\reh,\reh\ueh) \cdot \hat{\varphi}_{\epsilon} dxd\behk(s) \big |^{p}\bigg ] \\
&\leqs  \E^{\peh} \bigg [ \big |\sum_{k=1}^{\infty} \int_{0}^{T}(\int_{D} \reh \sigma_{k,\delta}(\reh,\reh\ueh) \cdot \hat{\varphi}_{\epsilon} dt\big |^{\frac{p}{2}} \bigg ]   \\
&\leqs \big (\sum_{k=1}^{\infty}|\sigma_{k}|_{L^{\gamma'}_{x}(L^{\infty}_{\rho,m})}^{2}\big )^{\frac{p}{2}} \E^{\peh} \bigg [ \big | \int_{0}^{T} |\reh|_{L^{\gamma}_{x}}^{2}|\hat{\varphi}_{\epsilon}|_{L^{\infty}_{x}}^{2}dt \big |^{p/2} \bigg ] \\ 
&\leqs  \E^{\peh} \bigg [ \big | \int_{0}^{T} |\reh|_{L^{\gamma}_{x}}^{2}| \reh|_{L^{\beta}_{x}}^{2}dt \big |^{p/2} \bigg ]  
\leqs \E^{\peh}[ |\reh|_{\LL{\infty}{\beta}}^{2p} ].
\\
\end{split}
\end{equation}
Hence, appealing to \eqref{Eq:Eps:Setup:Uniform_Bounds}, we may close each estimate and obtain \eqref{Eq:Del:Comp:Integrability_Gains} as claimed. \\
\end{proof}	

We now proceed to a proof of the tightness.  The proof is similar to the $n$ layer tightness proof, the main difference being that it is less clear how to estimate the weak continuity modulus of the momentum due to limited uniform estimates for the energy correction term.  Our strategy is to temporarily avoid this issue by proving simultaneously the tightness of the momentum minus the energy correction and the tightness of the momentum in a weaker space, then putting these together after an application of the enhanced Skorohod theorem in order to deduce the tightness of the momentum in $\CLw{\frac{2\beta}{\beta+1}}$.  Towards this end, enumerate a smooth, dense subset $\{ \phi_{k}\}_{k \geq 1}$ of $[L^{\frac{2\beta}{\beta+1}}_{x}]^{d}$.  Define the sequence of random variables $\{ (\hat{X}_{\epsilon},\hat{Y}_{\epsilon}) \}_{\epsilon > 0}$ via 
\begin{equation}
\begin{split}
	&\hat{X}_{\epsilon} = \left (\reh,\ueh,\reh \ueh, \reh^{\gamma}+\delta \reh^{\beta}, \{ \behk  \}_{k \geq1} \right )  \\
	%&\hat{Y}_{\epsilon} = \left (\reh \text{log}(\reh),{ \color{red} \reh \Div{\ueh} } \right) \\
	& \hat{Y}_{\epsilon} =  \left \{ \langle \reh \ueh(\mathbf{\cdot})\, ,\, \phi_{k} \rangle -\epsilon\int_{0}^{\mathbf{\cdot}}\langle \nabla \ueh \nabla \reh , \phi_{k} \rangle ds  \right \}_{k=1}^{\infty}.
\end{split}
\end{equation}
This sequence induces a measure on the space $E \times F$ endowed with its natural topology.
\begin{equation} \label{Eq:Del:Comp:Tightness_Space}
\begin{split}
	& E =  C_{t}\big ( [L^{\beta}_{x}]_{w} \big ) \times [L^{2}_{t}(H^{1}_{0,x})]_{w}  \times [\LL{\infty}{\frac{2\beta}{\beta+1}} ]_{w^{*}} \times \big [L^{1+\frac{1}{\beta}}_{t,x} \big ]_{w}  \times [C_{t}]^{\infty}  \\
	& F  = [C_{t}]^{\infty} 
\end{split}
\end{equation}
We may now establish the following tightness result.
\begin{Lem} \label{Lem:Del:Comp:Tightness}
The sequence of induced measures $\{\hat{\p}_{\epsilon} \circ (\hat{X}_{\epsilon},\hat{Y}_{\epsilon})^{-1}\}_{\epsilon > 0}$ is tight on $E \times F $.
\end{Lem}
\begin{proof} \quad We may consider each component separately, and we start by noting that the method of Lemma \ref{Lem:Eps:Comp:Tightness} can be repeated in order to deduce the tightness of $\peh \circ (\ueh,\widehat{W}_{\epsilon})^{-1}$, where we have denoted $\widehat{W}_{\epsilon} = \{ \hat{\beta}_{\epsilon}^{k}\}_{k=1}^{\infty}$.  The tightness of the sequence $\{\peh \circ (\reh \ueh, \reh^{\gamma}+\delta \reh^{\beta})^{-1}\}_{\epsilon >0}$ follows easily from Banach Alaogolu in view the uniform bounds \eqref{Eq:Eps:Setup:Uniform_Bounds} and the integrability gains \eqref{Eq:Del:Comp:Integrability_Gains}.  To treat the density sequence use the parabolic equation to estimate
\begin{align*}
&|\partial_{t}\reh|_{\LW{\infty}{-2}{\frac{2\beta}{\beta+1}}} \leq |\reh \ueh|_{\LL{\infty}{\frac{2\beta}{\beta+1}}} + \epsilon |\reh|_{\LL{\infty}{\frac{2\beta}{\beta+1}}} \\
&\leq  |\sqrt{\reh} \ueh|_{\LL{\infty}{2}} |\reh|_{\LL{\infty}{\beta}} + \epsilon |\reh|_{\LL{\infty}{\frac{2\beta}{\beta+1}}}  
\end{align*}
This bound can be used to show that the sequence $\{\peh \circ \reh^{-1}\}_{\epsilon >0}$ assigns high probability to sets of the form
$$
\{f \in \LL{\infty}{\beta} \mid |f|_{\LL{\infty}{\beta}} + |\partial_{t}f|_{\LW{\infty}{-2}{\frac{2\beta}{\beta+1}}} \leq M \}. $$
These are compact in $\CLw{\beta}$ view of Corollary \ref{Cor:Appendix:Time_Derivative_In_A_Negative_Sobolev_Space}.  To treat the sequence $\{ \peh \circ \hat{Y}_{\epsilon}^{-1} \}_{\epsilon >0}$, we use the uniform bounds \eqref{Eq:Eps:Setup:Uniform_Bounds} and proceed as in Lemma \ref{Lem:Eps:Comp:Tightness} to obtain the uniform estimate
\begin{equation} 
	 \sup_{\epsilon > 0} \E^{\peh} |\langle \reh\ueh(t) -\reh\ueh(s) , \phi_{k} \rangle -\epsilon\int_{s}^{t} \langle \nabla \ueh \nabla \reh(r) , \phi_{k} \rangle dr |^{p} \leqs  |\phi_{k}|_{C^{1}_{x}}^{p} |t-s|^{\frac{p}{2}}.
\end{equation}
For any $\gamma < \frac{1}{2}$ we may choose a $p$ large enough to ensure  for each $k$ the uniform holder estimates
$$
 \sup_{\epsilon>0} \E^{\peh} \left [ |\hat{Y}_{\epsilon}^{k}|_{C^{\gamma}_{t}}^{p} \right ]  \leqs |\phi|_{C^{1}_{x}}^{p}.
$$
We may now complete the proof by showing that the sequence of induced measures assign arbitrarily high probability to sets of the form
$$\prod_{j=1}^{\infty} \{ f \in C_{t}\mid |f|_{C^{\gamma}_{t}} \leq M(2^{j}|\phi|_{C^{1}_{x}})^{\frac{1}{p}} \}.
$$
These are compact in view of Tychonoff and Arzela Ascoli.  This completes the tightness proof.
\end{proof}
Now we are in a position to finish the proof of our compactness step.
\begin{proof}[Proof of Proposition \ref{Prop:Del:Comp:Skorohod} ] \quad In view of Remark \ref{Rem:Appendix:Examples_of_Jakubowksi_Spaces}, $E \times F$ is a Jakubowski space.  Thus, we may apply Theorem \ref{Thm:Appendix:Jakubowksi_Skorohod} to obtain a new probability space $(\Odh,\Fdh,\pdh)$, a collection of recovery maps $\{\hat{T}_{\epsilon} \}_{\epsilon > 0}$, and a limiting random variable $(\hat{X}_{\delta},\hat{Y}_{\delta})=(\rdh,\hat{m_{\delta}},\overline{\rdh^{\gamma}} + \delta \overline{\rdh^{\beta}}, \widehat{W}_{\delta}, \hat{Y}_{\delta})$.  Note that we may use the recovery maps to identify
\begin{align*}
&(\re \ue, \re^{\gamma}+\delta \re^{\beta},  \{ \langle \re\ue(\cdot)-\int_{0}^{\cdot}\nabla \ue \nabla \re ds, \phi_{k} \rangle \}_{k=1}^{\infty}) \\
&= (\reh \ueh, \reh^{\gamma}+\delta \reh^{\beta},\hat{Y}_{\epsilon}) \circ \hat{T}_{\epsilon} \to (\hat{m}_{\delta},\overline{\rdh^{\gamma}} + \delta \overline{\rdh^{\beta}}, \hat{Y}_{\delta}).
\end{align*}
Proceeding as in the proof of Proposition \ref{Prop:Eps:Skorohod}, we may obtain parts \ref{Prop:Del:Comp:Skorohod:Item:Recovery_Maps} and  \ref{Prop:Del:Comp:Skorohod:Item:Preserving_The_Equation} along with \ref{Prop:Del:Comp:Skorohod:Item:Preserving_The_Bounds}.  Our appeal to Theorem \ref{Thm:Appendix:Jakubowksi_Skorohod} immediately gives \eqref{Eq:Del:Comp:Skorohod:Pointwise_Convergence:Density},\eqref{Eq:Del:Comp:Skorohod:Pointwise_Convergence:Velocity},\eqref{Eq:Del:Comp:Skorohod:Pointwise_Convergence:Pressure}, and \eqref{Eq:Del:Comp:Skorohod:Pointwise_Convergence:Brownian_Motion}.  Moreover, the uniform bounds \eqref{Eq:Eps:Setup:Uniform_Bounds} and the Banach Alagolu theorem allow us to obtain part \ref{Prop:Del:Comp:Skorohod:Item:Weak_Convergences}.  The only part requiring more justification is \eqref{Eq:Del:Comp:Skorohod:Pointwise_Convergence:Momentum}, which we now explain.  \\
 Let us write $\hat{Y}_{\delta}=\{ \hat{m}_{\delta}^{k} \}_{k=1}^{\infty}$ the limit point obtained for the sequence $\{ Y_{\epsilon} \}_{\epsilon > 0}$.  Note that for all $k$ and all $p \geq 1$ 
 $$
\lim_{\epsilon \to 0}\E^{\pdh} \left [  \sup_{t \in [0,T]} \left | \epsilon\int_{0}^{t} \langle \nabla \ue \nabla \re(r) , \phi_{k} \rangle dr \right |^{p} \right ] = 0. 
 $$
Combining this with the pointwise convergence of $Y_{\epsilon} \to \hat{Y}_{\delta}$ yields the following $\pdh$ a.s. limits
\begin{align*}
 \langle \re \ue, \phi_{k} \rangle  \, \to \hat{m}_{\delta}^{k} & \quad \text{in} \quad C_{t}. 
\end{align*}
However, we also have the $\pdh$ a.s. convergence $\rho_{\epsilon}u_{\epsilon} \to \rho u$ in $\ [ \LL{\infty}{\frac{2\beta}{\beta+1}}]_{w*}$, which identifies a candidate limit at the outset. Arguing similarly as in the corresponding proof of Proposition \ref{Prop:Eps:Skorohod}, we may obtain \eqref{Eq:Del:Comp:Skorohod:Pointwise_Convergence:Momentum}.
\end{proof}
\subsection{Preliminary Identification Step}
%!TEX root = main.tex
We now define the filtration $\{ \mathcal{F}_{t}^{\delta}\}_{\delta > 0}$ using the restriction operators with respect to the spaces dictated by Propostion \ref{Prop:Del:Comp:Skorohod}. 
\begin{equation}
\hat{\mathcal{F}}_{t}^{\delta} = \sigma \big ( r_{t}\rdh , r_{t} \udh , r_{t} \big ( \rdh \udh \big ), r_{t}\widehat{W}_{\delta}\big ) .
\end{equation}

\begin{Lem} \label{Lem:Del:Id:Continuity_Equation}
The pair $(\rdh,\udh)$ satisfies the continuity equation, \ref{Eq:Del:Setup:Continuity_Equation} of Definition \ref{Def:Del:Setup:Delta_Layer_Approximation}.  
\end{Lem}
\begin{proof}  In view of part \ref{Prop:Del:Comp:Skorohod:Item:Preserving_The_Equation} of Proposition \ref{Prop:Del:Comp:Skorohod}, the pair $(\re,\ue)$ satisfy the same parabolic equation on the new probability space $(\Omega_{\delta},\F_{\delta},\pd)$.  Hence, in view of \eqref{Eq:Del:Comp:Skorohod:Pointwise_Convergence:Density} and \eqref{Eq:Del:Comp:Skorohod:Pointwise_Convergence:Momentum}, we may pass $\epsilon \to 0$ in the weak form of the parabolic equation for $(\re,\ue)$ to obtain the desired result $\pdh$ almost surely.  Indeed, the only additional detail is to argue that the viscous regularization tends to zero.  Recall the energy identity used in the $n \to \infty$ density upgrade.  In view of the representation of $(\re,\ue)$ in terms of the recovery maps, this identity holds on the new probability space, so estimating the moments on both sides(neglecting the initial data, which are controlled) yields
\begin{align*}
&\E^{\pdh} \left [ |\sqrt{\epsilon} \nabla \re |_{\LL{2}{2}}^{2p}\right ] \leqs  \E^{\pdh} \left [|\re^{2}\Div \ue|_{\LL{1}{1}}^{p} \right ]
\leq \E^{\pdh} \left [|\re|_{\LL{4}{4}}^{2p}|\ueh|_{L^{2}_{t}(H^{1}_{0,x})}^{p} \right ] \\
&\leq \E^{\pdh} \left [|\re|_{\LL{4}{4}}^{4p} \right ]^{\frac{1}{2}}\E^{\pdh} \left [|\ue|_{L^{2}_{t}(H^{1}_{0,x})}^{2p} \right ]^{\frac{1}{2}}.
\end{align*}
Hence, we may use the uniform bounds obtained in \eqref{Eq:Eps:Setup:Uniform_Bounds} and estimate for $\phi \in C_{c}^{\infty}(D)$
$$ 
\E^{\pdh} \left [|\epsilon \int_{0}^{T}\int_{D} |\nabla \re| |\nabla \phi| dxds |^{p} \right ] \leq \epsilon^{\frac{p}{2}} |\nabla \phi|_{L^{2}} \E^{\pdh} \left [ |\sqrt{\epsilon} \nabla \re|_{\LLs{2}}^{p} \right ] \to 0.
$$
\end{proof}
\begin{Lem} \label{Lem:Del:Id:Momentum_Equation}
The pair $(\rdh,\udh)$ satisfies the momentum equation \eqref{Eq:Del:Setup:Momentum_Equation} from Definition \ref{Def:Del:Setup:Delta_Layer_Approximation}, with a pressure law $\overline{\rdh^{\gamma}} + \delta \overline{\rdh^{\beta}}$.
\end{Lem}
\begin{proof}  Let us proceed with the method in Lemma \ref{Lem:Eps:Id:Momentum_Equation} by defining for each $\phi \in C^{\infty}(D)$ the $\{ \hat{\mathcal{F}}_{t}^{\delta} \}_{t=0}^{T}$ continuous, adapted stochastic process $ \{ \hat{M}^{\delta}_{t}(\phi) \}_{t=0}^{T}$ via 
\begin{align*}
&\hat{M}_{t}^{\delta}(\phi)= \int_{D} \rdh \udh(t) \cdot \phi \, dx  - \int_{D} \hat{m}_{0}^{\delta}\cdot \phi \, dx - \int_{0}^{t}\int_{D} [\rdh \udh \tensor \udh-2\mu \nabla \udh ]:\nabla \phi dxds \\
&+\int_{0}^{t}\int_{D} [ (\overline{\rdh^{\gamma}}+\delta \overline{\rdh^{\beta}}-\lambda \Div \udh) I] : \nabla \phi dxds .
\end{align*}
Note that by appealing again to \eqref{Eq:Eps:Setup:Uniform_Bounds}, we have 
\begin{align*} 
& \E^{\pdh} \bigg [ \big | \int_{0}^{t}\int_{D} \epsilon \nabla \ue \nabla \re \cdot \phi dxds \big |^{p} \bigg ] \\
&\leq \epsilon^{\frac{p}{2}}|\phi|_{L^{\infty}_{x}}\E^{\pdh} \bigg [ |\ue|_{\LW{2}{1}{2}}^{2p} \bigg ]^{\frac{1}{2}} \E^{\pdh} \bigg [ |\sqrt{\epsilon}\nabla \re|_{\LL{2}{2}}^{2p}\bigg ]^{\frac{1}{2}} \to 0 .
\end{align*}
Combining this observation with Proposition \ref{Prop:Del:Comp:Skorohod} yields all the necessary ingredients to proceed with the method in Lemma \ref{Lem:Eps:Id:Momentum_Equation} and identify 
\begin{align*}
\hat{M}_{t}^{\delta}(\phi) = \sum_{k=1}^{\infty} \int_{0}^{t}\int_{D}\rdh \sigma_{k}(\rdh,\rdh \udh) \cdot \phi dxd\bdhk(s).
\end{align*}
\end{proof}
\subsection{Strong Convergence of the Density}
%!TEX root = main.tex
Now to proceed to the proof of the strong convergence of the density.  The first step is the following weak continuity result:
\begin{Lem} \label{Lem:Del:Strong:Weak_Continuity_Of_The_Effective_Viscous_Flux} Let $K \subset \subset D$ be arbitrary, then the weak continuity of the effective viscous pressure holds on average, that is:
\begin{align*}
&\lim_{\epsilon \to 0} \E^{\pdh}[\int_{0}^{T}\int_{K}\big ((2\mu + \lambda)\Div \ue -\re^{\gamma} - \delta \re^{\beta} \big)\re dxdt] \\
&=\E^{\pdh}[\int_{0}^{T}\int_{K}\big ((2\mu + \lambda)\Div \udh -\rdhgb -\delta \rdhbb \big)\rdh dxdt]. 
\end{align*}
\end{Lem}
\begin{proof}
Recall that $\mathcal{A}= \nabla \Delta^{-1}$, where the inverse laplacian is understood to be well defined on compactly supported distributions in $\R^{d}$.  Let $\eta$ be a bump function supported in $D$.  Define the following two random test functions: $\varphi_{\epsilon} = \eta \mathcal{A} \left [\eta \re \right ]$ and $\hat{\varphi}_{\delta} = \eta \mathcal{A}\left [\eta \rdh \right ]$.  Using the parabolic equation for $\re$ driven by $\ue$ and the transport equation for $\rdh$ driven by $\udh$, we may check that
\begin{align*}
& \partial_{t} \varphi_{\epsilon} = \eta \mathcal{A}\circ \Div(\eta(\epsilon \nabla \re - \re \ue)) +\eta \mathcal{A}\left [\nabla \eta \cdot ( \re \ue - \epsilon \nabla \re) \right ] \\
& \partial_{t} \hat{\varphi}_{\delta} = - \eta \mathcal{A}\circ \Div(\eta \rdh \udh)+\eta \mathcal{A}\left [\nabla \eta \cdot \rdh \udh  \right ] 
\end{align*}
Using the momentum equation for $(\re,\ue)$ and $(\rdh,\udh)$ we may use the Ito product rule twice(see the remarks in Proposition \ref{Prop:Del:Comp:Integrability_Gains} regarding justification) to find the evolution of $\re \ue \cdot \varphi_{\epsilon}$ and $\rdh \udh \cdot  \hat{\varphi}_{\delta}$.  The first application yields the $\pdh$ a.s. equality
\begin{align*}
&\int_{D} \re \ue(T) \cdot \varphi_{\epsilon}(T)dx = \int_{D}m_{0,\delta} \cdot \varphi_{\epsilon}(0)dx + \int_{0}^{T}\int_{D} [ \re \ue \cdot \partial_{t} \varphi_{\epsilon} + [\re \ue \tensor \ue - 2\mu \nabla \ue ] : \nabla \varphi_{\epsilon}]dxdt \\
&+\int_{0}^{t}\int_{D}(-\lambda \Div \ue +\re^{\gamma} + \delta \re^{\beta})I ] : \nabla \varphi_{\epsilon}]-\epsilon \nabla \ue \nabla \re \cdot \varphi_{\epsilon}dxds \\
&+ \sum_{k=1}^{\infty} \int_{0}^{T}\int_{D}\re \sigma_{k,\delta}(\re,\re \ue) \cdot \varphi_{\epsilon} dxd\bdhk(s). 
\end{align*}
The second application yields the $\pdh$ a.s. equality
\begin{align*}
&\int_{D} \rdh \udh(T) \cdot \hat{\varphi}_{\delta}(T)dx = \int_{D}m_{0,\delta} \cdot \hat{\varphi}_{\delta}(0)dx + \int_{0}^{T}\int_{D} [ \rdh \udh \cdot \partial_{t} \hat{\varphi}_{\delta} +[\rdh \udh \tensor \udh - 2\mu \nabla \udh ]dxdt \\
&+\int_{0}^{T}\int_{D} [(-\lambda \Div \udh +\overline{\rdh^{\gamma}} + \delta \overline{\rdh^{\beta}})I ] : \nabla \hat{\varphi}_{\delta} ]dxdt \\
&+ \sum_{k=1}^{\infty} \int_{0}^{T}\int_{D}\rdh  \sigma_{k,\delta}(\rdh,\rdh \udh) \cdot \hat{\varphi}_{\delta} dxd\bdhk(s). 
\end{align*}
Note that
\begin{align*}
& [\re \ue \tensor \ue - 2\mu \nabla \ue +(-\lambda \Div \ue +\re^{\gamma} + \delta \re^{\beta})I ] : \nabla \varphi_{\epsilon} \\
&= [\re \ue \tensor \ue-2 \mu \nabla \ue +(-\lambda \Div \ue +\re^{\gamma}+\delta \re^{\beta})I] : \nabla \eta \tensor \mathcal{A}\left [\eta \re \right ] \\
&+\eta \re \ue \tensor \ue  : \nabla \mathcal{A}[\eta \re ]-2\mu \eta \nabla \ue : \nabla \mathcal{A}[\eta \re ] +\eta(\re^{\gamma}+\delta \re^{\beta}-\lambda \Div \ue) \re .
 \end{align*}
 Moreover, integrating by parts twice(justifying on a smooth approximation) reveals
 \begin{align*}
 &\int_{0}^{T}\int_{D}- 2\mu \eta \nabla \ue : \nabla \mathcal{A}[\eta \re ] = \int_{0}^{T}\int_{D}-2\mu \eta^{2}\Div \ue \re dxds \\
 &+\int_{0}^{T}\int_{D} \ue \cdot [\mathcal{A}(\eta \re)\nabla \eta - \nabla \eta \eta \re] dxds
 \end{align*}
Also note that $\varphi_{\epsilon}(0)=\hat{\varphi}_{\delta}(0)$.  Taking expectation(so that the stochastic integrals vanish) of both Ito product rules above yields two fundamental identities:
\begin{equation}
\begin{split}
&\E^{\pdh} \left [ \int_{0}^{T}\int_{D} \eta^{2} [(2\mu + \lambda)\Div \ue - \re^{\gamma}-\delta \re^{\beta}] \re dxds \right ] \\
&=I^{0}+ I^{A,\epsilon}_{1}+I^{A,\epsilon}_{2}+ I^{C,\epsilon}_{1}+I^{C,\epsilon}_{2}+I^{C,\epsilon}_{3}+ I^{P,\epsilon}_{1}+I^{P,\epsilon}_{2}. 
\end{split}
\end{equation}
\begin{equation}
\E^{\pdh} \left [ \int_{0}^{T}\int_{D} \eta^{2} [(2\mu + \lambda)\Div \udh - \overline{\rdh^{\gamma}}-\delta \overline{\rdh^{\beta}}] \rho dxds \right ] =I^{0}+  I^{C}_{1}+I^{C}_{2}+I^{C}_{3}+ I^{P}_{1}+I^{P}_{2}. 
\end{equation}
Our labeling convention should be interpreted as follows.  The terms $I^{A,\epsilon}_{1},I^{A,\epsilon}_{2}$ are ``artificial '' and will tend to zero as $\epsilon \to 0$, $I^{C,\epsilon}_{1},I^{C,\epsilon}_{2}, I^{C,\epsilon}_{3}$ are lower order``''cutoff" terms arising due to the localization of the estimate, and $I^{P,\epsilon}_{1},I^{P,\epsilon}_{2}$ are the principal terms arising irregardless of the boundary conditions.  More precisely, the contribution at the $\epsilon$ layer yields
\begin{align}
&I^{0} = \E^{\pdh} \left [\int_{D} \eta m_{0,\delta} \cdot \mathcal{A}[\eta \rho_{0,\delta}]dx \right ] \\
&I^{A,\epsilon}_{1} = \E^{\pdh}\left [ \int_{0}^{T}\int_{D} \epsilon \eta \re \ue \cdot \mathcal{A}[ \Div(\eta \nabla \re) -\nabla \eta \cdot \nabla \re] dxds \right ]\\
&I^{A,\epsilon}_{2} = -\E^{\pdh} \left [\int_{0}^{T}\int_{D} \epsilon \nabla \ue \nabla \re \cdot \varphi_{\epsilon} dxds \right ] \\
&I^{C,\epsilon}_{1} = \E^{\pdh} \left [\int_{0}^{T}\int_{D} [\re \ue \tensor \ue -2 \mu \nabla \ue ]: \nabla \eta \tensor \mathcal{A}\left [\eta \re \right ]dxds \right ] \\
&+\E^{\pdh} \left [\int_{0}^{T}\int_{D}(-\lambda \Div \ue +\re^{\gamma}+\delta \re^{\beta})I] : \nabla \eta \tensor \mathcal{A}\left [\eta \re \right ]dxds \right ] \\
&I^{C,\epsilon}_{2} = \E^{\pdh} \left [\int_{0}^{T}\int_{D} \re \ue \cdot \mathcal{A}[\nabla \eta \cdot \re \ue] dxds \right ] \\
&I^{C,\epsilon}_{3} = \E^{\pdh} \left [\int_{0}^{T}\int_{D} \ue \cdot [\mathcal{A}(\eta \re)\nabla \eta - \nabla \eta \eta \re] dxds \right ]\\
&I^{P,\epsilon}_{1} = \E^{\pdh} \left [\int_{0}^{T}\int_{D} \eta \left [ \re \ue \tensor \ue : \nabla \mathcal{A}[\eta \re] - \re \ue \cdot \mathcal{A} \circ \Div(\eta \re \ue) \right ]dxds \right ] \\
&I^{P,\epsilon}_{2} = -\E^{\pdh} \left [ \int_{D} \re \ue(T) \cdot \varphi_{\epsilon}(T)dx \right ]
\end{align}
In the limit as $\epsilon \to 0$ we expect to obtain the following contribution
\begin{align}
&I^{C}_{1} = \E^{\pdh} \left [\int_{0}^{T}\int_{D} [\rdh \udh \tensor \udh-2 \mu \nabla \udh ]: \nabla \eta \tensor \mathcal{A}\left [\eta \rdh \right ]dxds \right ] \\
&+\E^{\pdh} \left [\int_{0}^{T}\int_{D}(-\lambda \Div \udh +\rdh^{\gamma} + \delta \rdh^{\beta})I] : \nabla \eta \tensor \mathcal{A}\left [\eta \rdh \right ]dxds \right ] \\
&I^{C}_{2} = \E^{\pdh} \left [\int_{0}^{T}\int_{D} \rdh \udh \cdot \mathcal{A}[\nabla \eta \cdot \rdh \udh] dxds \right ] \\
&I^{C}_{3} = \E^{\pdh} \left [\int_{0}^{T}\int_{D} \udh \cdot [\mathcal{A}(\eta \rdh)\nabla \eta - \nabla \eta \eta \rdh]  dxds \right ]\\
&I^{P}_{1} = \E^{\pdh} \left [\int_{0}^{T}\int_{D} \eta \left [ \rdh \udh \tensor \udh : \nabla \mathcal{A}[\eta \rdh] - \rdh \udh \cdot \mathcal{A} \circ \Div(\eta \rdh \udh) \right ]dxds \right ] \\
&I^{P}_{2} = -\E^{\pdh} \left [ \int_{D} \rdh \udh(T) \cdot \hat{\varphi}_{\delta}(T)dx \right ]
\end{align}
First note that by the uniform bounds in \eqref{Eq:Eps:Setup:Uniform_Bounds} combined with an interpolation argument we obtain the estimate
\begin{equation}
I^{A,\epsilon}_{1}+I^{A,\epsilon}_{2} \leqs \sqrt{\epsilon} \E^{\pdh} \left [|\sqrt{\epsilon}\nabla  \rho_{\epsilon}|_{\LLs{2}}^{2} \right ]^{\frac{1}{2}} \left (\E^{\pdh}\left [|\reh \ueh|_{\LLs{2}}^{2} \right ]^{\frac{1}{2}} +\E^{\pdh} \left [|\ue|_{\LW{2}{1}{2}}^{2} \right ]^{\frac{1}{2}}  \right ) \to 0.
\end{equation}
To treat the remaining integrals, note that by the uniform bounds \eqref{Eq:Eps:Setup:Uniform_Bounds} and the Vitali convergence theorem, it suffices to establish the relevant $\pdh$ convergence, so the analysis essentially reduces to the same arguments as in the deterministic framework(by design).  We recall them here for the convienience of the reader. \\

Starting with the first cutoff term, note that for all $q \geq 1$, $\mathcal{A}: L^{\beta}_{x} \to L^{q}_{x}$ is compact.  Hence we may combine \eqref{Eq:Del:Comp:Skorohod:Pointwise_Convergence:Density}  with appendix result Theorem \ref{Thm:Appendix:Bogovoski} to obtain $\mathcal{A}[\eta \re] \to \mathcal{A}[\eta \rdh]$ strongly in $\LLs{q}$.  In view of a similar argument in Lemma \ref{Lem:Eps:Id:Momentum_Equation}, we have $\re \ue \tensor \ue \to \rdh \udh \tensor \udh$ in $\LL{2}{\frac{\beta d}{\beta(d-1)+d}}$. Combining these two observations, \eqref{Eq:Del:Comp:Skorohod:Pointwise_Convergence:Velocity} and \eqref{Eq:Del:Comp:Skorohod:Pointwise_Convergence:Pressure} we end up with a product of a weakly converging sequence and a strongly converging sequence and conclude $I^{C,\epsilon}_{1} \to I^{C}_{1}$.  A similar argument also yields $I^{P,\epsilon}_{2} \to I^{P}_{2}$. \\

Next note that $\mathcal{A}: L^{\frac{2\beta}{\beta+1}}_{x} \to L^{r}_{x}$ is compact for $\frac{1}{r} > \frac{1}{2}+\frac{1}{2\beta}-\frac{1}{d}$.  We may now use \eqref{Eq:Del:Comp:Skorohod:Pointwise_Convergence:Momentum} together with Theorem \ref{Thm:Appendix:Bogovoski} to conclude that $\mathcal{A}[\nabla \eta \cdot \re \ue] \to \mathcal{A}[\nabla \eta \cdot \rdh \udh]$ strongly in $\LL{m}{r}$ for all $m \geq 1$ and $r$ satsifying the relation above.   We may use again \eqref{Eq:Del:Comp:Skorohod:Pointwise_Convergence:Momentum} to obtain another weak times strong where the exponents match up appropriately since $\beta >d$, allowing us to conclude $I^{C,\epsilon}_{2}  \to I^{C}_{2}$.  
To treat the final cutoff term, simply argue as in the passage to the limit in the flux term of the continuity equation and obtain $I^{C,\epsilon}_{3} \to I^{C}_{3}$. \\    
%Next note that the identity map $I : L^{\frac{2\gamma}{\gamma+1}} \to W^{-1,p}_{x}$ is compact provided $\frac{1}{p} >\frac{1}{2}+\frac{1}{2\gamma}-\frac{1}{d}$, so appendix Lemma \eqref{lem:BoundedUpgrade} implies the strong convergence of $\re \ue \to \rdh \udh$ in $\LW{2}{-1}{p}$ $\, \pdh$ almost surely.  On the other hand, since $\mathcal{A}: L^{r}_{x} \to W^{1,r}_{x}$ is bounded for $\frac{1}{r} = \frac{1}{\beta}+\frac{1}{2}-\frac{1}{d}$, we may apply appendix lemma \eqref{lem:BoundedUpgrade} to obtain $\mathcal{A}[\nabla \eta \cdot \re \ue] \to \mathcal{A}[\nabla \eta \cdot \rdh \udh]$ weakly in $\LW{2}{1}{r}$ $\pdh$ almost surely.  As $\beta > 2\gamma$, we may combine these observations with the uniform bounds \eqref{eq: epsUnifBds} yield  
The treatment of the principle term $I^{P,1}$ is the nontrival part, but we have built most of the work into an appendix result based on the Div Curl lemma.  Working componentwise we may write
\begin{equation}
 I^{P,\epsilon}_{1} =\sum_{i,j=1}^{d} \E^{\pdh} \left [\int_{0}^{T}\int_{D}\ue^{i} \big (\eta \re \partial_{ij}\Delta^{-1}(\eta \re \ue^{j})-\eta\re \ue^{j}\partial_{ij}\Delta^{-1}(\eta \re) \big )dxds \right ]
\end{equation}
In view of \eqref{Eq:Del:Comp:Skorohod:Pointwise_Convergence:Density} and \eqref{Eq:Del:Comp:Skorohod:Pointwise_Convergence:Momentum}, we may appeal to Lemma \ref{Lem:Appendix:Div_Curl_Commutator} with $p=\beta$ and $q=\frac{2\beta}{\beta+1}$, making use of the compact embedding $L^{\frac{2\beta}{3+\beta}} \hookrightarrow W^{-1,2}_{x}$ for $\beta > \frac{3}{2}d$ in order to conclude that $\pdh$ a.s.
% L^{\frac{2\gamma}{\gamma-1}} embeds into W^{-1,p} for \frac{1}{p} > \frac{1}{2}-\frac{1}{d}-\frac{1}{2\gamma}
\begin{align*}
\big (\eta \re \partial_{ij}\Delta^{-1}(\eta \re \ue^{j})-\eta\re \ue^{j}\partial_{ij}\Delta^{-1}(\eta \re) \big ) 
\to \big (\eta \rdh \partial_{ij}\Delta^{-1}(\eta \rdh \udh^{j})-\eta\rdh \udh^{j}\partial_{ij}\Delta^{-1}(\eta \rdh) \big ) 
\end{align*}
strongly in $\LW{2}{-1}{2}$.  Appealing once more to \eqref{Eq:Eps:Setup:Uniform_Bounds} and Vitali we find $I^{P,\epsilon}_{1} \to I^{P}_{1}$.
\end{proof}

We now proceed to a proof of the strong convergence of the density.
\begin{Lem} \label{Lem:Del:Strong_Convergence_Of_Density}
The sequence of densities $\{\re\}_{\epsilon > 0}$ converges strongly to $\rdh$ in the sense that for all $p \geq 1$ and $r<\beta+1$
\begin{align}
\lim_{\epsilon \to 0} |\re-\rdh|_{L^{p} \big (\Odh ; \LLs{r} \big )}=0.
\end{align}
\end{Lem}
\begin{proof}.  We will begin by establishing $\overline{\rdh \text{log}(\rdh)} = \rdh \text{log}(\rdh)$ almost everywhere in $\Odh \times [0,T] \times D$.  
Using the Di Perna Lions commutator lemmas and the Hardy inequality, we may renormalize the parabolic equation for the $\re$ and the transport equation for $\rdh$ separately(and note they have the same initial data), using the renormalization $\beta(\rho)=\rho \text{log}(\rho)$.  As a result, we find that for all smooth $\psi$ with $\psi(T)=0$ the $\pdh$ a.s. inequality holds
\begin{align*}
	&\int_{0}^{T}\int_{D}\psi_{t}[\rdh \log{\rdh} - \re \log{\re}] dxds \leq \int_{0}^{T}\int_{D}\psi \left [ \rdh \Div \udh -\re \Div \ue\right ] dxds. 
\end{align*}
Let us further decompose the RHS into a portion over a compact set $K \subset \subset D$ where the weak continuity result Lemma \ref{Lem:Del:Strong:Weak_Continuity_Of_The_Effective_Viscous_Flux} may be applied, and a remainder $R_{K}^{\epsilon}(\psi)$  
\begin{align*}
&\int_{0}^{T}\int_{D}\psi_{t}[\rdh \log{\rdh} - \re \log{\re}] dxds \leq
\int_{0}^{T}\int_{K}\psi \left [ \rdh \Div \udh -\re \Div \ue\right ] dxds \\
&+\int_{0}^{T}\int_{D \setminus K}\psi \left [ \rdh \Div \udh -\re \Div \ue\right ] dxds \\	
& = \int_{0}^{T}\int_{K} \psi \rdh( \Div \udh - \overline{\rdh^{\gamma}}-\delta \overline{\rdh^{\beta}}) - \psi \re( \Div \ue - \re^{\gamma}-\delta \re^{\beta}) dxds \\
&+ \int_{0}^{T}\int_{K} \psi \rdh (\overline{\rdh^{\gamma}} +\overline{\delta \rdh^{\beta}} ) - \psi \re(\re^{\gamma}+\delta \re^{\beta}) dxds+R_{K}^{\epsilon}(\psi). 
\end{align*}
We now take expectations with respect to $\pdh$ on both sides of the inequality above and send $\epsilon \to 0$.  By Lemma \eqref{Lem:Del:Strong:Weak_Continuity_Of_The_Effective_Viscous_Flux}, the first term tends to zero.  Moreover, we may apply \eqref{Eq:Del:Comp:Skorohod:Weak_Convergences:Log_Renorm_Term} to conclude the remainders $R_{K}^{\epsilon}(\psi)$ converge to $\overline{R}_{K}(\psi)$ defined below.  Hence, we conclude
\begin{align} \label{Eq:Del:Strong_Convergence_Of_Density:Pre_Lebesgue}
&\int_{0}^{T}\int_{D}\psi_{t} \E^{\pdh} \left [ \rdh \log{\rdh} - \overline{\rdh \log{\rdh}} \right ] dxds \nonumber \\
&\leq \liminf_{\epsilon \to 0}\E^{\pdh} \left [\int_{0}^{T}\int_{K} \psi \rdh (\overline{\rdh^{\gamma}} +\delta \overline{\rdh^{\beta}} ) - \psi \reh(\reh^{\gamma}+\delta \reh^{\beta}) dxds \right ] + \overline{R}_{K}(\psi). \\
&\overline{R}_{K}(\psi) = \E^{\pdh} \left [ \int_{0}^{T}\int_{D \setminus K}\psi \left [ \rdh \Div \udh -\overline{\rdh \Div \udh}\right ] dxds\right ] 
\end{align}
Applying Minty's trick to the convex function $\beta(\rho) = \rho^{\gamma}+\delta \rho^{\beta}$, the first term on the RHS of \eqref{Eq:Del:Strong_Convergence_Of_Density:Pre_Lebesgue} is non-positive.  Hence, for any lebesgue point $s \in [0,T]$ of the function $s \to \E^{\pdh} \left [ \rdh \log{\rdh} - \overline{\rdh \log{\rdh}} \, \right ] (s) $, we may choose a sequence of test functions that approximate $1_{[0,s]}(t)$, so that their time derivatives approximate the negative of a dirac mass centered at the point $s$.  Hence, the following inequality holds almost everywhere in time  
\begin{equation} \label{Eq:Del:Strong_Convergence_Of_Density:Log_Energy}
\E^{\pdh} \left [  \int_{D} (\overline{\rdh \log{\rdh}}-\rdh \log{\rdh} )dx \right ](s) \leq 
\E^{\pdh} \left [ \int_{0}^{s}\int_{D \setminus K}\psi \left [ \rdh \Div \udh -\overline{\rdh \Div \udh}\right ] dxdr\right ]. 
\end{equation}
We can make the RHS of \eqref{Eq:Del:Strong_Convergence_Of_Density:Log_Energy} arbitrarily small by choosing $K$ close enough to $D$.  Hence, the quantity on the LHS of \eqref{Eq:Del:Strong_Convergence_Of_Density:Log_Energy} starts from zero, is always positive(in view of the convexity of $\rho \text{log}(\rho)$), and almost never increases.  Thus, the LHS vanishes almost everywhere in time and we conclude $\overline{\rdh \text{log}(\rdh)} = \rdh \text{log}(\rdh)$ almost everywhere in $\Odh \times [0,T] \times D$. \\  
Combining this with the weak convergence \eqref{Eq:Del:Comp:Skorohod:Weak_Convergences:Entropy}, we may apply Lemma \ref{Lem:Appendix:Convex_Function_Upgrade} to conclude the $\Odh \times D \times [0,T]$ almost everywhere convergence of $\re \to \rdh$ away from the vacuum regions of $\rdh$.  Moreover, since $\{ \re\}_{\epsilon >0}$ is a nonnegative sequence, weak convergence implies almost every where convergence on the vacuum regions of the limit.  Finally, we may use the uniform bounds \eqref{Eq:Eps:Setup:Uniform_Bounds} in order to deduce the desired claim.
\end{proof}
\subsection{Conclusion of the Proof}	
%!TEX root = main.tex
\begin{proof}[Proof of Theorem \ref{Thm:Del:Setup:Del_Layer_Existence}]
For each $\delta >0$ we apply Proposition \ref{Prop:Del:Comp:Skorohod} to construct a candidate $\delta$ layer approximation $(\rdh,\udh)$. Combining the preliminary limit passage, Lemmas \ref{Lem:Del:Id:Continuity_Equation} and \ref{Lem:Del:Id:Momentum_Equation} together with the strong convergence of the density, Lemma \ref{Lem:Del:Strong_Convergence_Of_Density}, we are able to identify $\overline{\rdh^{\gamma}}+\delta\overline{\rdh^{\beta}}= \rdh^{\gamma}+\delta \rdh^{\beta}$ and complete the identification procedure.  The uniform bounds can be argued as in the proof of Theorem \ref{Thm:Eps:Setup:Eps_Layer_Existence}.
\end{proof}

\section{Proof of the Main Result: $\delta \to 0$} \label{Section:Pf}
%!TEX root = main.tex
\subsection{$\delta \to 0$ Compactness Step}
%!TEX root = main.tex
Following Feiresil \cite{feireisl2004dynamics}, we introduce a collection of approximations from below of the renormalizations $\beta(\rho)=\rho$ and $\beta(\rho) = \rho \log{\rho}$, respectively.  These will be useful in our proof of the strong convergence of the density later on in the section.  Namely, let $T(\rho)=\rho1_{[0,1]}(\rho) +\phi(\rho)1_{[1,3]}(\rho)+2 \,1_{[3,\infty)}(\rho)$, where $\phi$ is chosen to ensure global smoothness and concavity of $T(\rho)$.  Use $T$ to define the approximations by letting for each $k$, $T_{k}(\rho) =kT(\frac{\rho}{k})$ and $L_{k}(\rho)=\rho \int_{1}^{\rho}\frac{T_{k}(z)}{z^{2}}$.  Note that by design $\rho L_{k}'(\rho)-L_{k}(\rho) = T_{k}(\rho)$.
We now proceed to our final application of the Skorohod theorem.
\begin{Prop} \label{Prop:Pf:Comp:Skorohod}
There exists a probability space $(\Omega,\mathcal{F},\p)$ along with a sequence of ``recovery" maps  $\{\hat{T}_{\delta} \}_{\delta > 0}$ and limit points 
\begin{align*}
&\left (\rho,u,\overline{\rho^{\gamma}},\{\beta_{k}\}_{k=1}^{\infty},\overline{\sqrt{\rho}u},  \overline{\rho\log{\rho}} ,\{ \overline{T_{k}(\rho)} \}_{k=1}^{\infty} \right ) \\
&\left (\{ \overline{L_{k}(\rho)} \}_{k=1}^{\infty},\{ \overline{T_{k}(\rho)\Div u} \}_{k=1}^{\infty}, \{ \overline{(\rho T_{k}'(\rho) - T_{k}(\rho)) \Div u}\}_{k=1}^{\infty} \right )\\
&\hat{T}_{\delta} : (\Omega,\F,\p) \to (\Odh,\Fdh,\pdh)
\end{align*}
such that the following hold\\
\begin{enumerate}
\item	\label{Prop:Pf:Comp:Skorohod:Item:Recovery_Maps}
The measure $\pdh$ may be recovered by pushing forward $\p$ under $\hat{T}_{\delta}$.\\
\item   \label{Prop:Pf:Comp:Skorohod:Item:Preserving_The_Equation}
The new sequence $\{(\rd,\ud)\}_{\delta> 0}$ defined by  $ \left(\rd,\ud \right )=\left (\rdh ,\udh \right ) \circ \hat{T}_{\delta}$ constitutes a $\delta$ layer approximation relative to the stochastic basis $(\Omega,\F,\p,\{ \mathcal{F}^{t}_{\delta}\}_{t \geq 0},W_{\delta} )$, where
$$ W_{\delta} := \widehat{W}_{\delta} \circ \hat{T}_{\delta} \quad \quad \quad \F^{t}_{\delta} := \sigma \left (r_{t}\rd, r_{t}\ud , r_{t}(\rd \ud), r_{t}W_{\delta} \right )$$
\item The following uniform bounds hold for all $p \geq 1$
\begin{equation} \label{Prop:Pf:Comp:Skorohod:Item:Preserving_The_Bounds}
\sup_{\delta > 0} \E^{\p} \bigg [ |\sqrt{\rd}\ud |_{\LL{\infty}{2}}^{2p} + |\rd|_{\LL{\infty}{\gamma}}^{\gamma p}+|\delta^{\frac{1}{\beta}}\rd|_{\LL{\infty}{\beta}}^{\beta p} + |\ud|_{L^{2}_{t}(H^{1}_{0,x})}^{2p}\bigg ] < \infty \\
\end{equation} \label{Prop:Pf:Comp:Skorohod:Item:Pointwise_Convergences}
\item Let $q \in (1,\gamma)$, then the following convergences hold $\p$ a.s.
\begin{align} 
\label{Eq:Pf:Comp:Skorohod:Pointwise_Convergence:1:Density1}&\rd \to \rho \quad &\text{in} \quad  &C_{t}\big ( [L^{\gamma}_{x}]_{w} \big ) \\ 
\label{Eq:Pf:Comp:Skorohod:Pointwise_Convergence:2:Velocity} &\ud \to u \quad &\text{in} \quad &L^{2}_{t}(H^{1}_{0,x})  \\
\label{Eq:Pf:Comp:Skorohod:Pointwise_Convergence:3:Momentum} & \rd \ud \to \rho u \quad & \text{in} \quad &\CLw{\frac{2\gamma}{\gamma+1}} \\
\label{Eq:Pf:Comp:Skorohod:Pointwise_Convergence:4:Pressure} & \rd^{\gamma} \to \rgb \quad & \text{in} \quad & [L^{1+\frac{\kappa}{\gamma}}_{t,x} \big ]_{w} \\
\label{Eq:Pf:Comp:Skorohod:Pointwise_Convergence:5:Brownian_Motion} &W_{\delta} \to W \quad &\text{in} \quad &[C_{t}]^{\infty} \\
\label{Eq:Pf:Comp:Skorohod:Pointwise_Convergence:6:Truncated_Density} &T_{k}(\rho_{\delta}) \to \overline{T_{k}(\rho)} \quad & \text{in} \quad & \text{$\CLw{\frac{4\gamma d}{2\gamma-d}}$ } \\
\label{Eq:Pf:Comp:Skorohod:Pointwise_Convergence:7:Truncated_Entropy} &L_{k}(\rho_{\delta}) \to \overline{L_{k}(\rho)} \quad & \text{in} \quad & \text{$\CLw{q}$} \\
\label{Eq:Pf:Comp:Skorohod:Pointwise_Convergence:6:Renorm_Truncated_Log} &(\rd T_{k}'(\rd)-T_{k}(\rd))\Div \ud  \to \overline{(\rho T_{k}'(\rho) - T_{k}(\rho)) \Div u} \quad & \text{in} \quad & \text{$\LL{\infty}{2}$} \\
\label{Eq:Pf:Comp:Skorohod:Pointwise_Convergence:6:Renorm_Log} &\rho_{\delta}\log(\rho_{\delta}) \to \overline{\rho \log(\rho)} \quad & \text{in} \quad & \text{$\CLw{q}$} 
\end{align}

\item \label{Prop:Pf:Comp:Skorohod:Weak_Convergences} The following additional convergences hold:
\begin{align}
&\label{Eq:Pf:Comp:Skorohod:Item:Weak_Convergences:1:Kinetic_Energy}\sqrt{\rho_{\delta}}u_{\delta} \to \overline{\sqrt{\rho} u} & \quad \text{in} \quad &L^{p}_{w*} \big ( \Omega ; \LL{\infty}{2}\big ) \\
&\label{Eq:Pf:Comp:Skorohod:Item:Weak_Convergences:2:Density}\rho_{\delta} \to \rho & \quad \text{in} \quad &L^{p}_{w{*}} \big ( \Omega ; \LL{\infty}{\gamma}\big ) \\
%& u_{\delta} \to u & \quad \text{in}  \quad &L^{p}_{w} \big (\Omega ;  \big ) \\
&\label{Eq:Pf:Comp:Skorohod:Item:Weak_Convergences:3:Entropy}\rho_{\delta} \log(\rho_{\delta}) \to \overline{\rho \log(\rho)} & \quad \text{in} \quad &L^{p}_{w^{*}} \big ( \Omega ; \LL{\infty}{q}\big ) \\
&\label{Eq:Pf:Comp:Skorohod:Item:Weak_Convergences:4:Truncation_Renorm_Term}T_{k}(\rho_{\delta})\Div \ud  \to \overline{T_{k}(\rho)\Div u} \quad & \text{in} \quad & L^{p}_{w}(\LLs{2}) \\
\end{align}
\end{enumerate}
\end{Prop}
In order to prove the tightness of the pressure, we will need the following integrability gains.
\begin{Prop}
The following integrability gains hold for arbitrary $p \geq 1$ and $\kappa< \min(\frac{2\gamma}{d}-1,\frac{\gamma}{2})$:
\begin{equation} \label{Prop:Pf:Comp:Integrability_Gains}
\sup_{\delta > 0} \E^{\pdh} \bigg [ \big | \int_{0}^{T}\int_{D} \rdh^{\gamma+\kappa} + \delta \rdh^{\beta + \kappa} dxdt \big |^{p} \bigg] < \infty.
\end{equation} 
\end{Prop}
%%% Citations
%%%
\begin{proof} The proof of this proposition follows along the same lines as the proof of the integrability gains at the $\epsilon$ layer.  We will give a limited amount of details and mainly indicate how the constraint on $\kappa$ arises.   
Introduce the following random test function
\begin{equation}
\hat{\varphi}_{\delta} = \mathcal{B}[\rdh^{\kappa} - \frac{1}{|D|}\int_{D}\rdh^{\kappa} dx].
\end{equation}
By the Di Perna Lions commutator lemmas, we may check that $\pdh$ almost surely, in the sense of distributions
\begin{equation*}
\partial_{t}(\rdh^{\kappa}-\frac{1}{|D|}\int_{D}\rdh^{\kappa}) + \Div(\rdh^{\kappa}\udh) +(\kappa-1)\left [ \rdh^{\kappa}\Div(\udh) - \frac{1}{|D|}\int_{D}\rdh^{\kappa}\Div(\udh)dx \right ] =0.
\end{equation*}
Applying the Bogovski operator on both sides yields
\begin{equation*}
\partial_{t}\hat{\varphi}_{\delta}= -\mathcal{B} [\Div(\rdh^{\kappa}\udh)] -(\kappa-1)\mathcal{B} \left [ \rdh^{\kappa}\Div(\udh) - \frac{1}{|D|}\int_{D}\rdh^{\kappa}\Div(\udh)dx \right ]. 
\end{equation*}
We may now use the Ito product rule together with the momentum equation to find the evolution of $\rdh \udh \cdot \hat{\varphi}_{\delta}$, 
\begin{equation*}
\begin{split}
&\int_{0}^{T}\int_{D} \rdh^{\gamma+1} + \delta \rdh^{\beta + 1} dxdt = \int_{D} \big [\rdh \udh(T) \cdot \hat{\varphi}_{\delta}(T) -m_{0,\delta} \cdot\hat{\varphi}_{\delta}(0) \big ] \\
&+\int_{0}^{T}\int_{D} [2\mu \nabla \udh+\lambda \Div \udh I-\rdh \udh \tensor \udh] : \nabla \hat{\varphi}_{\delta} \\
& + \int_{0}^{T}\int_{D} (\rdh^{\gamma}+\delta \rdh^{\beta})\int_{D}\frac{\rdh^{\kappa}}{|D|} dx + \int_{0}^{T}\int_{D}  \rdh \udh \cdot \mathcal{B}[\Div(\rdh^{\kappa} \udh)]dxds \\
&+ (\kappa -1)\int_{0}^{T}\int_{D}\rdh \udh \cdot \mathcal{B} \left [ \rdh^{\kappa}\Div(\udh) - \oint_{D}\rdh^{\kappa}\Div(\udh)dx \right ] dxds \\
&-\sum_{k=1}^{\infty}\int_{0}^{t}\int_{D}\rdh \sigma_{k,\delta}(\rdh,\rdh \udh) \cdot \hat{\varphi}_{\delta} dxd\bdhk(s).  \\
\end{split}
\end{equation*}
We estimate the new term from the renormalization as follows
\begin{align*}
&\E^{\pdh}\left [  |\rdh \udh \cdot \mathcal{B}(\rdh^{\kappa}\Div \udh)|_{\LLs{1}}^{m} \right ] \\
&\leq \E^{\pdh} \left [|\rdh|_{\LL{\infty}{\gamma}}^{3m} \right ]^{1/3} \E^{\pdh} \left [ |\udh|_{L^{2}_{t}(H^{1}_{0,x})}^{3m}\right ]^{1/3} \E^{\pdh} \left [ |\mathcal{B}(\rdh^{\kappa}\Div \udh)|_{\LL{2}{r}}^{3m} \right ]^{1/3},
\end{align*}
where $\frac{1}{r}=\frac{1}{2}+\frac{1}{d}-\frac{1}{\gamma}$.  Define the exponent $p$ via $\frac{1}{p}=\frac{1}{2}+\frac{2}{d}-\frac{1}{\gamma}$. The Sobolev embedding $W^{1,p}_{x} \hookrightarrow L^{r}_{x}$ gives
\begin{align*}
&\E^{\pdh} \left [ |\mathcal{B}(\rdh^{\kappa}\Div \udh)|_{\LL{2}{r}}^{3m} \right ]^{1/3} \leqs \E^{\pdh} \left [ |\rdh^{\kappa}\Div \udh|_{\LL{2}{p}}^{3m} \right ]^{1/3} \\ 
&\leqs \E^{\pdh} \left [ |\rdh^{\kappa}|_{\LL{\infty}{q}}^{6m} \right ]^{1/6}\leqs \E^{\pdh} \left [\Div \udh|_{\LL{2}{2}}^{6m} \right ]^{1/6},
\end{align*}
where $\frac{1}{q}=\frac{1}{p}-\frac{1}{2}$ by H\"older, which is valid in view of $\gamma >\frac{d}{2}$.  To control this final term we require $\kappa q<\gamma$ which leads to the condition $\kappa < \frac{2\gamma}{d}-1$.   Similar estimates yield uniform control of the terms $\E^{\pdh} \left [ |\mathcal{B} [\Div(\rdh^{\kappa} \udh)] \cdot \rdh \udh|_{\LLs{1}}^{m} \right ]$ and $\E^{\pdh} \left [|\rdh \udh \tensor \udh : \nabla \mathcal{B}(\rdh^{\kappa})|_{\LLs{1}}^{m} \right ]$.
The stochastic integrals and the term 
$$\E^{\pdh} \left [ |\rdh \udh(T) \cdot \hat{\phi}_{\delta}(T) |^{m}_{L^{1}_{x}} \right ]$$ 
are both controllable under the condition 
%$\kappa < \gamma(\frac{1}{2}+\frac{1}{d})-\frac{1}{2}$ which is less restrictive than%
 $\kappa < \frac{2\gamma}{d}-1$.  Finally, the dissipative term and mean value correction can be estimated provided that $\kappa < \frac{\gamma}{2}$.
\end{proof}
%
%
%
%
%%%%%%%%%%%% Delta Layer Tightness
%%%%%%%%%%%%
%%%%%%%%%%%%
%%%%%%%%%%%%
%%%%%%%%%%%%
Define the sequence of random variables $\{ (\hat{X}_{\delta},\hat{Y}_{\delta}) \}_{\delta > 0}$ via
\begin{equation}
\begin{split}
	&\hat{X}_{\delta} = \left ( \rdh,\udh,\rdh \udh, \rdh^{\gamma},\rdh \log{\rdh}, \{ \bdhk  \}_{k=1}^{\infty} \right )  \\
	%&\hat{Y}_{\delta} = (\hat{T}_{\delta},\hat{L}_{\delta},\hat{T}_{\delta}\Div{\udh},\hat{W}_{\delta}) \\
	&\hat{Y}_{\delta} = \left (\{ T_{k}(\rdh) \}_{k=1}^{\infty}, \{ L_{k}(\rdh) \}_{k=1}^{\infty}, \{ (\rdh T_{k}'(\rdh)-T_{k}(\rdh))\Div \udh \}_{k=1}^{\infty} \right ). \\
	%& \widehat{W}_{\delta} = \big \{ \bdhk \big \}_{k=1}^{\infty}, \quad \widehat{T}_{\delta} = \big \{ T_{k}(\rdh) \big \}_{k=1}^{\infty}, \quad 
	 %\widehat{L}_{\delta} = \big \{ L_{k}(\rdh) \big \}_{k=1}^{\infty}, \quad \widehat{T_{\delta}\Div{\ud}} = \big \{ T_{k}(\rdh)\Div{\udh} \big \}_{k=1}^{\infty}\\
\end{split}
\end{equation}
These random variables induce a measure on the space $E \times F$ where 
\begin{equation} \label{Eq:Pf:Comp:Tightness_Space}
\begin{split}
	& E =  C_{t}\big ( [L^{\gamma}_{x}]_{w} \big ) \times L^{2}_{t}(H^{1}_{0,x}) \times  \CLw{\frac{2\gamma}{\gamma+1}} \times \CLw{q} \times \big [L^{1+\frac{\kappa}{\gamma}}_{t,x} \big ]_{w} \times [C_{t}]^{\infty}\\
	& F =   \left [ \CLw{\frac{4\gamma d}{2\gamma -d}} \times \CLw{q} \times \LLs{2} \right ]^{\infty}. 
\end{split}
\end{equation}
The space is understood to be endowed with its natural product topology.
%%%%%%% Delta Layer Tightness
%%%%%%%
%%%%%%%
\begin{Lem}
The sequence of induced measures $\pdh \circ (\hat{X}_{\delta}, \hat{Y}_{\delta})^{-1}$ are tight on $E \times F$.
\end{Lem}
\begin{proof} We proceed componentwise. Note that the methods in Lemma \ref{Lem:Del:Comp:Tightness} can be used to deduce the tightness of $\{\pdh \circ \left ( \rdh,\udh,\rdh^{\gamma},\widehat{W}_{\delta} \right )^{-1} \}_{\delta > 0}$.  To treat the momentum sequence, argue as in the other layers and deduce for all $\phi \in C_{x}^{1}$ the following estimate holds    
\begin{equation}
	 \sup_{\delta > 0} \E^{\pdh} | \langle \rdh\udh(t) -\rdh\udh(s) , \phi \rangle |^{p} \leqs  |\phi|_{C^{1}_{x}}^{p} |t-s|^{\frac{p}{2}}.
\end{equation}
In view of Lemma \ref{Lem:Appendix:TightCrit}, the tightness of $\{ \pdh \circ (\rdh \udh)^{-1} \}_{\delta >0}$ follows.  Next we will treat the induced measures $\pdh \circ \hat{Y}_{\delta}$.   By the Di Perna Lions commutator lemmas, the following equality holds(in the analytic sense of distributions) $\pdh$ a.s.
\begin{equation*}
\partial_{t}L_{k}(\rdh) = - \Div(L_{k}(\rdh) \udh) - T_{k}(\rdh) \Div \udh. 
\end{equation*}
Noting that $L_{k}(\rho) \leqs k \rho$ and $T_{k}(\rho) \leqs k$ we obtain the estimate
\begin{align*}
&|\partial_{t} L_{k}(\rdh)|_{\LW{2}{-1}{\frac{2^{*}\gamma}{\gamma+2^{*}} \wedge 2}} \leqs |L_{k}(\rdh)|_{\LL{\infty}{\gamma}} |\udh |_{\LL{2}{2^{*}}} + |T_{k}(\rdh)|_{\LLs{\infty}}|\Div \udh|_{\LLs{2}} \\ 
&\leqs (k+|\rdh|_{\LL{\infty}{\gamma}}) |\udh|_{L^{2}_{t}(H^{1}_{0,x})}
\end{align*}
A similar estimate holds for $T_{k}(\rdh)$.  Let $\{ M_{k}\}_{k=1}^{\infty}$ be a sequence of real numbers and note that
\begin{align*}
 & \prod_{k=1}^{\infty} \{(f,g,h) \in  \CLw{\frac{4\gamma d}{2\gamma -d}} \times \CLw{q} \times \LLs{2} \mid |\partial_{t}f|_{\LW{2}{-1}{\frac{2^{*}\gamma}{\gamma+2^{*}} \wedge 2}} + |f|_{\LL{\infty}{\frac{4\gamma d}{2\gamma -d}}} \\
 &+|\partial_{t}g|_{\LW{2}{-1}{\frac{2^{*}\gamma}{\gamma+2^{*}} \wedge 2}} + |g|_{\LL{\infty}{q}} 
  + |h|_{\LLs{2}} \leq M_{k}2^{k} \}
\end{align*}
is a compact set in $\left [ \CLw{\frac{4\gamma d}{2\gamma -d}} \times \CLw{q} \times \LLs{2} \right ]^{\infty}$ by Corollary \ref{Cor:Appendix:Time_Derivative_In_A_Negative_Sobolev_Space} , Banach-Alagolu, and Tychonoff.  Hence, we may appeal to \eqref{Eq:Del:Setup:Uniform_Bounds} and Chebyshev to show that sets of this form have $\pdh$ probability arbitrarily close to one, uniformly in $\delta$.  Finally, one can treat $\{ \pdh \circ (\rdh \log{\rdh})^{-1} \}_{\delta >0}$ in a similar way, using the renormalized form of the continuity equation and the integrability gains to control the term $\rdh \Div \udh$.  
\end{proof}  
%%%%%%%%%%%% PROOF OF THE SKOROHOD PACKAGE
%%%%%%%%%%%%
%%%%%%%%%%%%
%%%%%%%%%%%%
\begin{proof}[Proof of Proposition \ref{Prop:Pf:Comp:Skorohod}]:
The proof follows along the lines of the $\epsilon$ layer.
%%%%%%%%%%%%%% PASSING THE LIMIT IN THE DELTA LAYER
%%%%%%%%%%%%%%
%%%%%%%%%%%%%%
%%%%%%%%%%%%%%
%%%%%%%%%%%%%%
\end{proof}

\subsection{ $\delta \to 0$ Preliminary Limit Passage}
%!TEX root = main.tex
Using the restriction operator according to the spaces described in Proposition \ref{Prop:Pf:Comp:Skorohod}, we define the filtration $\{\F_{t}\}_{t=0}^{T}$ via
\begin{equation}
\F_{t} = \sigma \big (r_{t}\rho , r_{t}u, r_{t}\big ( \rho u \big ), r_{t}W \big ) 
\end{equation}
\begin{Lem} 
The pair $(\rho,u)$ satisfies the continuity equation, part \ref{Eq:Pre:Continuity_Equation}  of \ref{Def:Pre:Weak_Solutions}.
\end{Lem}
\begin{proof}
In view of the strong convergence of the initial density laid forth in Hypothesis \ref{Hyp:data}, we may proceed as in Lemma \ref{Lem:Del:Id:Continuity_Equation}.
\end{proof}   

\begin{Lem} \label{Lem:Pf:Id:Momentum_Martingale}
For all $\phi \in C^{\infty}_{c}(D)$, the process $\{ M_{t}(\phi) \}_{t=0}^{T}$ defined by
$$
 M_{t}(\phi)  = \int_{D} \rho u(t) \cdot \phi dx  - \int_{D} m_{0}\cdot \phi - \int_{0}^{t}\int_{D} [\rho u \tensor u-2\mu \nabla u-\lambda \Div u I ] : \nabla \phi + \overline{\rho^{\gamma}}\Div \phi dxds. 
$$
is a continuous, $\{\mathcal{F}^{t}\}_{t=0}^{T}$ martingale satisfying for all $p \geq 1$
\begin{equation} \label{Eq:Pf:Id:Martingale_Moment_Bounds}
\E \left [\sup_{t \in [0,T]} |M_{t}(\phi)|^{p} \right ] < \infty
\end{equation}
\end{Lem}
\begin{proof} 
Introduce the continuous process $\{ M_{t}^{\delta}(\phi)\}_{t=0}^{T}$ defined by
$$
 M_{t}^{\delta}(\phi)  = \int_{D} \rd \ud(t) \cdot \phi dx  - \int_{D} m_{0,\delta}\cdot \phi - \int_{0}^{t}\int_{D} [\rd \ud \tensor \ud-2\mu \nabla \ud-\lambda \Div \ud I ] : \nabla \phi + \rd^{\gamma}\Div \phi dxds 
$$
We may use Proposition \ref{Prop:Pf:Comp:Skorohod} along with Hypothesis \ref{Hyp:data} to establish for all $t \in [0,T]$ the convergence $M_{t}^{\delta}(\phi) \to M_{t}(\phi)$ almost surely with respect to $\p$.  Indeed, the only additional steps(other than what was required at the $\epsilon$ layer) are noting
\begin{equation*}
\E \left [ |\int_{0}^{T}\int_{D}\delta \rd^{\beta}dx|^{p} \right ] \leq \delta^{\frac{\beta p}{\beta+\kappa}} \E \left [ |\int_{0}^{T}\int_{D}\rd^{\beta+\kappa}dxds|^{\frac{\beta p}{\beta+\kappa}} \right] \to 0. 
\end{equation*}
and also that the strong convergence of the intial data holds by Hypothesis \ref{Hyp:data}.  The estimate above also leads to the uniform bounds
\begin{equation}
\sup_{\delta >0 }\E \left [\sup_{t \in [0,T]} |M_{t}^{\delta}(\phi)|^{p} \right ] < \infty.
\end{equation}
This information is enough in order to use our usual procedure and verify that $\{ M_{t}(\phi) \}_{t \geq 0}$ is a continuous $\{\mathcal{F}^{t} \}_{t=0}^{T}$ martingale.  Additionally, we may check the convergence in $L^{p}_{w^{*}}(\Omega ; L^{\infty}[0,T])$, so the estimate \ref{Eq:Pf:Id:Martingale_Moment_Bounds} follows.
\end{proof}

\subsection{Strong Convergence of the Density}
%!TEX root = main.tex
We now want to work towards establishing the weak continuity of the effective viscous pressure.  At the $\epsilon$ layer, we chose a test function and begin by applying the Ito formula to find the evolution of $\rd \ud \cdot \varphi_{\delta}$.  Since we have not identified our martingale as a stochastic integral, this is slightly less straightforward.  Instead, we check only that our desired identity holds in $\p$ expectation.  Let us now define the following two random test functions:
\begin{align*}
	&\varphi_{\delta}(t,x,\omega) = \eta \mathcal{A} \left [\eta T_{k}(\rd) \right ] \\
	&\varphi(t,x,\omega) = \eta \mathcal{A}\left [\eta \overline{T_{k}(\rho)} \right ]
\end{align*}
By the Di Perna Lions commutator lemmas, we may verify the following identity:
\begin{equation}
\partial_{t} \varphi_{\delta} = -\eta \mathcal{A}\circ \Div(\eta T_{k}(\rd) \ud) +\eta \mathcal{A}\left [\nabla \eta \cdot  T_{k}(\rd) \ud  \right ] -\eta \mathcal{A} [\eta (\rd T_{k}'(\rd)-T_{k}(\rd) )\Div \ud] \\ 
\end{equation}
Sending $\delta \to 0$ and using the Skorohod step we find
\begin{equation}
\label{eq: phitd} \partial_{t} \varphi = - \eta \mathcal{A}\circ \Div(\eta \overline{T_{k}(\rho)} u)+\eta \mathcal{A}\left [\nabla \eta \cdot  \overline{T_{k}(\rho)} \ud  \right ] -\eta \mathcal{A} [\eta \overline{(\rho T_{k}'(\rho)-T_{k}(\rho) )\Div u }] 
 \end{equation}
We now establish the following averaged Ito product rule.
\begin{Lem}(Two Averaged Ito Product Rules)\label{Lem:Pf:Strong:Averaged_Ito_Product_Rule}
%Define the quantities $g_{\delta},g$ via the relations
%\begin{align*}
%& g_{\delta} = [\rd \ud \tensor \ud - 2\mu \nabla \ud +(-\lambda \Div \ud +\rd^{\gamma} + \delta \rd^{\beta})I ] : \nabla \varphi_{\delta} \\
%& g =  [\rho u \tensor u - 2\mu \nabla u +(-\lambda \Div u +\overline{\rho^{\gamma}} + \delta \overline{\rho^{\beta}})I ] : \nabla \varphi  
%\end{align*}
Define $\phi_{\delta}$ and $\phi$ as above, then the following two averaged Ito product rules hold $\p$ a.s. for all times $t \in [0,T]$ 
\begin{align*}
& \E^{\p} \int_{D} \rd \ud(t) \cdot \varphi_{\delta}(t)dx = \int_{D}m_{0,\delta} \cdot \varphi_{\delta}(0)dx +\E^{\p}  \int_{0}^{t}\int_{D}[\rd \ud \tensor \ud - 2\mu \nabla \ud ]: \nabla \varphi_{\delta}]dxds \\
&+\E^{\p} \int_{0}^{t}\int_{D} [ (-\lambda \Div \ud +\rd^{\gamma}+\delta \rd^{\beta})I ] : \nabla \varphi_{\delta}]dxds  
+\E^{\p} \int_{0}^{t}\int_{D} [ \rd \ud \cdot \partial_{t} \varphi_{\delta}] dxds. \\
& \E^{\p} \left [\int_{D} \rho u(t) \cdot \varphi(t)dx \right ] = \int_{D}m_{0} \cdot \varphi (0)dx + \E^{\p} \int_{0}^{t}\int_{D}[\rho u \tensor u - 2\mu \nabla u ] : \nabla \varphi dxds \\
&+\E^{\p} \int_{0}^{t}\int_{D} [(-\lambda \Div u +\overline{\rho^{\gamma}})I ] : \nabla \varphi  +\rho u \cdot \partial_{t} \varphi dxds. \\
\end{align*}
\end{Lem}
\begin{proof} \quad The first identity can be proved in the same way as at the $\epsilon$ layer.  To proceed to the second, denote by $\xi_{\kappa}$ the standard mollifier(localized at scale $\kappa$) and $\xi_{\kappa,x}$ the standard mollifier centered at the point $x$.  Extending $(\rho,u, \overline{T_{k}(\rho)}, \overline{\rho^{\gamma}}, \overline{\rho^{\beta}})$ by zero outside of $D$, we may define the quantities $g_{\kappa},\varphi_{\kappa}$ via
\begin{align*}
& g_{\kappa}(x,s) =  \int_{\R^{d}}[\rho u \tensor u(s) - 2\mu \nabla u(s) +(-\lambda \Div u(s) +\overline{\rho^{\gamma}}(s))I ] : \nabla \xi_{\kappa}(x-y)dy \\  
& \varphi_{\kappa}(x,t) = \left (\varphi(t) * \xi_{\kappa} \right )(x) \quad (\rho u)_{\kappa}(x,t) =\left (\rho u(t) * \xi_{\kappa} \right )(x) 
\end{align*}
By definition of the process $\{ M_{t}(\xi_{\kappa,x})\}_{t \geq 0}$ we obtain for all $\kappa >0$ and $x \in D$, the following equality holds $\p$ a.s.
$$
(\rho u)_{\kappa}(x,t)=m_{0,\kappa}(x) + \int_{0}^{t}g_{\kappa}(x,s) + M_{t}(\xi_{\kappa,x})
$$
By Lemma \ref{Lem:Pf:Id:Momentum_Martingale} the process $\{ M_{t}(\xi_{\kappa,x})\}_{t=0}^{T}$ is a martingale satisfying enough bounds to give a meaning to the stochastic integral below. Applying the classical Ito product rule for continuous one dimensional martingales, we obtain for each $x \in D$, the following equality holds $\p$ a.s.
\begin{align*}
&(\rho u)_{\kappa}(x,t) \cdot \varphi_{\kappa}(x,t)=m_{0,\kappa}(x) \cdot \varphi_{\kappa}(0) + \int_{0}^{t} \big [\varphi_{\kappa}(x,s) \cdot g_{\kappa}(x,s)+ \partial_{t}\varphi_{\kappa}\cdot (\rho u)_{\kappa}(x,s) \big ]ds \\ 
&+ \int_{0}^{t} \varphi_{\kappa}(x,s)dM_{s}(\xi_{\kappa,x}).
\end{align*}
Note the estimate
\begin{equation*}
\E \left [\int_{0}^{T}\phi_{\kappa}(x,s)^{2}d M(\xi_{\kappa,x})_{s} \right ] \leq \E \left [\sup_{t \in [0,T]} |\phi_{k}(x,s)|^{4} \right ]^{1/2} \E \left [  \langle M(\xi_{\kappa,x}) \rangle_{T}^{2}\right ]^{1/2}.
\end{equation*}
By the definition of quadratic variation(or the Doob Meyer Decomposition for continuous martingales) and the uniform bounds on the fourth moments of $M_{t}(\xi_{\kappa,x})$, the second moment of the quadratic variation is controlled.  The other term is estimated using the expression for $\partial_{t}\phi_{k,x}$ implied by the equation above.  Hence, the stochastic integral above is a martingale(rather than just a local martingale), and hence has mean zero.  Taking expectation and integrating over $D$ yields
\begin{align*}
&\E^{\p} \left [ \int_{D}(\rho u)_{\kappa}(x,t) \cdot \varphi_{\kappa}(x,t)dx \right ]=\int_{D}m_{0,\kappa}(x) \cdot \varphi_{\kappa}(0)dx \\
&+ \E^{\p} \left [\int_{0}^{t} \int_{D} \left [ \varphi_{\kappa}(x,s) \cdot g_{\kappa}(x,s)+ \partial_{t}\varphi_{\kappa}\cdot (\rho u)_{\kappa}(x,s) \right ]dx  \right ].
\end{align*}
Letting $\kappa \to 0$ and appealing to standard properties of mollifiers, we obtain the result.
\end{proof}
We may now use our averaged Ito product rule, together with the Skorohod step in order to obtain another weak continuity result. 
\begin{Lem} \label{Lem:Pf:Strong:Weak_Continuity_Of_The_Effective_Viscous_Pressure} Let $K \subset \subset D$ be arbitrary, then the following averaged version of the weak continuity of the effective viscous pressure holds:
\begin{equation} \label{Eq:Pf:Strong:Weak_Continuity_Of_The_Effective_Viscous_Pressure}
\begin{split}
&\lim_{\delta \to 0} \E^{\p}[\int_{0}^{T}\int_{K}\big (\rd^{\gamma} + \delta \rd^{\beta}-(2\mu + \lambda) \Div \ud \big)T_{k}(\rd) dxdt] \\
&=\E^{\p}[\int_{0}^{T}\int_{K}\left (\rgb -(2\mu + \lambda)\Div u \right)\overline{T_{k}(\rho)}dxdt] 
\end{split}
\end{equation}
\end{Lem}
\begin{proof} \quad In view of Lemma \ref{Lem:Pf:Strong:Averaged_Ito_Product_Rule}, we have again two fundamental identities the drive the result:
\begin{align*}
&\E^{\p} \left [ \int_{0}^{T}\int_{D} \eta^{2} [(2\mu + \lambda)\Div \ud - \rd^{\gamma}-\delta \rd^{\beta}] T_{k}(\rd) dxds \right ] \\
&=I^{0,\delta}+ I^{C,\delta}_{1}+I^{C,\delta}_{2}+I^{C,\delta}_{3}+I^{R,\delta}+ I^{P,\delta} \\
&\E^{\p} \left [ \int_{0}^{T}\int_{D} \eta^{2} [(2\mu + \lambda)\Div u - \overline{\rho^{\gamma}}-\delta \overline{\rho^{\beta}}] \overline{T_{k}(\rho)} dxds \right ]=I^{0}+  I^{C}_{1}+I^{C}_{2}+I^{C}_{3}+I^{R}+I^{P}
\end{align*}
Our labeling convention follows the same logic as the decomposition in Lemma \ref{Lem:Del:Strong:Weak_Continuity_Of_The_Effective_Viscous_Flux}.  There is only one term of a new character, $I^{R,\delta}$ arising from the renormalization of the continuity equation by $T_{k}(\rho)$.
\begin{align}
&I^{0,\delta} = \int_{D} \eta m_{0,\delta} \cdot \mathcal{A}[\eta \rho_{0,\delta}]dx \\
&I^{C,\delta}_{1} = \E^{\p} \left [\int_{0}^{T}\int_{D} [\rd \ud \tensor \ud -2 \mu \nabla \ud +(-\lambda \Div \ud +\rd^{\gamma})I] : \nabla \eta \tensor \mathcal{A}\left [\eta T_{k}(\rd) \right ]dxds \right ] \\
&I^{C,\delta}_{2} = \E^{\p} \left [\int_{0}^{T}\int_{D} \rd \ud \cdot \mathcal{A}[\nabla \eta \cdot T_{k}(\rd) \ud] dxds \right ] \\
&I^{C,\delta}_{3} = \E^{\pdh} \left [\int_{0}^{T}\int_{D} \ud \cdot [\mathcal{A}(\eta \rd)\nabla \eta - \nabla \eta \eta \rd] \right ] \\
&I^{R,\delta} = \E^{\p} \left [\int_{0}^{T}\int_{D} \eta \rd \ud \cdot \mathcal{A}\left [\eta(\rd T_{k}'(\rd)-T_{k}(\rd))\Div \ud \right ]dxds \right ] \\
& I^{P,\delta} = \E^{\p} \left [\int_{0}^{T}\int_{D} \eta \big [ \rd \ud \tensor \ud : \nabla \mathcal{A}[\eta T_{k}(\rd)] - \rd \ud \cdot \mathcal{A}[\Div(\eta T_{k}(\rd) \ud)] 
\big ]dxds \right ] \\
\end{align}
\begin{align}
&I^{0}=\int_{D} \eta m_{0} \cdot \mathcal{A}[\eta \rho_{0}]dx  \\
&I^{C}_{1} = \E^{\p} \left [\int_{0}^{T}\int_{D} [\rho u \tensor u+2 \mu \nabla u +(\lambda \Div u -\rho^{\gamma})I] : \nabla \eta \tensor \mathcal{A}\left [\eta \overline{T_{k}(\rho)} \right ]dxds \right ] \\
&I^{C}_{2} = \E^{\p} \left [\int_{0}^{T}\int_{D} \rho u \cdot \mathcal{A}[\nabla \eta \cdot \overline{T_{k}(\rho)} u] dxds \right ] \\
&I^{C}_{3} = \E^{\pdh} \left [\int_{0}^{T}\int_{D} u \cdot [\mathcal{A}(\eta \rho)\nabla \eta - \nabla \eta \eta \rho] \right ] \\
&I^{R} = \E^{\p} \left [\int_{0}^{T}\int_{D} \eta \rho u \cdot \mathcal{A}\left [\eta \overline{(\rho T_{k}'(\rho)-T_{k}(\rho))}\Div u \right ]dxds \right ] \\
& I^{P} = \E^{\p} \left [\int_{0}^{T}\int_{D} \eta \left [ \rho u \tensor u : \nabla \mathcal{A}[\eta \overline{T_{k}(\rho)}] - \rho u \cdot \mathcal{A} [\Div(\eta \overline{T_{k}(\rho)} u)] \right ]dxds \right ] 
\end{align}
Note that $m_{0,\delta} \to m_{0}$ strongly in $L^{\frac{2\gamma}{\gamma+1}}_{x}$, hence it suffices to note that $W^{1,\gamma}_{x} \hookrightarrow L^{(\frac{2\gamma}{\gamma+1})'}$ in order to conclude that $I^{0,\delta} \to I^{0}$.  The remainder of the analysis is devoted to working term by term and showing the convergence of each integral, similarly to the analysis at the previous layer.  This is accomplished by combining the information in Proposition \ref{Prop:Pf:Comp:Skorohod} with appendix Lemmas \ref{Lem:Appendix:Bounded_Operator_Upgrade} and \ref{Lem:Appendix:Div_Curl_Commutator}.
\end{proof}
Let us proceed by applying weak continuity result to establish the strong convergence of the density.
\begin{Lem} \label{Lem:Pf:Strong:Strong_Convergence_Of_Density}
The sequence of densities $\{\rd\}_{\delta > 0}$ converges strongly to $\rho$ in the sense that
\begin{align}
\lim_{\delta \to 0} |\rd-\rho|_{L^{p}\left ( \Omega ; \LLs{2-} \right )}=0.
\end{align}
\end{Lem}
\begin{proof}
Our strategy is the same as in the the previous layer.  The hypothesis \eqref{Hyp:gamma} ensures that $\gamma$ is large enough that the integrability gains imply $\rho \in L^{p}(\Omega ; \LLs{p})$ for some $p >2$.  Hence, weak solutions to the continuity equation are still renormalized solutions.  That is, it remains valid to apply the commutator lemmas and the Hardy inequality and obtain for each smooth $\psi$ with $\psi(T)=0$ the following equality holds $\p$ a.s.
\begin{align*}
	&\int_{0}^{T}\int_{D}\psi_{t}[L_{k}(\rho) - L_{k}(\rd)] dxds = \int_{0}^{T}\int_{D}\psi \left [ T_{k}(\rho) \Div u -T_{k}(\rd) \Div \rd\right ] dxds. \\
\end{align*}
In fact, the equality above should also contain the initial data, but since they are converging strongly in view of Hypothesis \ref{Hyp:data} we simply take the data to be zero.  Decompose the RHS into a portion over $K \subset \subset D$ and a portion near the boundary
\begin{align*}
&\int_{0}^{T}\int_{D}\psi_{t}[L_{k}(\rho) - L_{k}(\rd)] dxds = \int_{0}^{T}\int_{K}\psi \left [ T_{k}(\rho) \Div u -T_{k}(\rd) \Div \rd\right ] dxds + R_{K}^{\delta}(\psi)\\
& = \int_{0}^{T}\int_{K} \psi \left [\overline{T_{k}(\rho)}( \Div u - \overline{\rho^{\gamma}}) - T_{k}(\rd)( \Div \ud - \rd^{\gamma}) \right ]dxds \\ 
&+ \int_{0}^{T}\int_{K} \psi \left [\overline{T_{k}(\rho)} \, \overline{\rho^{\gamma}}-  T_{k}(\rd) \rd^{\gamma} \right ] dxds \\
\end{align*}
Taking expectation with respect to $\p$ on both sides and using the weak continuity result Lemma \ref{Lem:Pf:Strong:Weak_Continuity_Of_The_Effective_Viscous_Pressure}, we obtain
\begin{align*}
&\E \left [\int_{0}^{T}\int_{D}\psi_{t}[L_{k}(\rho) - L_{k}(\rd)] dxds\right ] \leq \liminf_{\delta \to 0 } \E \left [\int_{0}^{T}\int_{K} \psi \left [\overline{T_{k}(\rho)} \, \overline{\rho^{\gamma}}-  T_{k}(\rd) \rd^{\gamma} \right ] dxds \right ]  \\
&+ \E \left [ \int_{0}^{T}\int_{K} \psi [T_{k}(\rho)-\overline{T_{k}(\rho)}] \Div u \, dxds \right ] + 
\int_{0}^{T}\int_{D \setminus K}\psi \left [ T_{k}(\rho) \Div u -\overline{T_{k}(\rho) \Div u} \right ] dxds \\
& \leq |\Div u|_{L^{2}(\Omega;\LL{2}{2})}|T_{k}(\rho)-\overline{T_{k}(\rho)}|_{L^{2}(\Omega ;L^{2}_{t,x})}+\int_{0}^{T}\int_{D \setminus K}\psi \left [ T_{k}(\rho) \Div u -\overline{T_{k}(\rho) \Div u} \right ] dxds  \, .  
\end{align*}
Note that we used the montononicity of the pressure.  The uniform bounds on the density in $L^{p}(\Omega ; \LLs{p})$ imply that the first term tends to zero as $k \to \infty$(by interpolation and a straightforward lower semicontinuity argument).  Hence, using a sequence of $\psi$ approximating indicators(as in the $\epsilon$ layer) and making the remainder term arbitrarily small, we find by taking $k \to \infty$ that for all $t \in [0,T]$
\begin{align*}
\E \big [\int_{D}\overline{\rho\text{log}(\rho)}(t)-\rho \text{log}(\rho)(t) dx \big ] =0.
\end{align*}
Arguing as we did in the previous layer and using the uniform bounds, we conclude.
\end{proof}
%%%%%%%%%%% DELTA TO ZERO ENERGY INEQUALITY
%%%%%%%%%%%
%%%%%%%%%%%
%%%%%%%%%%%
%%%%%%%%%%%
%%%%%%%%%%%
The following lemma is due to Lions \cite{lions1998mathematical}, and will be used to deduce the strong convergence of the momentum.
\begin{Lem} \label{Lem:Pf:Strong:Lions_Lemma}
Let $\{ \rho_{n}\}_{n=1}^{\infty}$ and $\{ u_{n}\}_{n=1}^{\infty}$ be deterministic sequences satisfying the following
\begin{enumerate}
\item For all $p < \gamma$,  $\rho_{n} \to \rho \quad \text{in} \quad  \CL{p} \cap \CLw{\gamma}$
\item $u_{n} \to u \quad \text{in} \quad \LW{2}{1}{2}$
\item $\rho_{n}u_{n} \to \rho u \quad \text{in} \quad \CLw{\frac{2\gamma}{\gamma+1}}$
\end{enumerate}
then $\rho_{n}u_{n} \to \rho u$ in $\LL{2}{1}$.
\end{Lem}
\begin{proof}
See \cite{lions1998mathematical} page 34-35 for most of the ideas of the proof.  The rest are left to the reader.
\end{proof}

\subsection{Conclusion of the Proof of Theorem \ref{Thm:I:Main_Result}}
%!TEX root = main.tex
\begin{Lem}
The pair $(\rho,u)$ satisfies the momentum equation \eqref{Eq:Pre:Momentum_Equation} of Definition \ref{Def:Pre:Weak_Solutions}.
\end{Lem}
\begin{proof} We proceed with our usual strategy based on Lemma \ref{Lem:Appendix:Three_Martingales_Lemma}.  Namely, for each $\phi \in C^{\infty}_{c}(D)$ we introduce the continuous, $\{ \F^{t}\}_{t \geq 0}$ adapted process $\{ M_{t}\}_{t=0}^{T}$ defined by:
\begin{equation}
M_{t}(\phi)= \int_{D} \rho u(t) \cdot \phi \, dx  - \int_{D} m_{0} \cdot \phi \, dx - \int_{0}^{t}\int_{D} [\rho u \tensor u-2\mu \nabla u-\lambda \Div u I ] : \nabla \phi + \rho^{\gamma}\Div \phi \, dxds\\\, .
\end{equation}
In view of Lemma \ref{Lem:Pf:Id:Momentum_Martingale} and Lemma \ref{Lem:Pf:Strong:Strong_Convergence_Of_Density} , $\{ M_{t}(\phi) \}_{t=0}^{T}$ is a martingale.  To complete the usual strategy, note that
\begin{align*}
 &\E^{\p} \left [ \int_{0}^{T} \big (\int_{D}| \rd \sigma_{k}(\rd * \eta_{\delta}, (\rd \ud) *\eta_{\delta},x)-\rho \sigma_{k}(\rho, \rho u,x)|dx\big )^{2}ds \right ] \\
 &\leqs  \E^{\p} \left [ \int_{0}^{T} \big (\int_{D} |\rd-\rho| |\sigma_{k}(\rd * \eta_{\delta}, (\rd \ud) *\eta_{\delta},x) dx |^{2}ds \right ] \\ 
 &+\E^{\p} \left [ \int_{0}^{T} \big (\int_{D} \rho |\sigma_{k}(\rd * \eta_{\delta}, (\rd \ud) *\eta_{\delta},x)-\sigma_{k}(\rho, \rho u,x)| dx\big )^{2}ds \right ]  \\
 & \leqs |\sigma_{k}|_{L^{\frac{2\gamma}{\gamma-1}}_{x}(L^{\infty}_{\rho,m})}^{2} \E^{\p} \left [|\rd - \rho|_{\LL{2}{\frac{2\gamma}{\gamma+1}}}^{2} \right ] \\
 &+ \E^{\p} \left [ |\rho|_{\LL{\infty}{\gamma}}^{4}\right ]^{\frac{1}{2}} \E^{\p} \left [ \big ( \int_{0}^{T} \big (\int_{D} |\sigma_{k}(\rd * \eta_{\delta}, (\rd \ud) *\eta_{\delta},x)-\sigma_{k}(\rho, \rho u,x)|^{\frac{\gamma}{\gamma-1}} dx\big )^{2(1-\frac{1}{\gamma})}ds \big )^{2} \right ]^{1/2}.
  \end{align*}
The first term tends to zero by Lemma \ref{Lem:Pf:Strong:Strong_Convergence_Of_Density}.  To control the second term, note that
\begin{align*}
&\E^{\p} \left [ \big ( \int_{0}^{T} \big (\int_{D} |\sigma_{k}(\rd * \eta_{\delta}, (\rd \ud) *\eta_{\delta},x)-\sigma_{k}(\rho, \rho u,x)|^{\frac{\gamma}{\gamma-1}} dx\big )^{2(1-\frac{1}{\gamma})}ds \big )^{2} \right ] \\
&\leqs \E^{\p} \left [ \big ( \int_{0}^{T} \big (\int_{D} |\sigma_{k}(\rd * \eta_{\delta}, (\rd \ud) *\eta_{\delta},x)-\sigma_{k}(\rho, (\rd \ud) *\eta_{\delta},x)|^{\frac{\gamma}{\gamma-1}} dx\big )^{2(1-\frac{1}{\gamma})}ds \big )^{2} \right ] \\
&+ \E^{\p} \left [ \big ( \int_{0}^{T} \big (\int_{D} |\sigma_{k}(\rho, (\rd \ud) *\eta_{\delta},x)-\sigma_{k}(\rho, \rho u,x)|^{\frac{\gamma}{\gamma-1}} dx\big )^{2(1-\frac{1}{\gamma})}ds \big )^{2} \right ] \\
&\leqs |\sigma_{k}|_{L^{\infty}_{x,\rho,m}}^{4} \E^{\p} \left [|(\rd \ud)*\eta_{\delta}-\rho u|_{\LL{2}{1}}^{2} + |\rd * \eta_{\delta}-\rho|_{\LL{2}{1}}^{2} \right ].   
\end{align*} 
In the last line, we interpolated and used to the lipschitz Hypothesis \ref{Hyp:lip}.  Let us now explain why these terms go to zero.  Using Proposition \ref{Lem:Pf:Strong:Strong_Convergence_Of_Density} one can extract a subsequence such that $\rho_{n} \to \rho$ in $L^{2-}_{t,x}$ for all $q < 2$, $\p$ almost surely.  Using the method on page 23 of \cite{lions1998mathematical}, one can use the renormalizations of the transport equation to upgrade this convergence to $\p$ almost surely in $\CL{\gamma-}$.  Finally, appealing to Lemma \ref{Lem:Pf:Strong:Lions_Lemma} and Vitali we see that the quantity about tends to zero.  \\
This convergence suffices to complete our usual method and identify
$$
M_{t}(\phi) = \sum_{k=1}^{\infty}\int_{0}^{t}\int_{D}\rho \sigma_{k}(\rho,\rho u,x)dxd\beta_{k}(s).
$$ 
We may also verify the energy bounds \eqref{Eq:Pre:Energy_Estimate} using the lower-semicontinuity of the norm in the usual way.  The desired continuity and measurability conditions imposed in Part \ref{Def:Pre:Weak_Solutions:Item:Measurability} of Definition \ref{Def:Pre:Weak_Solutions}  follows from the construction of the filtration $\{ \mathcal{F}_{t}\}_{t=0}^{T}$. 
 \end{proof}
This completes the proof of our main result.

\setcounter{section}{0}
\setcounter{subsection}{0}
\renewcommand\thesection{\Alph{section}}
\addtocounter{section}{1}
\section*{Appendix A} \label{Appendix:A}
%!TEX root = main.tex
%%%%%%%%%%%%% Background on stochastic processes and the Skorohod Theorem
%%%%%%%%%%%%%
%%%%%%%%%%%%%
\subsection{Random Variables on Topological Spaces and the Skorohod Theorem}
Let $(\Omega, \mathcal{F},\p)$ be a probability space and $(E,\tau,\mathcal{B}_{\tau})$ be a topological space endowed with its Borel sigma algebra.  A mapping $X : \Omega \to (E,\tau)$ is called an ``$E$ valued random variable" provided it is a measurable mapping between these spaces.  Every $E$ valued valued random variable induces a probability measure on $(E,\tau,\mathcal{B}_{\tau})$ by pushforward, which we denote $\p \circ X^{-1}$.  A  sequence of probability measures $\left \{\p_{n} \right \}_{n=1}^{\infty}$ on $\mathcal{B}_{\tau}$ is said to be ``tight" provided that for each $\xi >0$ there exists a $\tau$ compact set $K_{\xi}$ such that $\p_{n}(K_{\xi}) \geq 1-\xi$ for all $n \geq 1$.  

A collection $ \left \{ X_{t} \right \}_{t=0}^{T}$ is an $E$ valued stochastic process provided that for each $t$, $X_{t}$ is an $E$ valued random variable.  An $E$ valued stochastic process is progressively measurable with respect to the filtration $\left \{ \F^{t} \right \}_{t=0}^{T}$ provided that for each $t\leq T$,    
$$X \mid_{[0,t]} \, : \Omega \times [0,t]  \to (E, \tau , \mathcal{B}_{\tau})$$
is measurable with respect to the product sigma algebra $\mathcal{F}_{t} \times \mathcal{B}([0,t])$.
\begin{Def}
A topological space $(E,\tau)$ is called a Jakubowski space provided there exists a countable sequence $\left \{G_{k} \right \}_{k=1}^{\infty} : E \to \R$ of $\tau$ continuous functionals which separate points in $E$.
\end{Def} 
Our main interest in such spaces is the following fundamental result: 
\begin{Thm} \label{Thm:Appendix:Jakubowksi_Skorohod}
Let $(E,\tau)$ be a Jakubowski space.  Suppose that $\{ \hat{X}_{k} \}_{k \geq 1}$ is a sequence of $E$ valued random variables on a sequence of probability spaces $\{(\hat{\Omega}_{k}, \hat{\F}_{k},\hat{\p}_{k})\}_{k \geq 1}$ such that $\left \{\hat{\p}_{k} \circ \hat{X}_{k}^{-1} \right \}_{k=1}^{\infty}$ is tight. 

Then there exists a new probability space $(\Omega , \F, \p)$ endowed with an $E$ valued random variable X and a sequence of ``recovery'' maps $ \{ \widehat{T}_{k} \}_{k=1}^{\infty}$ 
$$ \widehat{T}_{k} : (\Omega , \F, \p) \to (\hat{\Omega}_{k}, \hat{\F}_{k},\hat{\p}_{k})$$
with the following two properties:
\begin{enumerate}
\item For each $k$, the measure $\hat{\p}_{k}$ may be recovered from $\p$ by pushing forward $\widehat{T}_{k}$.
\item The new sequence  $\{X_{k}\}_{k \geq 1} := \{\widehat{X}_{k} \circ \widehat{T}_{k}\}_{k \geq 1}$ converges $\p$ a.s. to $X$(with respect to the topology $\tau$).
\end{enumerate}
\end{Thm}
\begin{proof}
This result is a combination of the versions of the Skorohod theorem proved in \cite{MR1453342} and \cite{MR1385671}.  It can be proved by modifying the proof in \cite{MR1453342} in a very slight way.  Namely, at the point in the proof where the classical Skorohod theorem for metric spaces is applied, one may apply the Skorohod theorem in \cite{MR1385671} to obtain the recovery maps. 
\end{proof}
\begin{Rem} \label{Rem:Appendix:Examples_of_Jakubowksi_Spaces}
It is straightforward to check that the following are examples of Jakubowski spaces: Polish spaces, dual spaces of separble Banach spaces $B_{w*}$ endowed with the weak star topology, and $C_{t}(B_{w})$ for reflexive Banach spaces $B$.

Also, the class of Jakubowski spaces is closed under countable products.  In particular, given a Jakubowski space $(E,\tau)$, $E^{\infty}$ is also a Jakubowski space with respect to the $\tau$ product topology.  Similarly, for finite products of different Jakubowski spaces.   
\end{Rem}
%%%%%%%%%%%%%%% Series of One Dimensional Stochastic Integrals and their Identification
%%%%%%%%%%%%%%%
%%%%%%%%%%%%%%%
%%%%%%%%%%%%%%%
\subsection{Series of One Dimensional Stochastic Integrals} 
By a stochastic basis, we mean a probability space $(\Omega,\mathcal{F},\p)$ together with a a filtration $\{ \mathcal{F}^{t}\}_{t=0}^{T}$ and a collection $\{ \beta_{k}\}_{k=1}^{\infty}$ of $\{ \mathcal{F}^{t}\}_{t=0}^{T}$ one dimensional Brownian motions.
\begin{Prop} \label{Prop:Appendix:Defining_A_Series_Of_Stochastic_Integrals}
Let $(\Omega, \mathcal{F},\p, \{ \mathcal{F}^{t}\}_{t=0}^{T}, \{ \beta_{k}\}_{k=1}^{\infty} )$ be a stochastic basis endowed with a collection of $\{\mathcal{F}^{t}\}_{t=0}^{T}$ progressively measurable proceeses $\{ f_{k}\}_{k=1}^{\infty} : \Omega \times [0,T] \to \R$, such that 
\begin{equation*}
\sum_{k=1}^{\infty} \int_{0}^{T}\E^{\p} \left [f_{k}^{2}(s) \right ]ds < \infty.  
\end{equation*}
Then we may construct an $\{ \mathcal{F}^{t}\}_{t=0}^{T}$ martingale $\{ M_{t}\}_{t=0}^{T}$ with $\p$ a.s. continuous paths of the form 
\begin{equation*}
M_{t} = \sum_{k=1}^{\infty} \int_{0}^{t}f_{k}(s)d\beta_{k}(s).
\end{equation*}
The series above converges uniformly in time in probability and the quadratic variation process is given by 
$$\langle M \rangle_{2}^{t}(\omega)=\sum_{k=1}^{\infty} \int_{0}^{t}f_{k}^{2}(s,\omega)ds$$. 
\end{Prop}
\begin{proof}
This is a consequence of the Kolomogorov Three Series theorem and the construction of the one dimensional stochastic integral.  See Krylov \cite{MR2800911} for more discussion.
\end{proof}
The next lemma, taken from \cite{MR3098063},  provides a procedure for identifying a continuous, adapted process as a series of one dimensional stochastic integrals.  
\begin{Lem} \label{Lem:Appendix:Three_Martingales_Lemma}
Let $(\Omega, \mathcal{F},\p, \{ \mathcal{F}_{t}\}_{t=0}^{T}, \{ \beta_{k}\}_{k=1}^{\infty} )$ be a stochastic basis endowed with a continuous $\{ \mathcal{F}_{t}\}_{t=0}^{T}$ martingale $\left \{M_{t} \right \}_{t=0}^{T}$.  Moreover, suppose the following are also $\{ \mathcal{F}_{t}\}_{t=0}^{T}$ martingales 
\begin{enumerate}
\item $(\omega,t) \to  M_{t}^{2}(\omega)-\sum_{k=1}^{\infty} \int_{0}^{t}f_{k}^{2}(\omega,s)ds$ 
\item $(\omega,t) \to M_{t}(\omega)\beta^{k}_{t}(\omega) - \int_{0}^{t}f_{k}(s)ds$ \quad (for each $k \geq 1$)
\end{enumerate}
then the process $\{ M_{t}\}_{t=0}^{T}$ may be identified as 
$$
M_{t}=\sum_{k=1}^{\infty} \int_{0}^{t}f_{k}(s)d\beta_{k}(s).
$$
\end{Lem}
%%%%%%%%%%%%%%
%%%%%%%%%%%%%%
%%%%%%%%%%%%%% \ref{
%%%%%%%%%%%%%%
\subsection{The Space of Weakly Continuous Functions in $L^{m}_{x}$}
This section contains a useful tightness criterion for probability measures over the topological space $\CLw{p}$.
 \begin{Lem} \label{Lem:Appendix:Compact_Sets_of_Weakly_Continuous_Functions}
Let $\{f_{n}\}_{n=1}^{\infty}$ be a sequence in $\LL{\infty}{m}$ with $1<m<\infty$.  Suppose that the following two criterion are met:
\begin{enumerate}
\item $\sup_{n} |f_{n}|_{\LL{\infty}{m}} < \infty$ 
\item For all  $\phi \in C_{c}^{\infty}(D)$ in a dense subset of $L^{m}_{x}$, the following sequence in $C_{t}$ is equicontinuous  
$$ \left \{ t \to \int_{D}f_{n}(x,t)\phi(x)dx \right \}_{n=1}^{\infty}. $$
\end{enumerate}
Then there exists an $f \in \CLw{m}$ and a subsequence such that
 $$ f_{n_{k}} \to f \quad \text{in} \quad \CLw{m}.$$
\end{Lem}
\begin{proof}
See the Appendix in Lions \cite{MR1422251}.
\end{proof}
A straightforward application of the lemma above yields
\begin{Cor} \label{Cor:Appendix:Time_Derivative_In_A_Negative_Sobolev_Space} 
For any positive $M$, integer $k$ and $q > 1$, the following sets are compact in $\CLw{p}$
$$\{ f \in \CLw{p} \mid |f|_{\LL{\infty}{m}} + |\partial_{t}f|_{\LW{q}{-k}{p}}\leq M \}$$
\end{Cor}
The tightness criterion can now be stated as follows
\begin{Lem} \label{Lem:Appendix:TightCrit}
Let $\left \{f_{n} \right \}_{n=1}^{\infty}$ be a collection of $\CLw{m}$ valued random variables, each defined on a probability space $(\Omega_{n},\F_{n},\p_{n})$ such that
\begin{enumerate}
\item $$\sup_{n} \E^{\pn} |f_{n}|_{\LL{\infty}{m}} < \infty$$
\item For any $\phi \in C^{\infty}_{c}(D)$, there exists an integer $k$, $\gamma > 0$, and $p > \frac{1}{\gamma}$ such that 
$$ \sup_{n} \E^{\pn} \left [ \big | \langle f_{n}(t)-f_{n}(s),\phi \rangle \big |^{p} \right ] \leq |\phi|_{C^{k}_{x}}^{p}|t-s|^{\gamma p} $$
\end{enumerate}
for all $0 \leq s,t \leq T$.  Then the sequence of induced measures $\left  \{ \pn \circ f_{n}^{-1} \right \}_{n=1}^{\infty}$ are tight on $\CLw{m}$.
\end{Lem}
\begin{proof}
\quad Enumerate a countable collection $ \left \{ \phi_{j} \right \}_{j=1}^{\infty}$ in $C_{c}^{\infty}(D)$ which is dense in $L^{m'}_{x}$. The second hypothesis of the lemma implies that for all $s < \gamma$ and $j \geq 1$
\begin{equation*}
\sup_{n \geq 1}\E | \langle f_{n},\phi_{j} \rangle |_{W^{s,p}_{t}}^{p} \leqs |\phi_{j}|_{C^{k}}^{p}.
\end{equation*}  
Choosing $\alpha >0$ sufficiently small to apply the Sobolev embedding theorem gives
\begin{equation*}
\sup_{n \geq 1}\E | \langle f_{n},\phi_{j} \rangle |_{C^{\alpha}_{t}}^{p} \leqs |\phi_{j}|_{C^{k}}^{p}.
\end{equation*}  
Given a small number $\xi >0$, define a set $K_{\xi}$ by 
$$ K_{\xi}=\{ f \in \LL{\infty}{m} \mid |f|_{\LL{\infty}{m}} \leq M \xi^{-1} \} \cap \bigcap_{j=1}^{\infty} \left \{f \in \LL{\infty}{m} \mid  | \langle f,\phi_{j} \rangle |_{C^{\alpha}_{t}} \leq (2^{j}\xi^{-1})^{\frac{1}{p}}|\phi|_{C^{k}} \right \}.$$
Lemma \ref{Lem:Appendix:Compact_Sets_of_Weakly_Continuous_Functions} implies this set is sequentially compact in $\CLw{m}$.  Sequential compactness and compactness are equivalent in $\CLw{m}$.  Applying Chebyshev, then using the uniform bounds, we find that
$$
\sup_{n \geq 1} \p \circ f_{n}^{-1}(K_{\xi}^{c}) \leqs \xi.
$$
\end{proof}

\addtocounter{section}{1}
\setcounter{Thm}{0}
\setcounter{equation}{0}
\section*{Appendix B} \label{Appendix:B}
%!TEX root = main.tex
\subsection{Weak Convergence Upgrades}
The following lemma is simple, but fundamental enough to state explicitly.
\begin{Lem} \label{Lem:Appendix:Bounded_Operator_Upgrade}
Let $E,F$ be Banach spaces and use $E_{w}$ to denote the space endowed with its weak topology. Let $T: E \to F$ be a bounded linear operator.  Suppose the sequence $\left \{ f_{n}\right \}_{n=1}^{\infty}$ converges to $f$ in $C_{t}(E_{w})$.  Then  $\left \{ Tf_{n}\right \}_{n=1}^{\infty}$ converges to $Tf$ in $L^{q}_{t}(F_{w})$ for all $1 \leq q < \infty$. 
\end{Lem}

\begin{Lem} \label{Lem:Appendix:Compact_Operator_Upgrade}
Make the same assumptions as in Lemma \ref{Lem:Appendix:Bounded_Operator_Upgrade} above.  In addition, assume $T$ is compact.  Then $\left \{ Tf_{n}\right \}_{n=1}^{\infty}$ converges to $Tf$ in $L^{q}_{t}(F)$ for all $1 \leq q < \infty$.   
\end{Lem}

\begin{proof}
Since bounded operators preserve weak convergence, for each $t \in [0,T]$ we have $Tf_{n}(t) \to Tf(t)$ weakly in $F$.  If $T$ is compact then the convergence is strong.  Combining the uniform bounds in $C_{t}(E_{w})$ with the Vitali convergence theorem gives both claims. 
\end{proof}

\begin{Lem} \label{Lem:Appendix:Convex_Function_Upgrade}
Let $(\Omega,\F,\mu)$ be a finite measure space. Let $\{ f_{n}\}_{n=1}^{\infty}$ in $L^{p} \left (\Omega,\F,\mu \right  )$ converge weakly to  $f \in L^{p} \left (\Omega,\F,\mu \right)$.  Moreover, assume there is a convex function $\varphi : \R \to \R$ such that $\{ \varphi(f_{n}) \}_{n=1}^{\infty}$ converges weakly to $\varphi(f)$ in $L^{1} \left (\Omega,\F,\mu \right)$.  Denote by $\mathcal{C}$ the subset of $\R$ where $\varphi$ is strictly convex.  

Then there is a full $\mu$ measure set $\Omega'$ such that $\{ f_{n}(\omega)\}_{n=1}^{\infty}$ converges pointwise to $f(\omega)$ for all $\omega \in \mathcal{C} \cap \Omega'$.
\end{Lem}

\subsection{Some Tools from the Deterministic Compressible Theory}

The following result is a consequence of the Div Curl lemma.  Denote $\mathcal{R}_{ij} = \partial_{ij}\Delta^{-1}$, understood to be well defined on compactly supported distributions.
\begin{Lem} \label{Lem:Appendix:Div_Curl_Commutator}
Let $D$ be a smooth, bounded domain and $\eta$ a smooth cutoff.  Let $B$ be a Banach space.  Suppose $\{ f_{n}\}_{n=1}^{\infty}$ converges to $f$ in $\CLw{p}$ and $\{ g_{n}\}_{n=1}^{\infty}$ converges to $g$ in $\CLw{q}$.  Also, assume the embedding $L^{r}_{x} \hookrightarrow B$ is compact, where $\frac{1}{p}+\frac{1}{q}=\frac{1}{r}<1$. 

Then the following convergence holds: 
$$
\eta \left ( f_{n}\mathcal{R}_{ij}[ \eta g_{n}]- g_{n}\mathcal{R}_{i,j}[\eta f_{n}] \right )\to \eta \left (f \mathcal{R}_{ij}[\eta g]-g \mathcal{R}_{i,j}[\eta f] \right ) 
$$
weakly in $L^{m}_{t}(B)$ for all $1 \leq m < \infty$.
\end{Lem}
\begin{proof}
Combine the corresponding result in Feiresil \cite{feireisl2001compactness} with Lemma \ref{Lem:Appendix:Compact_Operator_Upgrade}(using the compact injection operator from $L^{r}_{x}$ to $B$).
\end{proof}

Next we collect some properties of the Bogovoski operator $\mathcal{B}$.  Recall that classically, $\mathcal{B}[g]$ is defined to be the solution to the problem
\begin{equation}
\begin{cases}
\Div v = g \quad &\text{in} \quad D\\
v= 0 \quad &\text{on} \quad \partial D\\
\end{cases}
\end{equation}
for $g \in L^{p}(D)$ such that $\int_{D}gdx=0$.  For our purposes, it is useful to have an extension of this operator to the negative Sobolev spaces.  We recall a result from \cite{MR2240056}.  Define the space $L^{p}_{0}(D) = \{f \in L^{p}(D) \mid  \int_{D}fdx=0\}$. For $s \in [0,1]$ define $\widehat{W}^{s,p}(D) : = W^{s,p}(D) \cap L^{p}_{0}(D)$.  Furthermore, let $\widehat{W}^{-s,p}(D) : =[ \widehat{W}^{s,p'}(D)]'$.
\begin{Thm} \label{Thm:Appendix:Bogovoski}
Let $1<p<\infty$ and $s \in [-1,1]$.  Then there exists a bounded linear operator $\mathcal{B}:\widehat{W}^{s,p}(D) \to [W^{s+1,p}(D)]^{d}$ such that $\Div \mathcal{B}[g] = g$ for all $g \in \widehat{W}^{s,p}(D)$.
\end{Thm}

\subsection{Lemmas on Parabolic Equations}
The following lemma provides an energy equality for sufficiently integrable weak solutions to the parabolic Neumann problem driven by a rough velocity field.
\begin{Lem} \label{Lem:Appendix:Energy_Identity_Parabolic_Neumann}
Let $u \in \LW{2}{1}{2}$ and $p > d$.  Suppose $\rho \in \LL{\infty}{p}$ is a distributional solution of
\begin{equation} 
\begin{cases} 
\partial_{t}\rho - \epsilon \Delta \rho + \Div(\rho u) = 0 & \quad \text{in} \quad D \times [0,T] \\
\frac{\partial \rho}{\partial n} = 0 & \quad \text{in} \quad \partial D \times [0,T] \\
\rho(x,0) = \rho_{0}(x) & \quad \text{in} \quad D \\
\end{cases}
\end{equation}
Then for all times $0 \leq t \leq T$, the energy identity holds:
$$
 \frac{1}{2}\int_{D} \rho^{2}(t)dx + \epsilon \int_{0}^{t}\int_{D}|\nabla \rho|^{2}dxdt =\frac{1}{2}\int_{D} \rho_{0}^{2}dx - \frac{1}{2}\int_{0}^{t}\int_{D}\Div u \, \rho^{2}dxdt .
 $$
\end{Lem}
We also need a variant of the usual $L^{q}_{t}(W^{2,p}_{x})$ estimates for the parabolic Neumann problem.  A similar result is proved in the appendix to \cite{berthelin2013stochastic}.  The lemma below states that by giving up the optimal exponent, one can retain a form of the usual estimate even if the solution.  Recalling the splitting from section \ref{Section:Tau}.  Choose $\tau$ such that $\frac{T}{\tau}$ is an even integer and let $t_{k}=k\tau$ for $k=0,...,\frac{T}{\tau}$.  Define $\htd$ by \eqref{Eq:Tau:Setup:htd}.    
\begin{Lem} \label{Lem:Appendix:Mixed_Parabolic_Hyperbolic_Estimate}
Let $1<p<\infty$ and $F \in \LLs{p}$. Suppose that $\rho$ solves 
\begin{equation}
\begin{cases}
\partial_{t} \rho - \epsilon \Delta \rho = F & \quad \text{in} \quad D \times \bigcup_{k=0}^{\frac{T}{2\tau}-1} (t_{2k},t_{2k+1}] \\
\partial_{t}\rho = 0 &\quad \text{in} \quad D \times \bigcup_{k=0}^{\frac{T}{2\tau}-1} (t_{2k+1},t_{2k+2}] \\
\frac{\partial \rho}{\partial n} = 0 &\quad \text{on} \quad \partial D \times [0,T] \\ 
\rho(0,x)=\rho_{0} &\quad \text{in} \quad D
\end{cases}
\end{equation}
Then for all $q<p$
\begin{equation}
|\partial_{t}\rho|_{\LLs{q}} + |\rho|_{\LW{q}{2}{q}} \leq C(\epsilon,d) \left ( |\rho_{0}|_{W^{2,p}_{x}}+ |F \htd|_{\LLs{p}} \right ).
\end{equation} 
\end{Lem}

$\mathbf{Acknowledgements}$
The author is grateful to his thesis advisor Konstantina Trivisa for suggesting this problem and patiently reviewing several drafts of this work.  He recieved support through her National Science Foundation grant, award DMS 1211519.  He is also extremely grateful to Sam-Punshon Smith for helping check many of the calculations in the paper and providing a variety of useful comments.  Moroever, he thanks his co-advisor Sandra Cerrai for sparking his interest in stochastic PDE's and providing valuable feedback on an earlier draft.  The author thanks Zdzislaw Brze\'zniak for a helpful discussion on the Skorohod theorem and bringing the paper \cite{MR1385671} to his attention. Finally, thanks to Pierre-Emmanuel Jabin for valuable advice on improving the exposition and introduction.

\bibliographystyle{plain}
\bibliography{bibliography}
\end{document}